\documentclass[12pt]{amsart}
\numberwithin{equation}{section}

\usepackage{amsmath}
\usepackage{amscd,amsthm,amssymb,amsfonts}
\usepackage{mathrsfs}
\usepackage{dsfont}
\usepackage{stmaryrd}
\usepackage{euscript}
\usepackage{expdlist}
\usepackage{enumerate}

\input xy

\newtheorem{thm}{Theorem}[section]

\newtheorem*{thm*}{Theorem}

\newtheorem{lm}[thm]{Lemma}
\newtheorem{cor}[thm]{Corollary}
\newtheorem*{cor*}{Corollary}
\newtheorem{prop}[thm]{Proposition}
\newtheorem*{conj*}{Conjecture}

\theoremstyle{Remark}

\newtheorem*{remark}{Remark}

\theoremstyle{definition}
\newtheorem*{defn*}{Definition}

\newtheorem{I_Remark*}{Remark}
\newtheorem{defn}[thm]{Definition}

\newcommand{\nc}{\newcommand}

\newcommand{\beq}{\begin{equation}}
\newcommand{\eeq}{\end{equation}}
\newcommand{\bpmx}{\begin{pmatrix}}
\newcommand{\epmx}{\end{pmatrix}}
\newcommand{\bbmx}{\begin{bmatrix}}
\newcommand{\ebmx}{\end{bmatrix}}

\def\parref#1{\ref{#1}}
\def\thmref#1{Theorem~\parref{#1}}

\def\propref#1{Proposition~\parref{#1}}
\def\corref#1{Corollary~\parref{#1}}

\def\lmref#1{Lemma~\parref{#1}}

\def\makeop#1{\expandafter\def\csname#1\endcsname
  {\mathop{\rm #1}\nolimits}\ignorespaces}

\makeop{Hom}   \makeop{End}   \makeop{Aut}   
\makeop{Pic} \makeop{Gal}       \makeop{Div} \makeop{Lie}
\makeop{PGL}   \makeop{Corr} \makeop{PSL} \makeop{sgn} \makeop{Spf}
 \makeop{Tr} \makeop{Nr} \makeop{Fr} \makeop{disc}
\makeop{Proj} \makeop{supp} \makeop{ker}   \makeop{Im} \makeop{dom}
\makeop{coker} \makeop{Stab} \makeop{SO} \makeop{SL} \makeop{SL}
\makeop{Cl}    \makeop{cond} \makeop{Br} \makeop{inv} \makeop{rank}
\makeop{id}    \makeop{Fil} \makeop{Frac}  \makeop{GL} \makeop{SU}
\makeop{Trd}   \makeop{Sp} \makeop{Tr}    \makeop{Trd} \makeop{Res}
\makeop{ind} \makeop{depth} \makeop{Tr} \makeop{st} \makeop{Ad}
\makeop{Int} \makeop{tr}    \makeop{Sym} \makeop{can} \makeop{SO}
\makeop{torsion} \makeop{GSp} \makeop{Tor}\makeop{Ker} \makeop{rec}
\makeop{Ind} \makeop{Coker}
 \makeop{vol} \makeop{Ext} \makeop{gr} \makeop{ad}
 \makeop{Gr}\makeop{corank} \makeop{Ann}
\makeop{Hol} 
\makeop{Fitt} \makeop{Mp} \makeop{CAP}
\def\Jac{\mathrm{Jac}}

\def\makebb#1{\expandafter\def
  \csname bb#1\endcsname{{\mathbb{#1}}}\ignorespaces}
\def\makebf#1{\expandafter\def\csname bf#1\endcsname{{\bf
      #1}}\ignorespaces}
\def\makegr#1{\expandafter\def
  \csname gr#1\endcsname{{\mathfrak{#1}}}\ignorespaces}
\def\makescr#1{\expandafter\def
  \csname scr#1\endcsname{{\EuScript{#1}}}\ignorespaces}
\def\makecal#1{\expandafter\def\csname cal#1\endcsname{{\mathcal
      #1}}\ignorespaces}

\def\doLetters#1{#1A #1B #1C #1D #1E #1F #1G #1H #1I #1J #1K #1L #1M
                 #1N #1O #1P #1Q #1R #1S #1T #1U #1V #1W #1X #1Y #1Z}
\def\doletters#1{#1a #1b #1c #1d #1e #1f #1g #1h #1i #1j #1k #1l #1m
                 #1n #1o #1p #1q #1r #1s #1t #1u #1v #1w #1x #1y #1z}
\doLetters\makebb   \doLetters\makecal  \doLetters\makebf
\doLetters\makescr
\doletters\makebf   \doLetters\makegr   \doletters\makegr

    \def\setminus{\smallsetminus}

\normalsize

\makeop{Ram} \makeop{Rep} \makeop{mass}

\makeop{Bl}

\def\diag#1{\mathrm{diag}(#1)}
\def\el{\ell}

\def\cI{\mathcal I}

\def\sR{\mathscr R}

\def\ot{\otimes}

\def\hookto{\hookrightarrow}
\def\longto{\longrightarrow}
  \nc{\opp}{\mathrm{opp}} \nc{\ul}{\underline}

\def\XYmatrix{\xymatrix@M=8pt} 
\def\ncmd{\newcommand}
\ncmd{\xysubset}[1][r]{\ar@<-2.5pt>@{^(-}[#1]\ar@<2.5pt>@{_(-}[#1]}
\ncmd{\XYmatrixc}[1]{\vcenter{\XYmatrix{#1}}}
\ncmd{\xyto}[1][r]{\ar@{->}[#1]}
\ncmd{\xyinj}[1][r]{\ar@{^(->}[#1]}
\ncmd{\xysurj}[1][r]{\ar@{->>}[#1]}
\ncmd{\xyline}[1][r]{\ar@{-}[#1]}
\ncmd{\xydotsto}[1][r]{\ar@{.>}[#1]}
\ncmd{\xydots}[1][r]{\ar@{.}[#1]}
\ncmd{\xyleadsto}[1][r]{\ar@{~>}[#1]}
\ncmd{\xyeq}[1][r]{\ar@{=}[#1]} \ncmd{\xyequal}[1][r]{\ar@{=}[#1]}
\ncmd{\xyequals}[1][r]{\ar@{=}[#1]}
\ncmd{\xymapsto}[1][r]{l\ar@{|->}[#1]}\ncmd{\xyimplies}[1][r]{\ar@{=>}[#1]}
\ncmd{\xyiso}{\ar[r]_-{\sim}}
\def\injxy{\ar@{^(->}}

\newcommand{\pMX}[4]{\begin{pmatrix}
{#1}& {#2}\\
{#3}&{#4}\end{pmatrix} }

\newcommand{\seesaw}[4]{{#1}\ar@{-}[rd]\ar@{-}[d]&{#2}\ar@{-}[d]\\
{#3}\ar@{-}[ru]&{#4}}

\def\x{{\times}}

\def\e{\varepsilon} 

\newcommand\stt[1]{\left\{#1\right\}}
\def\ep{\epsilon}

\def\half{\frac{1}{2}}

\def\ka{\kappa}

\renewcommand\pmod[1]{\,(\mbox{mod }{#1})}

\usepackage{color}
\usepackage{hyperref}
\usepackage{MnSymbol}
\usepackage{mathdots}
\usepackage{tikz-cd}

\def\SO{{\rm SO}}
\def\M{{\rm Mat}}
\def\Rep{{\rm Rep}}
\def\Irr{{\rm Irr}}
\def\Ind{{\rm Ind}}
\def\ind{{\rm ind}}
\def\Jac{{\rm Jac}}
\def\St{{\rm St}}

\def\ka{\kappa}
\def\d{\delta}
\def\D{\Delta}
\def\bt{\boxtimes}
\def\ul{\underline}
\def\a{\alpha}
\def\ep{\epsilon}
\def\la{\lambda}
\def\({\left(}
\def\){\right)}

\title{Local newforms for generic representations of $p$-adic $\SO_{2n+1}$: Reduction}
\author{Yao Cheng \\  with an appendix by Chi-Heng Lo}
\date{\today}
\address{Department of Applied Mathematics and Data Science, Tamkang University, No. 151, Yingzhuan Road, Tamsui District, 
New Taipei City 251, Taiwan (R.O.C),  Lui-Hsien Memorial Science Hall.}
\email{briancheng@o365.tku.edu.tw}

\address{Department of Mathematics, National University of Singapore, 119076, Singapore}
\email{ch\_lo@nus.edu.sg}

\begin{document}
\maketitle
\begin{abstract}
We prove that if the space of newforms is non-zero for every irreducible generic supercuspidal representation of $\SO_{2n+1}$ then it is also 
non-zero for all irreducible generic representations of $\SO_{2n+1}$. 
\end{abstract}

\section{Introduction}
This paper is a sequel to \cite{YCheng2025}, in which we proved the uniqueness part of the newform conjecture for the $p$-adic group 
$\SO_{2n+1}$ proposed by Gross (\cite{Gross2015}) and verified the expected arithmetic properties of newforms, conditional on their existence. 
The aim of the present paper is to reduce the existence part of the conjecture to the case of generic supercuspidal representations. 

In the literature, the newform conjecture for generic supercuspidal representations was treated in the PhD 
thesis of Tsai (\cite{Tsai2013}); however, it seems that Tsai's proofs contain some issues (see \S\ref{SS:Tsai lemma}). 
In the present paper, we do not attempt to resolve these issues. Therefore, it may be safer to say that the existence part of the conjecture 
remains conditional.

\subsection{Main results}
To state the conjecture and our results, let $F$ be a finite extension of $\bbQ_p$, and let $\SO_{2n+1}(F)$ denote the split odd special 
orthogonal group defined over $F$. In \cite{Gross2015} (see also \cite{Tsai2013} and \cite{YCheng2025}), Gross defined a 
family $\stt{K_{n,m}}_{m\ge 0}$ of open compact subgroups of $\SO_{2n+1}(F)$, generalizing the families introduced by Casselman 
(\cite{Casselman1973}) for $n=1$ and Roberts--Schmidt (\cite{RobertsSchmidt2007}) for $n=2$ to arbitrary $n$. He also defined a family 
$\stt{J_{n,m}}_{m\ge 0}$ of open compact subgroups of $\SO_{2n+1}(F)$ such that $K_{n,0}=J_{n,0}$ and, for each $m\ge 1$, 
$K_{n,m}$ is a normal subgroup of $J_{n,m}$ of index $2$. 

Let $U_n(F)$ be the maximal unipotent subgroup of $\SO_{2n+1}(F)$, which consists of upper triangular matrices.
Fix an unramified additive character $\psi$ of $F$, and define a non-degenerate character $\psi_{U_n}$ of $U_n(F)$ by 
\[
\psi_{U_n}(u)=\psi\left(u_{1,2}+\cdots+u_{n-1,n}+2^{-1}u_{n,n+1}\right)
\]
for $u=(u_{i,j})\in U_n(F)$.

Let $\pi$ be an irreducible generic representation of $\SO_{2n+1}(F)$, i.e., ${\rm Hom}_{U_n(F)}\left(\pi,\psi_{U_n}\right)\ne 0$, 
with associated $L$-parameter $\phi_\pi$ (\cite{JiangSoudry2003}, \cite{JiangSoudry2004} and \cite{Arthur2013}). 
The $\epsilon$-factor $\epsilon(s,\phi_\pi,\psi)$ attached to $\phi_\pi$, $\psi$ and the standard 
representation of the $L$-group of $\SO_{2n+1}(F)$ (\cite{Tate1979}), can be expressed as
\[
\epsilon(s,\phi_\pi,\psi)
= 
\e_\pi q^{-c_\pi\left(s-\frac{1}{2}\right)},
\]
for some $\e_\pi=\pm 1$ and an integer $c_\pi\ge 0$, where $q$ is the cardinality of the residue field of $F$.

\begin{conj*}[Gross]
Let $\pi$ be an irreducible generic representation of $\SO_{2n+1}(F)$. Then 
\begin{itemize}
\item[(1)] the subspaces satisfy $\pi^{K_{n,m}}=0$ for $0\le m<c_\pi$ and $\dim_\bbC\pi^{K_{n,c_\pi}}=1$;
\item[(2)] the action of $J_{n,c_\pi}/K_{n,c_\pi}$ on $\pi^{K_{n,c_\pi}}$ is given by the scalar $\e_\pi$;
\item[(3)] the natural pairing of the one-dimensional spaces 
\[
{\rm Hom}_{K_{n,c_\pi}}\left(1,\pi\right)\,\x\,{\rm Hom}_{U_n(F)}\left(\pi,\psi_{U_n}\right)\longto\bbC,
\]
is non-degenerate.
\end{itemize}
\end{conj*}

This conjecture holds when $n=1$ and $2$, due to the works of Casselman (\cite{Casselman1973}) and Roberts--Schmidt 
(\cite{RobertsSchmidt2007}), respectively. It also holds when $\pi$ is unramified. In the prequel to the present paper, we obtain 
the following results:

\begin{thm}[\cite{YCheng2025}]\label{T:mainA}
Let $\pi$ be an irreducible generic representation of $\SO_{2n+1}(F)$. Then 
\begin{itemize}
\item[(1)] the subspaces satisfy $\pi^{K_{n,m}}=0$ for $0\le m<c_\pi$ and $\dim_\bbC\pi^{K_{n,c_\pi}}\le 1$;
\item[(2)] if $\pi^{K_{n,c_\pi}}\ne 0$, then the action of $J_{n,c_\pi}/K_{n,c_\pi}$ on $\pi^{K_{n,c_\pi}}$ is given by the scalar $\e_\pi$;
\item[(3)] if $\pi$ is tempered and $\pi^{K_{n,c_\pi}}\ne 0$, then the natural pairing of the one-dimensional spaces 
\[
{\rm Hom}_{K_{n,c_\pi}}\left(1,\pi\right)\,\x\,{\rm Hom}_{U_n(F)}\left(\pi,\psi_{U_n}\right)\longto\bbC,
\]
is non-degenerate.
\end{itemize}
\end{thm}

In view of \thmref{T:mainA}, it remains to verify the existence part of the conjecture, namely, to prove that $\pi^{K_{n,c_\pi}}\ne 0$. 
Following the prequel paper, $\pi^{K_{n,c_\pi}}$ is referred to as the space of newforms of $\pi$. The aim of the present paper is to further reduce 
the existence part of the conjecture to the case of generic supercuspidal representations. More precisely, we obtain the following result in this 
paper. 

\begin{thm}\label{T:mainB}
If the space of newforms is non-zero for every irreducible generic supercuspidal representation, then it is also non-zero for all irreducible 
generic representations. 
\end{thm}

By \thmref{T:mainA} and \thmref{T:mainB}, to settle the newform conjecture, it suffices to prove:
\begin{itemize} 
\item that the space of newforms is non-zero for every irreducible generic supercuspidal representation;
\item that  the natural pairing appearing in the third assertion of the newform conjecture is non-degenerate for non-tempered 
         generic representations. 
\end{itemize}
As mentioned, Tsai's proofs for supercuspidal representations seem to contain some issues; however, we believe that her strategy is essentially 
correct. We hope to settle this conjecture for supercuspidal representations, either by following Tsai's approach or by providing a new proof.
For the assertion concerning the non-degeneracy of the natural pairing, on the other hand, we currently have no approach to proving it. 

\subsection{Outline of the proof}
The proof of \thmref{T:mainB} relies on the following deductions:
\begin{itemize}
\item[(i)] from generic representations to square-integrable representations;
\item[(ii)] from square-integrable representations to seed representations;
\item[(iii)] from seed representations to supercuspidal representations. 
\end{itemize}
More concretely, let $\pi$ be an irreducible generic representation of $\SO_{2n+1}(F)$. Assume that $\pi$ is non-supercuspidal. Then 
\begin{equation}\label{E:intro}
\pi\subset\Pi=\tau_1\x\cdot\cdot\cdot\x\tau_\el\rtimes\pi_0,
\end{equation}
for some irreducible essentially square-integrable representations $\tau_j$ of $\GL_{r_j}(F)$ for $j=1,\ldots,\el$, and some irreducible 
generic representation $\pi_0$ of $\SO_{2n_0+1}(F)$. Here we follow Tadi\'c's notation for induced representations of $\SO_{2n+1}(F)$. 

Our strategy is to prove that if the space of newforms of $\pi_0$ is non-zero, then that of $\pi$ is also non-zero. 
To be more precise, since $\pi^{K_{n,m}}\ne 0$ for some $m$ (\cite[Proposition 7.6]{Tsai2016}), there exists a smallest integer $a_\Pi\ge 0$ 
such that $\Pi^{K_{n,a_\Pi}}\ne 0$. Such an integer can be explicitly determined (see \lmref{L:useful}). Suppose that the space of newforms 
of $\pi_0$ is non-zero and that 
\begin{equation}\label{E:intro 1}
a_\Pi=c_\pi.
\end{equation}
Then the same lemma and \thmref{T:mainA} (1) imply that 
\[
\dim_\bbC\Pi^{K_{n,a_\Pi}}=\dim_\bbC\pi_0^{K_{n_0,c_{\pi_0}}}=1. 
\]
Thus, under the aforementioned assumptions, we must show that 
$
\Pi^{K_{n,c_\pi}}=\pi^{K_{n,c_\pi}}.
$ 
In fact, we prove that 
\begin{equation}\label{E:intro 2}
\Pi^{K_{n,m}}=\pi^{K_{n,m}}
\end{equation}
for every $m\ge 0$. We point out that the assumption $a_\Pi=c_\pi$ may not always hold. Now, we sketch the reduction process. 

Suppose that $\pi$ is non-tempered. Then $\pi_0$ is tempered and $\pi=\Pi$ by the standard module conjecture (\cite{Muic2001}). 
In particular, \eqref{E:intro 2} holds for every $m\ge 0$. Since $a_\Pi=c_\pi$ in this case, the proof is reduced to the case of  
tempered representations. 

Suppose that $\pi$ is tempered but not square-integrable. Then $\pi_0$ is square-integrable, and we again have $a_\Pi=c_\pi$ in this case. 
Since $\pi\ne \Pi$ in general, to prove that \eqref{E:intro 2} holds for every $m\ge 0$, we follow the argument of \cite{AOY2024} to show that 
a tempered representation of $\SO_{2n+1}(F)$ has a non-zero vector fixed by $K_{n,m}$ for some $m\ge 0$ if and only if such a representation
is generic. Now, since $\Pi$ is completely reducible and $\pi$ is the unique generic summand of $\Pi$, we obtain the desired result. 
Hence, the proof reduces to the case of square-integrable representations. 

Suppose that $\pi$ is square-integrable but not supercuspidal. In this case, the argument is considerably more involved than in the previous 
cases. In fact, most of this paper is devoted to reducing the proof of the existence part of the conjecture from square-integrable representations 
to supercuspidal representations. Difficulties arise because (i) $a_\Pi\ne c_\pi$ in general, and (ii) $\Pi$ is not completely reducible. 
To remedy this, we introduce the ``seed representations" (see Definition \ref{D:seed}).  These are (generic) square-integrable representations 
whose $L$-factor has the form
\[
L(s+\ka, \chi)^e\cdot L(s+\ka',\chi')^{e'}
\]
for some $\ka, \ka'\in\frac{1}{2}+\bbZ_{\ge 0}$ and $e, e'\in\stt{0,1}$. Here $\chi$ and $\chi'$ are two unramified quadratic  (possibly trivial) 
characters of $F^\x$. 

Assume that $\pi$ is not a seed representation. Then we attach to $\pi$ a seed representation $\pi_0$ such that \eqref{E:intro} holds and, more
importantly, \eqref{E:intro 1} also holds. To establish \eqref{E:intro 2} for every $m\ge 0$, we investigate the Jacquet modules of $\Pi$ and 
$\pi$ using the results of Atobe in \cite{Atobe2020}.  This reduces the proof to the case of seed representations. 

Suppose that $\pi$ is a seed representation. If the $L$-factor of $\pi$ is equal to $1$, then we realize $\pi$ as a subrepresentation of $\Pi$ as 
in \eqref{E:intro}, where $\pi_0$ is now a supercuspidal representation. In this case, we have $a_\Pi=c_\pi$, and the above Jacquet module 
argument continues to work. 

Assume that the $L$-factor of $\pi$ is $L(s+\ka,\chi)$. In this case, we realize $\pi$ as in \eqref{E:intro 1} with 
$\el=1$, where $\pi_0$ is a square-integrable representation whose $L$-factor is equal to $1$, and $\Pi$ is a standard module. 
Note that, under the assumption that the space of newforms is non-zero for every supercuspidal representation, the space of newforms of 
$\pi_0$ is also non-zero, as we have already proved. It follows that
\[
\dim_\bbC\Pi^{K_{n,a_\Pi}}=1. 
\]
Denote by $f$ a basis vector of this one-dimensional subspace. On the other hand, we have $c_\pi=a_\Pi+1$ in this case. Thus, we have
to investigate the subspace $\Pi^{K_{n,a_\Pi+1}}$ and show that 
\[\Pi^{K_{n,a_\Pi+1}}\cap \pi\ne \stt{0}.
\] 
To this end, we consider the level-raising operators (see \S\ref{SS:level-raising operator})
\[
\theta,\, \theta':\Pi^{K_{n,a_\Pi}}\longto\Pi^{K_{n,a_\Pi+1}}
\]
following Roberts--Schmidt (\cite{RobertsSchmidt2007}). We explicitly determine the elements 
\[
\theta(f),\, \theta'(f)\in\Pi^{K_{n,a_\Pi+1}}. 
\]
It is interesting to note that the Hecke operators on the space of newforms of generic representations (see \cite{KondoYasuda2012}) appear
in this determination. 

A crucial result, proved by Atobe in a special case and by Lo (see \S\ref{S:appendix}) in general, is that $\Pi$ 
has length $2$. It follows that $\pi$ can be realized as the kernel of a standard intertwining map
\[
M:\Pi\longto\Pi'=\tau_1^\vee\rtimes\pi_0. 
\]
We have $a_{\Pi'}=a_\Pi$. To prove this case, we show that there exist $c, c'\in\bbC$ such that 
\[
M\(c\,\theta(f)+c'\theta'(f)\)=0
\quad\text{and}\quad
c\,\theta(f)+c'\theta'(f)\ne 0. 
\]
The key point is that
\[
0\ne M(f)\in\Pi'^{K_{n,a_{\Pi}}}
\]
and
\[
M(\theta(f))=\theta(M(f)),\,\,M(\theta'(f))=\theta'(M(f))
\]
so that our determination of $\theta(f), \theta'(f)$ also applies to $\theta(M(f)), \theta'(M(f))$. 

Finally, assume that the $L$-factor of $\pi$ is $L(s+\ka,\chi)L(s+\ka',\chi')$. We realize $\pi$ as in \eqref{E:intro} with $\el=1$, where 
$\pi_0$ is a square-integrable representation whose $L$-factor is $L(s+\ka,\chi)$, and $\Pi$ is a standard module. Since $\Pi$ again 
has length $2$, $c_\pi$ is equal to $a_\Pi+1$, and the space of newforms of $\pi_0$ is non-zero (provided that the space of 
newforms of every supercuspidal representation is non-zero), we can apply the same argument to conclude that the space of newforms of 
$\pi$ is also non-zero. This completes the proof sketch of \thmref{T:mainB}. 

\subsection{Structure of the paper}
In \S\ref{S:1st red}, we reduce the proof of \thmref{T:mainB} from generic representations to generic square-integrable representations. 
In \S\ref{S:P for 2nd red}, we collect some results that will be used in the second reduction. These include Tad\'ic's result on Jacquet modules 
in \S\ref{SS:Tadic}, $P_{n+1}$-theory for $\GL_{n+1}$ in \S\ref{SS:P-theory}, representations of $\GL_r(F)$ in \S\ref{SS:rep of GL}, and 
the construction of generic square-integrable representations of $\SO_{2n+1}(F)$ in \S\ref{SS:construction}. In \S\ref{S:2nd red}, we reduce the 
proof of \thmref{T:mainB} from generic square-integrable representations to seed representations. Seed representations are introduced in 
Definition \ref{D:seed}. The key to the proof is \propref{P:key Jac}, whose proof occupies a substantial part of this section. 
In \S\ref{S:P for 3rd red}, we prepare for the final reduction. This includes the level-raising operators in \S\ref{SS:level-raising operator} and 
the Hecke operators on $\GL_r$ in \S\ref{SS:Hecke GL}. In \S\ref{S:3rd red}, we reduce the proof of \thmref{T:mainB} from seed representations 
to generic supercuspidal representations. The keys to the proof are \lmref{L:length 2} and the explicit computation of the level-raising operators in 
\S\ref{SS:compute operator I} and \S\ref{SS:compute operator II}. Finally, \lmref{L:length 2} is proved in \S\ref{S:appendix}. 

\subsection{Notation and conventions}
\subsubsection{}
For consistency, we follow the notation and conventions in \cite[\S 1.5]{YCheng2025}. Thus, $F$ is a finite extension of $\mathbb{Q}_p$,
$\frak{o}$ is the valuation ring of $F$, $\frak{p}$ is the maximal ideal of $\frak{o}$, $\varpi$ is a uniformizer of $\frak{p}$, and
$\frak{f}=\frak{o}/\frak{p}$ is the residue field of $F$ with $|\frak{o}/\frak{p}|=q$. Let $|\cdot|$ be the absolute value on $F$, 
normalized so that $|\varpi|=q^{-1}$. Fix an additive character $\psi$ of $F$ that is trivial on $\frak{o}$ but non-trivial on $\frak{p}^{-1}$. 
Denote by $\M_{m,n}(F)$ the space of $m\times n$ matrices with entries in $F$. When $m=n$, we simply write
$\M_{n}(F)=\M_{m,n}(F)$. Let $E^n_{i,j}\in\M_n(F)$ denote the matrix with a single non-zero entry $1$ in the $(i,j)$-th position. 
The identity element in $\M_n(F)$ is denoted by $I_n$, while $J_n\in\M_n(F)$ is defined by $J_n=\sum_{i=1}^n E^n_{i,n+1-i}$. 

\subsubsection{}
Let $G$ be an $\ell$-group in the sense of \cite[\S 1]{BZ1976}. Denote by $\Rep(G)$ the category of smooth complex representations of 
$G$ of finite length. The set of equivalence classes of irreducible smooth complex representations of $G$ is denoted by $\Irr(G)$. 
If $\pi\in\Rep(G)$, we denote by $\pi^\vee\in\Rep(G)$ the contragredient of $\pi$ and by $\omega_\pi$ the central character of $\pi$
(if it exists). If $G$ is the trivial group, we denote by $1_G$ the trivial representation of $G$. Let $\sR(G)$ denote the Grothendieck 
group of $\Rep(G)$, which is a free $\bbZ$-module with basis $\Irr(G)$. The natural map
\[
s.s.:\Rep(G)\longto\sR(G);\quad\pi\mapsto s.s.(\pi)
\]
is called the semi-simplification. We interpret the expression:
\[
s.s.(\pi)=\sum_{i\in I}\pi_i
\]
where $\pi_i\in\Irr(G)$, as being taken with multiplicity. The group $\sR(G)$ has a natural partial order $\ge$, defined by
$\pi_1\ge \pi_2$ if $\pi_1-\pi_2$ is a positive linear combination of irreducible representations.
 
Let $M,N$ be closed subgroups of $G$ such that $M$ normalizes $N$, $M\cap N=\stt{e}$ and the subgroup $P=MN$ of $G$ is 
closed. Every $\sigma\in\Rep(M)$ can then be regarded naturally as an object of $\Rep(P)$, and we denote by $\Ind_P^G(\sigma)$ the 
normalized induced representation of $G$. On the other hand, given $\pi\in\Rep(G)$, let 
$\Jac_N(\pi)$ denote the normalized Jacquet module of $\pi$ with respect to $N$ (\cite[\S 1]{BZ1977}). 
For the groups considered in this paper, one has $\Ind_P^G(\sigma)\in\Rep(G)$ and $\Jac_N(\pi)\in\Rep(M)$; thus $\Ind_P^G$ and 
$\Jac_N$ induce exact functors:
\[
\Ind_P^G:\Rep(M)\longto\Rep(G)
\quad\text{and}\quad
\Jac_N:\Rep(G)\longto\Rep(M).
\]

\subsubsection{}\label{SS:2.3}
We follow the setup in \cite[\S 2]{YCheng2025} for $\SO_{2n+1}$. In particular, the parabolic subgroups $P_S(F)=M_SN_S$ of 
$\SO_{2n+1}$ under consideration are indexed by subsets $S$ of the set of standard simple roots $\Delta_n=\stt{\a_1,\ldots,\a_n}$ of 
$\SO_{2n+1}$.  Our convention is that $P_\emptyset=\SO_{2n+1}$; hence $M_{\Delta_n}=T_n$ and $N_{\Delta_n}=U_n$.  
On the other hand, let $r\in\bbN$ and $\mathbf{r}=(r_1,\ldots, r_\el)$ be a partition of $r$. Denote by 
$P_{\mathbf{r}}=M_{\mathbf{r}}N_{\mathbf{r}}$ the upper block triangular parabolic subgroup of $\GL_r$, whose Levi subgroup 
\[
M_{\mathbf{r}}(F)\simeq\GL_{r_1}(F)\x\cdot\cdot\cdot\x\GL_{r_\el}(F). 
\]
When $r_j=1$ for $j=1,\ldots, \el=r$, we also denote $M_{\mathbf{r}}=A_r$ and $N_{\mathbf{r}}=Z_r$. 
Define non-degenerate characters $\psi_{U_n}$ and $\psi_{Z_r}$ on $U_n(F)$ and $Z_r(F)$, respectively, by 
\[
\psi_{U_n}(u)=\psi\left(u_{1,2}+\cdot\cdot\cdot+u_{n-1,n}+2^{-1}u_{n,n+1}\right)
\quad\text{for $u=(u_{i,j})\in U_n(F)$,}
\]
and 
\[
\psi_{Z_r}(z)=\psi\left(z_{1,2}+\cdot\cdot\cdot+z_{r-1,r}\right)
\quad\text{for $z=(z_{i,j})\in Z_r(F)$.}
\]

For the groups $\GL_r(F)$ and $\SO_{2n+1}(F)$, we also use the following notation for induced representations, following Zelevinsky and 
Tadi\'c.  Namely, if $\mathbf{r}=(r_1,\ldots, r_\el)$ is a partition of $r$, and if $\tau_j\in\Rep\left(\GL_{r_j}(F)\right)$ for $j=1,\ldots, \el$, 
then we denote
\[
\tau_1\x\cdot\cdot\cdot\x\tau_\el
=
\Ind_{P_{\mathbf{r}}(F)}^{\GL_r(F)}\left(\tau_1\boxtimes\cdot\cdot\cdot\boxtimes\tau_\el\right). 
\]
While if $S\subset \Delta_n$ is the subset such that $M_S\simeq\GL_{r_1}(F)\x\cdots\x\GL_{r_\el}(F)\x\SO_{2n_0+1}(F)$,  then we denote 
\[
\tau_1\x\cdot\cdot\cdot\x\tau_\el\rtimes\pi_0
=
\Ind_{P_S(F)}^{\SO_{2n+1}(F)}(\tau_1\boxtimes\cdot\cdot\cdot\boxtimes\tau_\el\boxtimes\pi_0),
\]
where $\tau_j\in\Rep\left(\GL_{r_j}(F)\right)$ for $j=1,\ldots, \el$ and  $\pi_0\in\Rep(\SO_{2n_0+1}(F))$. 

Let $m\ge 0$ be an integer. Denote by $K_{n,m}\subset J_{n,m}\subset\SO_{2n+1}(F)$ the open compact subgroups defined in 
\cite[Section 5]{Gross2015} (see also \cite[Section 7]{Tsai2013} and \cite[Section 3]{YCheng2025}). 

\subsubsection{}
Throughout this paper, we will need the local Langlands correspondence for $\SO_{2n+1}(F)$, which was established by Arthur in 
\cite{Arthur2013} in full generality\footnote{Conditional on the twisted weighted fundamental lemma.} (see also \cite{AGIKMS24}). However, 
since we only require the local Langlands correspondence for generic representations of $\SO_{2n+1}(F)$, and, more importantly, the 
explicit construction of representations from the associated $L$-parameter, we will mainly follow the correspondence established by 
Jiang–Soudry in \cite{JiangSoudry2003} and \cite{JiangSoudry2004}. We remark that their correspondence is based on automorphic 
descent together with the work of Mui\'c in \cite{Muic1998}, and is compatible with Arthur's. 

We understand that $1_{\GL_0(F)}=1_{\SO_1(F)}$ is unramified, generic, and supercuspidal, with $L$-parameter $\phi=\emptyset$. 

\section{First reduction}\label{S:1st red}
The aim of this section is to carry out the first reduction; namely, we reduce the proof of \thmref{T:mainB} from generic representations to 
generic square-integrable representations. This reduction is fairly straightforward and follows from the results in \cite{YCheng2025} and the local 
Gross--Prasad conjecture (now a theorem of Waldspurger) in \cite{GrossPrasad1992}. We begin with the following definition.

\begin{defn}
Let $\Pi\in\Rep(\SO_{2n+1}(F))$ with $n\ge 0$. If $\Pi^{K_{n,m}}\ne 0$ for some $m\ge 0$, then we denote by $a_\Pi\ge 0$ the smallest 
non-negative integer such that $\Pi^{K_{n,a_\Pi}}\ne 0$. 
\end{defn} 

\begin{remark}
Let $\Pi\in\Rep(\SO_{2n+1})$ with $n\ge 0$. If $\Pi$ is of Whittaker type, i.e., $\dim_\bbC{\rm Hom}\(\Pi,\psi_{U_n}\)=1$, then the proof of 
\cite[Proposition 7.6]{Tsai2016} (see also \cite[Lemma 4.2]{YCheng2022}) and a result of Soudry (\cite[Proposition 6.1]{Soudry1993}) imply
that there exists $m\ge 0$ such that $\Pi^{K_{n,m}}\ne 0$. Hence, for such a representation, $a_\Pi$ is defined. 
\end{remark}

\begin{lm}\label{L:GP}
Let $\pi$ be an irreducible tempered representation of $\SO_{2n+1}(F)$. Then $\pi^{K_{n,m}}\ne 0$ for some $m\ge 0$ if and only if 
$\pi$ is generic. 
\end{lm}

\begin{proof}
When $\pi$ is supercuspidal, this was proved by Tsai in \cite[Corollary 3.4.2 and Theorem 7.3.1]{Tsai2013}. For generic $\pi$ (not 
necessarily tempered), the assertion that $\pi^{K_{n,m}}\ne 0$ for some $m\ge 0$ was proved in \cite[Proposition 7.6]{Tsai2016}
(see also \cite[Lemma 4.2]{YCheng2022}). 

When $\pi$ is tempered, the converse statement--namely, that $\pi^{K_{n,m}}\ne 0$ for some $m\ge 0$ implies that $\pi$ is generic--was 
proved in \cite{AOY2024} in the course of establishing a newform theory for the unramified group ${\rm U}_{2n+1}(F)$. The key 
ingredients in their proof are 
\begin{itemize}
\item[-] a lemma of Gan--Savin (\cite[Lemma 12.5]{GanSavin2012});
\item[-] the Gan--Gross--Prasad conjecture (\cite{GanGrossPrasad2012}) for $({\rm U}_{2n+1}, {\rm U}_{2n})$, established by 
               Beuzart-Plessis (\cite{BP2014}, \cite{BP2015} and \cite{BP2016});
\item[-] the fact that the intersection of the open compact subgroups of ${\rm U}_{2n+1}(F)$ defined in \cite{AOY2024} with the
               subgroup ${\rm U}_{2n}(F)\subset{\rm U}_{2n+1}(F)$ is a hyperspecial maximal open compact subgroup of ${\rm U}_{2n}(F)$. 
\end{itemize}
A similar argument also appeared in \cite{Atobe2025} in the development of a newform theory for the unramified group ${\rm U}_{2n}(F)$. 
Now, our proof follows word for word the one in \cite{AOY2024}, except that, instead of using the aforementioned results of 
Beuzart-Plessis, we apply the results of Waldspurger (\cite{Waldspurger2010}, \cite{Waldspurger2012a}).
\end{proof}

We also need one useful lemma. To describe it, suppose that $\tau$ is an irreducible generic representation of $\GL_r(F)$. The $\ep$-factor
$\ep(s,\phi_\tau,\psi)$ associated with the $L$-parameter $\phi_\tau$ of $\tau$ and $\psi$ is of the form
\[
\epsilon(s,\phi_\tau,\psi)
= 
\e_\tau q^{-c_\tau\left(s-\frac{1}{2}\right)},
\]
for some $\e_\tau\in\bbC^\x$ and an integer $c_\tau\ge 0$.  

\begin{lm}\label{L:useful}
Let $\tau_i$ be irreducible generic representations of $\GL_{r_i}(F)$ for $i=1,\ldots, \el$,  and let $\pi_0\in\Rep\(\SO_{2n_0+1}(F)\)$ with 
$n_0\ge 0$. Define
\[
\Pi=\tau_1\x\cdot\cdot\cdot\x\tau_\el\rtimes\pi_0\in\Rep\(\SO_{2n+1}(F)\). 
\]
Suppose that $\pi_0^{K_{n_0,m}}\ne 0$ for some $m\ge 0$. Then $\Pi^{K_{n,m}}\ne 0$ for some $m\ge 0$, and 
\[
a_\Pi=a_{\pi_0}+2\sum_{i=1}^\el c_{\tau_i}. 
\]
Moreover, we have 
\[
\dim_\bbC\Pi^{K_{n,a_\Pi}}=\dim_\bbC\pi_0^{K_{n_0,a_{\pi_0}}}. 
\]
\end{lm}

\begin{proof}
This follows from the proofs of \cite[Lemma 6.1 and Corollary 6.2]{YCheng2025}. 
\end{proof}

Now, we can prove the following. 

\begin{prop}\label{P:sq to gen}
If the space of newforms is non-zero for every irreducible generic square-integrable representation, then it is also non-zero for all
irreducible generic representations. 
\end{prop}

\begin{proof}
By \cite[Lemma 6.1]{YCheng2025} and the proof of \cite[Corollary 6.2]{YCheng2025}, if the space of newforms is non-zero for every 
irreducible generic tempered representation, then it is also non-zero for all irreducible generic representations. In particular, the proof reduces 
to the case of tempered representations. 

Now, let $\pi$ be an irreducible generic tempered representation of $\SO_{2n+1}(F)$ that is not square-integrable. 
Let $\phi$ be the associated $L$-parameter. 
From \cite[\S 3 and \S 4]{JiangSoudry2004}, the $L$-parameter has the following decomposition:
\[
\phi_1\oplus\cdot\cdot\cdot\oplus\phi_\el\oplus\phi_0\oplus\phi_\el^\vee\oplus\cdot\cdot\cdot\oplus\phi_1^\vee,
\]
where
\begin{itemize}
\item[--] $\phi_0$ is a discrete $L$-parameter of $\SO_{2n_0+1}(F)$;
\item[--] $\phi_i$ is an irreducible $r_i$-dimensional representation of the Weil--Deligne group $WD_F$ of $F$ for $i=1,\ldots, \el$; and 
\item[--] $\phi^\vee_i$ is the dual of $\phi_i$ for each $i$. 
\end{itemize}

Let $\pi_0$ (resp. $\tau_i$) be the irreducible generic square-integrable representation of $\SO_{2n_0+1}(F)$ (resp. $\GL_{r_i}(F)$) 
associated with $\phi_0$ (resp. $\phi_i$) under the local Langlands correspondence. Then, by $loc.$ $cit.$, $\pi$ is an irreducible 
summand of 
\[
\Pi=\tau_1\x\cdot\cdot\cdot\x\tau_\el\rtimes\pi_0\in\Rep\(\SO_{2n+1}(F)\).
\]
Note that all other irreducible summands of $\Pi$ are also tempered. Since $\pi_0$ is generic, it follows from a result of 
Rodier (\cite{Rodier1973}) that
\[
\dim_\bbC{\rm Hom}_{U_n(F)}\(\Pi,\psi_{U_n}\)=1. 
\]
Consequently, $\pi$ is the only generic summand of $\Pi$; hence 
\[
\pi^{K_{n,m}}=\Pi^{K_{n,m}}
\]
for each $m\ge 0$ by \lmref{L:GP}. To complete the proof, note that the decomposition of $\phi$ yields
\[
c_\pi=c_{\pi_0}+\sum_{i=1}^\el\(c_{\tau_i}+c_{\tau_i^\vee}\)=c_{\pi_0}+2\sum_{i=1}^\el c_{\tau_i}. 
\]
On the other hand, our assumption on $\pi_0$ implies that $a_{\pi_0}=c_{\pi_0}$. Now the proof follows immediately from \lmref{L:useful}. 
\end{proof}

Because of \propref{P:sq to gen}, we shall focus on square-integrable representations of $\SO_{2n+1}(F)$ in the remainder of this paper.

\section{Preliminaries for the second reduction}\label{S:P for 2nd red}

\subsection{A result of Tadi\'c}\label{SS:Tadic}
In this subsection, we introduce a fundamental result of Tadi\'c in \cite{Tadic1995}, which is indispensable to our reduction. 
Following Tadi\'c, we put
\[
\sR=\bigoplus_{r\geq 0}\sR(\GL_r(F))
\quad\text{and}\quad
\sR(S)=\bigoplus_{n\geq 0}\sR(\SO_{2n+1}(F)).
\]
Note that when $n=r=0$, the corresponding Grothendieck group is that of all finite-dimensional complex vector spaces. 

The functors of parabolic induction and Jacquet modules endow $\sR$ with the structure of a $\bbZ_+$-graded Hopf algebra, where 
$\bbZ_+:=\bbZ_{\ge 0}$. More concretely, we can define a multiplication $m$ on $\sR$ as follows. If $\tau_i\in\Irr(\GL_{r_i}(F))$ for $i=1,2$, 
then 
\[
m(\tau_1,\tau_2):=s.s.(\tau_1\x\tau_2)\in\sR(\GL_{r_1+r_2}(F)).
\]
We then extend $m$ $\bbZ$-linearly to all of $\sR$. In this way, $\sR$ becomes a commutative, associative graded ring.
The induced map $\sR\ot\sR\longto\sR$ is again denoted by $m$. 

To define the comultiplication $m^*:\sR\longto\sR\ot\sR$, let $\mathbf{r}=(r_1,\ldots,r_k)$ be a partition of $r$ and 
let $\tau\in\Rep(\GL_r(F))$.  We may naturally consider 
\[
s.s.\left(\Jac_{N_{\mathbf{r}}(F)}(\tau)\right)\in\sR(\GL_{r_1}(F))\ot\cdot\cdot\cdot\ot\sR(\GL_{r_k}(F)).
\]
Now, set\footnote{We understand that $N_{(r,0)}=N_{(0,r)}$ is the trivial group.} 
\[
m^*(\tau)=\sum_{i=0}^{r}s.s.\left(\Jac_{N_{(i,r-i)}(F)}(\tau)\right)\in\sR\ot\sR.
\]
Again, one extends $m^*$ $\bbZ$-linearly to all of $\sR$. With the multiplication $m$ and the comultiplication $m^*$, 
$\sR$ becomes a $\bbZ_+$-graded Hopf algebra. 

The group $\sR(S)$ becomes a $\bbZ_+$-graded $\sR$-module and a $\bbZ_+$-graded comodule over $\sR$, through 
parabolic induction and Jacquet modules, respectively. In fact, the multiplication $\mu:\sR\x\sR(S)\longto\sR(S)$ can be defined 
as follows. For an irreducible smooth representation $\tau$ (resp. $\sigma$) in $\sR$ (resp. $\sR(S)$), we set 
\[
\mu(\tau,\sigma)=s.s.(\tau\rtimes\sigma).
\]
We extend $\mu$ $\bbZ$-linearly to $\sR\x\sR(S)$. The induced map $\sR\ot\sR(S)\longto\sR(S)$ is also denoted by $\mu$. 
In this way, $\sR(S)$ becomes a $\bbZ_+$-graded $\sR$-module.

We also define a $\bbZ$-linear map $\mu^*:\sR(S)\longto\sR\ot\sR(S)$ by first defining it on an irreducible smooth representation $\pi$ 
as
\[
\mu^*(\pi)=\sum_{\substack{S\subset\Delta_n\\|S|\le 1}}s.s.\left(\Jac_{N_S(F)}(\pi)\right). 
\]
This endows $\sR(S)$ with a $\bbZ_+$-graded comodule structure over $\sR$. 

It is not hard to see that $\sR(S)$ is not a Hopf module over $\sR$ (\cite[Remark 7.3]{Tadic1995}); however, the structure of $\sR(S)$ 
over $\sR$ was determined by Tadi\'c. To describe his result, note that $\sR\ot\sR$ acts naturally on $\sR\ot\sR(S)$; we also denote 
this natural action by $\rtimes$. Now, define
\[
M^*=(m\ot 1)\circ(\vee\ot m^*)\circ s\circ m^*:\sR\longto\sR\ot\sR,
\]
where $1$ denotes the identity map, $\vee$ denotes the contragredient map and $s$ denotes the transposition map
$\sum_i  x_i\ot y_i\mapsto\sum_i y_i\ot x_i$. The following is the main result in \cite{Tadic1995}. 

\begin{thm}[Tadic]\label{T:Tadic}
$\sR(S)$ is an $M^*$-Hopf module over $\sR$. Thus, for $\tau\in\Irr(\GL_r(F))$ and $\sigma\in\Irr(\SO_{2n+1}(F))$, we have 
\[
\mu^*(\tau\rtimes\sigma)=M^*(\tau)\rtimes\mu^*(\sigma).
\]
\end{thm}

The following definition is crucial for our second reduction. 

\begin{defn}
Let $\Pi\in\Rep(\SO_{2n+1}(F))$ and write 
\[
\mu^*\(\Pi\)=\sum_{i\in I}\tau_i\bt\pi_i,
\]
where $\tau_i\in\Irr(\GL_{r_i}(F))$ and $\pi_i\in\Irr(\SO_{2(n-r_i)+1}(F))$. Define subsets $I_1\subset I_2$ of $I$ by 
\[
I_1
=
\stt{i\in I\mid\text{$\tau_i$ is unramified and $\pi_i$ is generic}}
\subset
I_2
=
\stt{i\in I\mid\text{$\pi_i$ is generic}}. 
\]
Now set 
\[
\mu^*_{ur}\(\Pi\):=\sum_{i\in I_1}\tau_i\le \mu^*_{\GL}\(\Pi\):=\sum_{i\in I_2}\tau_i\in\sR, 
\]
and define
\[
|\mu^*_{ur}\(\Pi\)|=|I_1|. 
\]
\end{defn}

\begin{remark}
Note that if $\Pi$ is of Whittaker type, then $|\mu^*_{ur}(\Pi)|\ge 1$, since in this case $\mu^*_{ur}(\Pi)$ contains the trivial representation of 
$1_{\GL_0(F)}$. One of our reductions is based on counting $|\mu^*_{ur}(\Pi)|$ for certain representations $\Pi$. 
\end{remark}

\subsection{$P_{n+1}$-theory}\label{SS:P-theory}
In this subsection, we apply the $P_{n+1}$-theory developed in \cite[\S 3]{BZ1977} to derive some useful results needed for our reduction 
to square-integrable representations. It was first observed by Gelbart and Piatetski-Shapiro that $P_{n+1}$-theory has applications to the 
study of local Rankin--Selberg integrals for $\SO_{2n+1}\x\GL_n$. In \cite{Tsai2013}, Tsai also used $P_{n+1}$-theory to investigate 
local newforms for generic supercuspidal representations of $\SO_{2n+1}(F)$. Our derivations are based on combining their results. 

\subsubsection{Irreducible representations of $P_{n+1}$}
Let $n\ge 0$ be an integer.  Denote by $P_{n+1}$ the mirabolic subgroup of $\GL_{n+1}(F)$; that is, $P_{n+1}$ is the subgroup of 
matrices with the last row $(0,0,\cdots,0,1)$. We also denote 
\[
P_{n+1}(\frak{o})=P_{n+1}\cap\GL_{n+1}(\frak{o}).  
\]

The set $\Irr(P_{n+1})$ is classified in \cite[\S 3]{BZ1977} using induction and Jacquet modules. To describe their results, let $0\le r\le n$ be an 
integer and let $Q_r\subset P_{n+1}$ be the subgroup defined by 
\[
Q_r
=
\stt{\pMX{a}{b}{0}{z}\mid \text{$a\in\GL_r(F)$, $b\in\M_{r,n+1-r}(F)$, and $z\in Z_{n+1-r}$}}. 
\]
Each $\tau\in\Rep(\GL_r(F))$ induces a smooth representation $\tau\boxtimes\psi_{Z_{n+1-r}}$ of $Q_r$ via
\[
\pMX{a}{b}{0}{z}\longmapsto\tau(a)\psi_{Z_{n+1-r}}(z). 
\]
We thus obtain a normalized compactly induced representation 
\[
\ind_{Q_r}^{P_{n+1}}\left(\tau\boxtimes\psi_{Z_{n+1-r}}\right)
\]
of $P_{n+1}$. Now, \cite[Corollary 3.5]{BZ1977} asserts that 
\[
\sigma\in\Irr(P_{n+1})\Longleftrightarrow\sigma\simeq\ind_{Q_r}^{P_{n+1}}\left(\tau\boxtimes\psi_{Z_{n+1-r}}\right)
\]
for some $0\le r\le n$ and $\tau\in\Irr(\GL_r(F))$. Note that when $r=0$, so that $\tau$ is the trivial representation, 
$\ind_{Z_{n+1}(F)}^{P_{n+1}}\left(\psi_{Z_{n+1}}\right)$ is the standard (irreducible) representation of Gelfand--Greav. 

\begin{lm}\label{L:P fixed}
Let $\tau\in\Rep(\GL_r(F))$ with $0\le r\le n$. Then 
\[
\ind_{Q_r}^{P_{n+1}}\left(\tau\boxtimes\psi_{Z_{n+1-r}}\right)^{P_{n+1}(\frak{o})}\ne 0
\Longleftrightarrow \tau^{\GL_r(\frak{o})}\ne 0. 
\]
\end{lm}

\begin{proof}
Suppose that we have
\[
0\ne f\in\ind_{Q_r}^{P_{n+1}}\left(\tau\boxtimes\psi_{Z_{n+1-r}}\right)^{P_{n+1}(\frak{o})}. 
\]
By the Iwasawa decomposition for $\GL_n(F)$, the set
\[
\stt{\diag{I_r, t, 1}\in P_{n+1}\mid t\in A_{n-r}(F)}
\]
contains a set of representatives of 
\[
Q_r\backslash P_{n+1}/P_{n+1}(\frak{o}).
\]
It follows that 
\[
f\left(\diag{I_r, t_0, 1}\right)=v\in\tau
\]
is non-zero for some $t_0\in A_{n-r}(F)$. Now, we have $v\in\tau^{\GL_r(\frak{o})}$; indeed, if $a\in\GL_r(\frak{o})$, then 
\[
\tau(a)v
=
f\left(\diag{a,t_0,1}\right)=v. 
\]

Conversely, let $0\ne v\in\tau^{\GL_r(\frak{o})}$, and define 
\[
f_v\in\ind_{Q_r}^{P_{n+1}}\left(\tau\boxtimes\psi_{Z_{n+1-r}}\right)^{P_{n+1}(\frak{o})},
\]
by requiring that
\[
\supp(f_v)=Q_rP_{n+1}(\frak{o})
\quad\text{and}\quad
f_v(I_{n+1})=v. 
\]
Since $f_v$ is well-defined, the result follows 
\end{proof}

\subsubsection{Results of Gelbart--Piatetski-Shapiro and Tsai}
We now introduce relevant results of Gelbart--Piatetski-Shapiro and Tsai. Let $Y_n\subset P_{\a_n}(F)$ be the subgroup 
defined by 
\[
Y_n=\stt{\begin{pmatrix}I_n&&Y\\&1\\&&I_n\end{pmatrix}\mid\text{$Y\in\M_n(F)$ with $J_n{}^tYJ_n=-Y$}}. 
\]
Then $Y_n$ is normal in $P_{\a_n}(F)$, and 
\[
P_{\a_n}(F)/Y_n\simeq P_{n+1}
\]
by \cite[Proposition 2.2]{GPSR1987}. 

By restricting the action of $\pi\in\Rep(\SO_{2n+1}(F))$ to $P_{\a_n}(F)$, we obtain a smooth representation $\pi|_{P_{\a_n}(F)}$
of $P_{\a_n}(F)$. Using the aforementioned isomorphism, 
\[
\Jac_{Y_n}\left(\pi|_{P_{\a_n}(F)}\right)
\]
can be viewed as a smooth representation of $P_{n+1}$. We then have the following lemma due to Gelbart--Piatetski-Shapiro:

\begin{lm}\label{L:GPS}
Let $\pi\in\Rep(\SO_{2n+1}(F))$. Then 
\[
\Jac_{Y_n}\left(\pi|_{P_{\a_n}(F)}\right)\in\Rep(P_{n+1}).
\]
Moreover, if 
\[
d:=\dim_\bbC{\rm Hom}_{U_n(F)}\left(\pi,\psi_{U_n}\right)\le 1,
\]
then 
\[
s.s.\left(\Jac_{Y_n}\left(\pi|_{P_{\a_n}(F)}\right)\right)
=
d\cdot \ind_{Z_{n+1}(F)}^{P_{n+1}}\left(\psi_{Z_{n+1}}\right)
+
\sum_{i\in I}\ind_{Q_{r_i}}^{P_{n+1}}\left(\tau_i\bt\psi_{Z_{n+1-r_i}}\right)
\in\sR\(P_{n+1}(F)\)
\]
for some $\tau_i\in\Irr\(\GL_{r_i}(F)\)$ with $r_i\ge 1$, which satisfy 
\[
\tau_i|\cdot|^{\frac{1-n}{2}}\le\mu^*_{\GL}(\pi). 
\]
\end{lm}

\begin{proof}
It suffices to prove the lemma when $\pi\in\Irr\(\SO_{2n+1}(F)\)$. In this case, the first assertion was proved in 
\cite[Proposition 8.2]{GPSR1987} when $\pi$ is generic. Their argument certainly carries over to non-generic representations as well.

We now explain that the second assertion is already contained in their proof of the same proposition. When $\pi$ is generic, they proved that 
\[
\ind_{Z_{n+1}(F)}^{P_{n+1}}\left(\psi_{Z_{n+1}}\right)
\]
appears in $s.s.\left(\Jac_{Y_n}\left(\pi|_{P_{\a_n}(F)}\right)\right)$ with multiplicity one. On the other hand, if $\pi$ is non-generic but
$s.s.\left(\Jac_{Y_n}\left(\pi|_{P_{\a_n}(F)}\right)\right)$ contains $\ind_{Z_{n+1}(F)}^{P_{n+1}}\left(\psi_{Z_{n+1}}\right)$ as a summand, 
then the argument in \cite[p. 110]{GPSR1987} implies that $\pi$ must be generic, which is a contradiction. We thus conclude that 
$\ind_{Z_{n+1}(F)}^{P_{n+1}}\left(\psi_{Z_{n+1}}\right)$ appears in $s.s.\left(\Jac_{Y_n}\left(\pi|_{P_{\a_n}(F)}\right)\right)$ with multiplicity $d$. 

It remains to show that if 
\[
\ind_{Q_{r_i}}^{P_{n+1}}\left(\tau_i\bt\psi_{Z_{n+1-r_i}}\right)
\]
appears in $s.s.\left(\Jac_{Y_n}\left(\pi|_{P_{\a_n}(F)}\right)\right)$ for some $\tau_i\in\Irr\(\GL_{r_i}(F)\)$ with $r_i\ge 1$, then 
\[
\tau_i|\cdot|^{\frac{1-n}{2}}\le\mu^*_{\GL}(\pi). 
\]
This follows from their proof in \cite[pp. 111-114]{GPSR1987}. To prove that $s.s.\left(\Jac_{Y_n}\left(\pi|_{P_{\a_n}(F)}\right)\right)$ 
has finite length, their idea is to show that $\tau_i|\cdot|^{\frac{1-n}{2}}$ appears in the following (twisted) Jacquet module:
\[
\Jac_{U_{n-r_i}(F),\psi_{U_{n-r_i}}}\(\Jac_{N_{\a_{r_i}}(F)}\(\pi\)\),
\]
which is known to have finite length. We note that the twist by $|\cdot|^{\frac{1-n}{2}}$ is needed here because both their induction 
and their Jacquet modules are unnormalized. Now, in view of our definition of $\mu^*_{\GL}(\pi)$, the proof follows. 
\end{proof}

Next, we state a result of Tsai. To this end, recall that the definitions and properties of the relevant families of open compact subgroups 
$K_{n,m}\subset J_{n,m}$ of $\SO_{2n+1}(F)$ can be found in \cite{Gross2015} and \cite{Tsai2013} (see also \cite{YCheng2022} and 
\cite{YCheng2025}). 

\begin{lm}\label{L:Tsai}
Let $\pi\in\Rep\(\SO_{2n+1}(F)\)$. The natural map
\[
\pi^{K_{n,m}}
\longto
\Jac_{Y_n}\(\pi|_{P_{\a_n}(F)}\)^{P_{n+1}(\frak{o})}
\]
is injective for each integer $m\ge 0$. 
\end{lm}

\begin{proof}
This lemma follows from \cite[Lemma 3.4.1]{Tsai2013}, as we now explain. Let $\SO_{2n}(F)\subset{\rm SL}_{2n}(F)$ be the split 
even special orthogonal group defined by $J_{2n}$. It can be embedded into $\SO_{2n+1}(F)$ via
\begin{equation}\label{E:embedding}
\pMX{a}{b}{c}{d}\longmapsto\begin{pmatrix}a&&b\\&1\\c&&d\end{pmatrix},
\end{equation}
where $a,b,c,d\in\M_n(F)$. On the other hand, we regard $\GL_n(F)$ as a subgroup of $P_{n+1}$ in the usual way. Now, 
\cite[Lemma 3.4.1]{Tsai2013} implies that the natural map
\begin{equation}\label{E:Tsai}
\pi^{\(K_{n,m}\cap\SO_{2n}(F)\)}\longto\Jac_{Y_n}\(\pi|_{P_{\a_n}(F)}\)^{\GL_{n}(\frak{o})}
\end{equation}
is injective for each integer $m\ge 0$. We note that the group $H_{x_{m_+}}$ in $loc.$ $cit.$ is a subgroup of $K_{n,m}\cap\SO_{2n}(F)$;
hence Tsai's result is in fact stronger. 

To deduce our lemma from this, we simply note that the image of the map
\[
\(K_{n,m}\cap P_{\a_n}(F)\)/\(K_{n,m}\cap Y_n\)\hookto P_{\a_n}(F)/Y_n\simeq P_{n+1}
\] 
is $P_{n+1}(\frak{o})$ for each $m\ge 0$. Consequently, the image of the map \eqref{E:Tsai}, when restricted to $\pi^{K_{n,m}}$, lies in 
\[
\Jac_{Y_n}\(\pi|_{P_{\a_n}(F)}\)^{P_{n+1}(\frak{o})}. 
\]
This finishes the proof. 
\end{proof}

\subsubsection{Some consequences}
The following results, which are easy consequences of those in the previous subsections, in fact motivate this project.

\begin{prop}\label{P:mu=0}
Let $\pi\in\Rep\(\SO_{2n+1}(F)\)$. Suppose that 
\[
\mu^*_{ur}(\pi)=0. 
\]
Then $\pi^{K_{n,m}}=0$ for every $m\ge 0$. 
\end{prop}

\begin{proof}
By \lmref{L:Tsai}, it suffices to show that 
\[
\Jac_{Y_n}\(\pi|_{P_{\a_n}(F)}\)^{P_{n+1}(\frak{o})}
=0. 
\]
This would follow if we could show that
\[
\ind_{Q_{r_i}}^{P_{n+1}}\left(\tau_i\bt\psi_{Z_{n+1-r_i}}\right)^{P_{n+1}(\frak{o})}=0
\]
for each irreducible summand $\ind_{Q_{r_i}}^{P_{n+1}}\left(\tau_i\bt\psi_{Z_{n+1-r_i}}\right)$ of 
$s.s.\(\Jac_{Y_n}\(\pi|_{P_{\a_n}(F)}\)\)$ by \lmref{L:GPS}. To this end, we apply the same lemma and the assumption 
\[
\mu^*_{ur}(\pi)=0
\]
to conclude that $r_i\ge 1$ and $\tau_i^{\GL_{r_i}(\frak{o})}=0$ for each $i$. The proposition now follows immediately from \lmref{L:P fixed}. 
\end{proof}

As a corollary, we obtain:

\begin{cor}\label{C:key P}
Let $\Pi\in\Rep\(\SO_{2n+1}(F)\)$ and let $\pi$ be a subrepresentation of $\Pi$. Suppose that 
\[
|\mu^*_{ur}\(\Pi\)|=|\mu^*_{ur}(\pi)|.
\]
Then 
\[
\Pi^{K_{n,m}}=\pi^{K_{n,m}}
\]
for each $m\ge 0$. 
\end{cor}

\begin{proof}
Since $\pi$ is a subrepresentation of $\Pi$, we have 
\[
\mu^*_{ur}(\pi)\le \mu^*_{ur}\(\Pi\). 
\]
The assumption  
\[
|\mu^*_{ur}\(\Pi\)|=|\mu^*_{ur}(\pi)|
\]
then implies $\mu^*_{ur}\(\Pi\)=\mu^*_{ur}(\pi)$; hence 
\[
\mu^*_{ur}(\sigma)=0,
\]
where $\sigma=\Pi/\pi$. By the following exact sequence 
\[
0\longto\pi^{K_{n,m}}\longto\Pi^{K_{n,m}}\longto\sigma^{K_{n,m}}\longto 0,
\]
and \propref{P:mu=0}, we deduce that $\pi^{K_{n,m}}=\Pi^{K_{n,m}}$ for each $m\ge 0$. This completes the proof. 
\end{proof}

\begin{remark}
\corref{C:key P} is particularly important for our reduction. This is because square-integrable representations $\pi$ can be realized as 
subrepresentations of certain fully induced representations $\Pi$. Our previous result in \cite{YCheng2025} then allows us to compute 
$a_{\Pi}$ explicitly and to show that
\[
\dim_\bbC\Pi^{K_{n,a_\Pi}}=1. 
\]
Now, if $a_{\Pi}=c_{\pi}$ and $|\mu^*_{ur}(\Pi)|=|\mu^*_{ur}(\pi)|$, then we can conclude immediately that $\pi^{K_{n,c_\pi}}\ne 0$. 
This simple observation turns out to be very useful.
\end{remark}


\subsection{Representations of $\GL_r(F)$}\label{SS:rep of GL}
Fix an irreducible supercuspidal representation $\rho$ of $\GL_k(F)$ with $k>0$. Let $x,y\in\bbR$ such that $x-y\in\bbZ$. 
If $x-y\ge 0$, then we denote by $\Delta_\rho[x,y]$ the unique irreducible subrepresentation of 
\[
\d_\rho[x,y]
:=
\rho|\cdot|^x\x\rho|\cdot|^{x-1}\x\cdot\cdot\cdot\x\rho|\cdot|^y. 
\]
Then $\Delta_\rho[x,y]$ is an essentially square-integrable representation of $\GL_{k(x-y+1)}(F)$, and we have 
\[
\D_\rho[x,y]^\vee\simeq\D_{\rho^\vee}[-y, -x]. 
\]
When $x-y<0$, we understand that $\D_\rho[x,y]$ is the trivial representation of $\GL_0(F)$. 

If $y=-x\le 0$, then we also denote 
\[
\St_\rho(x)=\D_\rho[x,-x],
\]
which is called a Steinberg representation of $\GL_{k(2x+1)}(F)$. In general, 
\begin{equation}\label{E:twist}
\D_\rho[x,y]=\St_\rho\(\frac{x-y}{2}\)|\cdot|^{\frac{x+y}{2}}
\end{equation}
is a twist of a Steinberg representation. Note that the $L$-parameter of $\St_\rho(x)$ is  
\[
\phi_\rho\bt S_{2x+1},
\]
where $\phi_\rho$ is the $L$-parameter of $\rho$, while $S_d$ denotes the irreducible $d$-dimensional representation of 
${\rm SL}_2(\bbC)$. 

Following Atobe in \cite{Atobe2020}, define $\Irr_\rho\(\GL_{k\el}(F)\)$ to be the subset of $\Irr\(\GL_{k\el}(F)\)$ consisting of $\tau$ with
cuspidal support of the form $\rho|\cdot|^{x_1}\x\cdot\cdot\cdot\x\rho|\cdot|^{x_\el}$, i.e.,
\[
\tau\hookto\rho|\cdot|^{x_1}\x\cdot\cdot\cdot\x\rho|\cdot|^{x_\el}
\]
for some $x_1,\ldots, x_\el\in\bbR$. We understand that the trivial representation of $\GL_0(F)$ is contained in $\Irr_\rho\(\GL_0(F)\)$. 
To describe the representations in $\Irr_\rho\(\GL_{k\el}(F)\)$, let $\Omega_\el\subset\bbR^\el$ be the subset consisting of the 
elements 
\begin{equation}\label{E:ul(z)}
\ul{z}=\(x_1,x_1-1,\ldots, y_1,\, x_2, x_2-1,\ldots, y_2, \ldots,\, x_t, x_t-1, \ldots, y_t\)
\end{equation}
such that 
\begin{itemize}
\item[--] $x_i-y_i\in\bbZ_+$ for $i=1,\ldots, t$;
\item[--] $x_1\le x_2\le\cdot\cdot\cdot\le x_t$;
\item[--] $y_{i-1}\le y_i$ if $x_{i-1}=x_i$. 
\end{itemize}
In particular, we have $\Omega_0=\stt{0}$. Denote by $\D_\rho(\ul{z})$ the unique irreducible subrepresentation of 
\[
\d_\rho(\ul{z}):=\D_\rho[x_1, y_1]\x\D_\rho[x_2,y_2]\x\cdot\cdot\cdot\x\D_\rho[x_t, y_t]. 
\]
Our convention is that $\d_\rho(\ul{z})=\D_\rho(\ul{z})$ is the trivial representation of $\GL_0(F)$ when $\el=0$. 

Now, we have the following observation, which, due to Atobe (\cite[Lemma 2.2]{Atobe2020}), is based on the results of 
Zelevinsky (\cite{Zelevinsky1980}). 

\begin{lm}\label{L:Atobe}
The map 
\[
\Omega_\el\longto\Irr_\rho\(\GL_{k\el}(F)\);\quad
\ul{z}\longmapsto\D_\rho(\ul{z})
\]
gives rise to a bijection. 
\end{lm}

In the computations below, we need to determine when the representations $\D_\rho[x,y]$ and $\d_\rho(\ul{z})$ contain an unramified 
constituent. This question is addressed in the next two lemmas. 

\begin{lm}\label{L:unram 1}
Let $\rho$ be an irreducible supercuspidal representation of $\GL_k(F)$ and let $x,y\in\bbR$ such that $x-y\in\bbZ$. Then 
\[
\D_\rho[x,y]^{\GL_{k(x-y+1)}(\frak{o})}\ne 0
\]
if and only if
\begin{itemize}
\item[(i)] $x<y$, or
\item[(ii)] $x=y$, $k=1$, and $\rho$ is an unramified character of $\GL_1(F)=F^\x$. 
\end{itemize}
\end{lm}

\begin{proof}
This follows immediately from our convention and \cite[Corollary 1.3]{Matringe2013}. 
\end{proof}

\begin{lm}\label{L:unram 2}
Let $\rho$ be an irreducible supercuspidal representation of $\GL_k(F)$ and let $\ul{z}\in\Omega_\el$ be given by \eqref{E:ul(z)}. 
Then 
\[
\dim_\bbC\d_\rho(\ul{z})^{\GL_{k\el}(\frak{o})}\le 1,
\]
and we have
\[
\d_\rho(\ul{z})^{\GL_{k\el}(\frak{o})}\ne 0
\]
if and only if $x_i=y_i$ for $i=1,\ldots, t=\el$, $k=1$, and $\rho$ is an unramified character of $\GL_1(F)=F^\x$. 
\end{lm}

\begin{proof}
By \cite[Proposition 1.5]{Matringe2013}, we have
\[
\d_\rho(\ul{z})^{\GL_{k\el}(\frak{o})}
\simeq
\D_\rho[x_1, y_1]^{\GL_{k\(x_1-y_1+1\)}(\frak{o})}
\x\cdot\cdot\cdot\x
\D_\rho[x_t, y_t]^{\GL_{k\(x_t-y_t+1\)}(\frak{o})}. 
\]
Now the lemma follows from \lmref{L:unram 1} and the well-known fact that 
\[
\dim_\bbC\tau^{\GL_r(\frak{o})}\le 1
\]
for every $\tau\in\Irr(\GL_r(F))$. 
\end{proof}

In this paper, we are mostly interested in the case where $\rho=\chi$ is an unramified quadratic character (possibly trivial) of 
$\GL_1(F)=F^\x$. In this case, we can explicitly compute the associated $L$-factor and $\epsilon$-factor. 

\begin{lm}\label{L:cond}
Let $\chi$ be an unramified quadratic character of $F^\x$ and let $2x\in\bbZ_+$.  
Then we have 
\[
L\(s, \phi_\chi\bt S_{2x+1}\)
=
\(1-\chi(\varpi) q^{-(s+x)}\)^{-1}
\quad\text{and}\quad
\epsilon\(s,\phi_\chi\bt S_{2x+1},\psi\)
=
\(-\chi(\varpi)\)^{2x}q^{-2x\(s-\frac{1}{2}\)}. 
\]
In particular, 
\[
c_{\St_\chi(x)}=2x. 
\]
\end{lm}

\begin{proof}
It is equivalent to compute 
\[
L\(s, \St_\chi(x)\)
\quad\text{and}\quad
\epsilon\(s,\St_\chi(x),\psi\),
\]
where the $L$-factor and $\epsilon$-factor are those defined in \cite{GodementJacquet1972} or \cite{JPSS1983}. 

From \cite[\S 8]{JPSS1983} (see also \cite{Tate1979}), we deduce that 
\[
L\(s,\St_\chi(x)\)
=
L(s+x,\chi)
=
\(1-\chi(\varpi) q^{-(s+x)}\)^{-1}. 
\]
To compute the $\epsilon$-factor, note that 
\[
L(s,\chi)=-\chi(\varpi)q^s\cdot L(-s,\chi)
\quad\text{and}\quad
\epsilon(s,\chi,\psi)=1. 
\]
Since $\St_\chi(x)$ is self-dual, we have 
\[
\gamma\(s,\St_\chi(x),\psi\)
=
\epsilon\(s,\St_\chi(x),\psi\)
\frac{L(s+x,\chi)}{L(1-s+x,\chi)}. 
\]
It then follows from the multiplicativity of the $\gamma$-factor that
\begin{align*}
\epsilon\(s,\St_\chi(x),\psi\)
&=
\frac{L(1-s+x,\chi)}{L(s+x,\chi)}\cdot\gamma\(s,\St_\chi(x),\psi\)\\
&=
\frac{L(1-s+x,\chi)}{L(s+x,\chi)}\cdot\prod_{i=0}^{2x}\gamma(s-x+i,\chi,\psi)\\
&=
\frac{L(1-s+x,\chi)}{L(s+x,\chi)}\cdot\prod_{i=1}^{2x}\frac{L(1-s+x-i, \chi)}{L(s-x+i, \chi)}\\
&=
\prod_{i=0}^{2x-1}\frac{L(-s+x-i,\chi)}{L(s-x+i,\chi)}\\
&=
\prod_{i=0}^{2x-1}-\chi(\varpi) q^{-(s-x+i)}
=
\(-\chi(\varpi)\)^{2x}q^{-2x\(s-\frac{1}{2}\)}. 
\end{align*}
This completes the proof. 
\end{proof}

As a corollary, we obtain the following result. 

\begin{cor}\label{C:cond}
Let $\chi$ be an unramified quadratic character of $F^\x$ and let $x,y\in\bbR$ with $x-y\in\bbZ_+$. Then 
\[
c_{\D_\chi[x,y]}=x-y. 
\] 
\end{cor}

\begin{proof}
It follows from \eqref{E:twist} and \lmref{L:cond} that
\[
c_{\D_\chi[x,y]}=c_{\St\(\tfrac{x-y}{2}\)}=x-y,
\]
as desired. 
\end{proof}

\subsection{Representations of $\SO_{2n+1}(F)$}\label{SS:construction}
In this subsection, we describe the explicit construction from a given discrete $L$-parameter $\phi$ of $\SO_{2n+1}(F)$ (with $n>0$)
to its associated irreducible generic square-integrable representation $\pi$, following \cite[\S 2]{JiangSoudry2004}. 
This explicit construction plays an important role in our reduction. We also define the seed representations in this subsection. 

For convenience, given $\el\in\bbZ_+$, set
\[
\Irr_\el(\GL)=\bigsqcup_{r\ge \el}\Irr(\GL_r(F))
\quad\text{and}\quad
\Irr_\el(\SO)=\bigsqcup_{n\ge \el}\Irr(\SO_{2n+1}(F)).
\]
Recall that if $\rho\in\Irr(\GL_r(F))$ (with $r>0$) is a self-dual supercuspidal representation, then exactly one of 
\[
L\(s,\phi_\rho,\Lambda^2\)
\quad\text{or}\quad
L\(s,\phi_\rho,{\rm Sym}^2\)
\]
has a simple pole at $s=0$. Furthermore, when 
$L\(0,\phi_\rho,\Lambda^2\)=\infty$, $r$ must be even. 

Returning to the discrete $L$-parameter $\phi$, we can write 
\begin{equation}\label{E:L-par}
\phi=\bigoplus_{\rho\in I_\phi}\bigoplus_{\ka\in I_{\phi,\rho}}\phi_\rho\bt S_{2\ka+1},
\end{equation}
where
\begin{itemize}
\item[--] $I_\phi$ is a finite subset of $\Irr_1(\GL)$ consisting of self-dual supercuspidal representations $\rho$;
\item[--] if $L\(0,\phi_\rho, \Lambda^2\)=\infty$, then $I_{\phi, \rho}$ is a finite subset of $\bbZ_+$;
\item[--] if $L\(0,\phi_\rho,{\rm Sym}^2\)=\infty$, then $I_{\phi, \rho}$ is a finite subset of $\frac{1}{2}+\bbZ_+$. 
\end{itemize}
It is important to note that $\rho$ is unramified if and only if $\rho$ is an unramified quadratic character (possibly trivial) of $F^\x$. 
Note also that in this case
\[
L\(0,\phi_\rho,{\rm Sym}^2\)=\infty. 
\]
There are two such characters, and we denote them by $\chi$ and $\chi'$ from now on. In the following, we understand that 
if $\rho\nin I_\phi$, where $\rho\in\Irr_1(\GL)$ is a self-dual supercuspidal representation, then $I_{\phi, \rho}=\emptyset$.  

To construct $\pi$ from $\phi$, denote for each $\rho\in I_\phi$
\[
d_\rho=|I_{\phi, \rho}|
\quad\text{and}\quad
I_{\phi, \rho}=\stt{0\le \ka_{1}(\rho)<\ka_2(\rho)<\cdot\cdot\cdot<\ka_{d_\rho}(\rho)}. 
\]
Define the following (possibly empty) subsets of $I_\phi$:
\begin{align*}
I^0_\phi&=\stt{\rho\in I_\phi\mid\text{$L\(0,\phi_\rho,\Lambda^2\)=\infty$ and $d_\rho$ is odd}},\\
I^1_\phi&=\stt{\rho\in I_\phi\mid\text{$L\(0,\phi_\rho,\Lambda^2\)=\infty$ and $d_\rho$ is even}},\\
I^2_\phi&=\stt{\rho\in I_\phi\mid L\(0,\phi_\rho, {\rm Sym}^2\)=\infty}. 
\end{align*}
Then $I_\phi$ is the disjoint union of these subsets. We further decompose $I^0_\phi$ into the disjoint union of the following  
(possibly empty) subsets:
\begin{align*}
I^{00}_\phi&=\stt{\rho\in I^0_\phi\mid\text{$d_\rho=1$ and $\ka_1(\rho)=0$}},\\
I^{01}_\phi&=\stt{\rho\in I^0_\phi\mid\text{$d_\rho\ge 3$ and $\ka_1(\rho)=0$}},\\
I^{02}_\phi&=\stt{\rho\in I^0_\phi\mid\ka_1(\rho)\ge 1}. 
\end{align*}

By \cite{JiangSoudry2003}, there exists a unique irreducible generic supercuspidal representation $\sigma$ of $\SO_{2n_0+1}(F)$ 
for some $n_0\ge 0$ whose $L$-parameter is 
\[
\bigoplus_{\rho\in I^0_\phi}\phi_\rho. 
\]
Note that $I^0_\phi$ may be empty; in this case, our convention throughout is that $\sigma$ is the trivial representation of the trivial 
group $\SO_1(F)$. Now, to each $\rho\in I_\phi\setminus I^{00}_\phi$, we attach a finite set of essentially square-integrable 
representations in $\Irr_1(\GL)$ as follows. 
\begin{itemize}
\item
If $\rho\in I^{01}_\phi$, define
\[
\D_{\rho,j}=\D_\rho[\ka_{2j+1}(\rho), -\ka_{2j}(\rho)]
\quad\text{for $j=1,\ldots, \tfrac{d_\rho-1}{2}$},
\]
and let $J_\rho=\stt{1,\ldots,\frac{d_\rho-1}{2}}$. 
\item
If $\rho\in I^{02}_\phi$, define 
\[
\D_{\rho, 0}=\D_\rho[\ka_1(\rho), 1]
\quad\text{and}\quad
\D_{\rho,j}=\D_\rho[\ka_{2j+1}(\rho), -\ka_{2j}(\rho)]
\quad\text{for $j=1,\ldots, \tfrac{d_\rho-1}{2}$},
\]
and let $J_\rho=\stt{0, 1,\ldots,\frac{d_\rho-1}{2}}$. 
\item
If $\rho\in I^1_\phi$, define
\[
\D_{\rho, j}
=
\D_\rho[\ka_{2j}(\rho), -\ka_{2j-1}(\rho)]\quad\text{for $j=1,\ldots, \tfrac{d_\rho}{2}$},
\]
and let $J_\rho=\stt{1,\ldots,\frac{d_\rho}{2}}$. 
\item
If $\rho\in I^2_\phi$ and $d_\rho$ is even, define
\[
\D_{\rho, j}
=
\D_\rho[\ka_{2j}(\rho), -\ka_{2j-1}(\rho)]\quad\text{for $j=1,\ldots, \tfrac{d_\rho}{2}$},
\]
and let $J_\rho=\stt{1,\ldots,\frac{d_\rho}{2}}$. 
\item
If $\rho\in I^2_\phi$ and $d_\rho$ is odd, define
\[
\D_{\rho, 0}=\D_\rho[\ka_1(\rho), \tfrac{1}{2}]
\quad\text{and}\quad
\D_{\rho,j}=\D_\rho[\ka_{2j+1}(\rho), -\ka_{2j}(\rho)]
\quad\text{for $j=1,\ldots, \tfrac{d_\rho-1}{2}$},
\]
and let $J_\rho=\stt{0, 1,\ldots,\frac{d_\rho-1}{2}}$. 
\end{itemize}

Then, we have
\[
\Pi
:=
\(\bigtimes_{\rho\in I_\phi\setminus I^{00}_\phi}\bigtimes_{j\in J_\rho}\D_{\rho, j}\)\rtimes \sigma\in\Rep\(\SO_{2n+1}(F)\),
\]
and $\pi$ is the unique irreducible generic constituent of $\Pi$, which is, in fact, a subrepresentation. We observe that
\begin{itemize}
\item
for each $\rho\in I_\phi\setminus I^{00}_\phi$ and $j\in J_\rho$, there exists $x_{\rho, j}>0$ (see \eqref{E:twist}) such that 
\[
|\omega_{\D_{\rho,j}}(y)|=|y|^{x_{\rho,j}},
\]
where $y\in F^\x$; and
\item
the representation 
\[
\(\bigtimes_{\rho\in I_\phi\setminus I^{00}_\phi}\bigtimes_{j\in J_\rho}\D_{\rho, j}\)
\in\Rep(\GL_{n-n_0}(F))
\]
is indeed irreducible (\cite[Theorem 9.7]{Zelevinsky1980}). 
\end{itemize}

Now, we define the ``seed representations." Recall that the unramified quadratic characters (possibly trivial) of $F^\x$ are denoted by 
$\chi$ and $\chi'$. 

\begin{defn}\label{D:seed}
Let $\pi$ be an irreducible generic square-integrable representation of $\SO_{2n+1}(F)$ with $n\ge 0$. We say that $\pi$ is a seed 
representation if either $n=0$ and $\pi=1_{\SO_1(F)}$, or if $n>0$ and the associated 
$L$-parameter $\phi$, written as in \eqref{E:L-par}, satisfies
\[
0\le |I_{\phi, \chi}|,\, |I_{\phi, \chi'}|\le 1. 
\]
In the latter case, we also call $\phi$ a seed $L$-parameter. 
\end{defn}

The next lemma serves as the first step in reducing the proof to the case of seed representations.

\begin{lm}\label{L:inj conj}
Let $\D_1,\ldots, \D_\el\in\Irr_1(\GL)$ be essentially square-integrable representations, and let 
$\sigma\in\Irr_0(\SO)$ be a generic tempered representation. Assume the following:
\begin{itemize}
\item[--] for each $j=1,\ldots, \el$, there exists $x_j>0$ such that $|\omega_{\D_j}(y)|=|y|^{x_j}$, for all $y\in F^\x$;
\item[--] the representation $\D_1\x\cdot\cdot\cdot\x\D_\el$ is irreducible; and
\item[--] there exists $1\le j_0\le\el$ such that the unique irreducible generic constituent $\pi_0$ of 
\[
\D_{j_0}\x\cdot\cdot\cdot\x\D_{\el}\rtimes \sigma
\]
is tempered.
\end{itemize}
Then the unique irreducible generic constituent of 
\[
\D_{1}\x\cdot\cdot\cdot\x\D_{j_0-1}\rtimes \pi_0 
\]
is the same as the unique irreducible generic constituent of 
\[
\D_{1}\x\cdot\cdot\cdot\x\D_{\el}\rtimes \sigma. 
\]
Indeed, these unique irreducible generic constituents are subrepresentations. 
\end{lm}

\begin{proof}
This is a simple consequence of the generalized injectivity conjecture proposed by Casselman--Shahidi in \cite{CasselmanShahidi1998},
and proved by Hanzer in \cite{Hanzer2020} (see also \cite{Hanzer2010}). Indeed, our assumptions and the main result of
\cite{Hanzer2020} imply that $\pi_0$ is a subrepresentation of 
\[
\D_{j_0}\x\cdot\cdot\cdot\x\D_{\el}\rtimes \sigma. 
\]
On the other hand, if $\pi$ denotes the unique irreducible generic constituent of 
\[
\D_{1}\x\cdot\cdot\cdot\x\D_{j_0-1}\rtimes \pi_0, 
\]
then $\pi$ is also a subrepresentation for the same reason. Since 
\[
\pi\subset \D_{1}\x\cdot\cdot\cdot\x\D_{j_0-1}\rtimes \pi_0\subset \D_{1}\x\cdot\cdot\cdot\x\D_{\el}\rtimes\sigma, 
\]
it follows that $\pi$ is also the unique irreducible generic constituent of $\D_{1}\x\cdot\cdot\cdot\x\D_{\el}\rtimes\sigma$. 
\end{proof}

\section{Second reduction}\label{S:2nd red}
The aim of this section is to prove the following proposition. 

\begin{prop}\label{P:seed to sq}
If the space of newforms is non-zero for every seed representation, then it is also non-zero for all irreducible generic square-integrable
representations. 
\end{prop}

\subsection{Seed representations associated with square-integrable representations}\label{SS:red to seed}
Let $\pi$ be an irreducible generic square-integrable representation of $\SO_{2n+1}(F)$ with $n>0$, and let $\phi$ be the associated 
$L$-parameter, written as in \eqref{E:L-par}. 
Set 
\[
\phi_{00}=\bigoplus_{\rho\in I_\phi\setminus\stt{\chi,\chi'}}\bigoplus_{\ka\in I_{\phi,\rho}} \phi_{\rho}\bt S_{2\ka+1}. 
\]
We now associate to $\phi$ a seed $L$-parameter $\phi_0$ as follows.

\begin{itemize}
\item If both $|I_{\phi, \chi}|$ and $|I_{\phi, \chi'}|$ are even (including $0$), then we define
\[
\phi_0=\phi_{00}. 
\]
\item If one of $|I_{\phi, \chi}|$ or $|I_{\phi, \chi'}|$ is odd and the other is even, say $|I_{\phi, \chi}|$ is odd, then we define
\[
\phi_0
=
\phi_\chi\bt S_{2\ka_\chi+1}\oplus\phi_{00},
\]
where $\ka_\chi=\min\stt{\ka\mid \ka\in I_{\phi, \chi}}$. 
\item If both $|I_{\phi, \chi}|$ and $|I_{\phi, \chi'}|$ are odd, then we define
\[
\phi_0=
\phi_\chi\bt S_{2\ka_\chi+1}\oplus\phi_{\chi'}\bt S_{2\ka_{\chi'}+1}\oplus\phi_{00},
\]
where $\ka_\chi=\min\stt{\ka\mid \ka\in I_{\phi, \chi}}$ and $\ka_{\chi'}=\min\stt{\ka'\mid \ka'\in I_{\phi, \chi'}}$. 
\end{itemize}
Let $\pi_0$ be the irreducible generic square-integrable representation with $L$-parameter $\phi_0$. This representation is referred to as 
the seed representation associated with $\pi$. 
 
\subsection{A key proposition}
From the construction at the beginning of \S\ref{SS:construction} and \lmref{L:inj conj} 
(see also the proof of \lmref{L:cond non-seed}), we find that if $\phi\ne \phi_0$, then $\pi$ is the unique generic subrepresentation of 
\[
\Pi=\D_1\x\cdot\cdot\cdot\x\D_\el\rtimes\pi_0,
\]
where, for each $1\le j\le \el$, $\D_j$ is an irreducible essentially square-integrable representation in $\Irr_\chi\(\GL_{r_j}(F)\)$ or 
$\Irr_{\chi'}\(\GL_{r_j}(F)\)$ for some $r_j>0$. In particular,  \lmref{L:useful} implies that
\[
a_\Pi=a_{\pi_0}+2\sum_{j=1}^\el c_{\D_j}. 
\]
We next compute $c_\pi$. 

\begin{lm}\label{L:cond non-seed}
Let the notation be as above. Then
\[
c_{\pi}=c_{\pi_0}+2\sum_{j=1}^\el c_{\D_j}. 
\]
\end{lm}

\begin{proof}
Write $\phi_0$ as in \eqref{E:L-par} with $\phi$ replaced by $\phi_0$. The assumption that $\phi\ne\phi_0$ implies that 
\[
I_{\phi, \chi}\setminus I_{\phi_0,\chi}\ne\emptyset
\quad\text{or}\quad
I_{\phi, \chi'}\setminus I_{\phi_0,\chi'}\ne\emptyset. 
\]
Assume that both sets are non-empty. In this case, we can write
\[
\phi
=
\bigoplus_{\ka\in I_{\phi, \chi}\setminus I_{\phi_0,\chi}}\phi_\chi\bt S_{2\ka+1}
\oplus\phi_0\oplus
\bigoplus_{\ka'\in I_{\phi,\chi'}\setminus I_{\phi_0,\chi'}}\phi_{\chi'}\bt S_{2\ka'+1}.  
\]
Denote
\[
I_{\phi, \chi}\setminus I_{\phi_0,\chi}
=
\stt{0<\ka_1<\ka_2\cdot\cdot\cdot<\ka_{2a}}
\quad
\text{and}
\quad
I_{\phi, \chi'}\setminus I_{\phi_0,\chi'}
=
\stt{0<\ka'_1<\ka'_2\cdot\cdot\cdot<\ka'_{2b}}
\]
for some $a,b\in\bbN$. Then, by \lmref{L:cond}, 
\[
c_\pi=c_{\pi_0}+2\sum _{j=1}^{2a} \ka_j+2\sum_{j=1}^{2b} \ka'_j. 
\]
On the other hand, the construction at the beginning of \S\ref{SS:construction} and \lmref{L:inj conj} imply that
\[
\Pi=\bigtimes_{j=1}^a\D_\chi[\ka_{2j}, -\ka_{2j-1}]\,\,\x\,\,\bigtimes_{j=1}^b\D_{\chi'}[\ka'_{2j}, -\ka'_{2j-1}]\rtimes\pi_0. 
\]
It then follows from \corref{C:cond} that 
\[
\sum_{j=1}^a c_{\D_\chi[\ka_{2j}, -\ka_{2j-1}]}
+
\sum_{j=1}^b c_{\D_{\chi'}[\ka'_{2j},-\ka'_{2j-1}]}
=
\sum_{j=1}^{2a}\ka_j
+
\sum_{j=1}^{2b}\ka'_j. 
\]
Together, we obtain
\[
c_\pi
=
c_{\pi_0}
+
2\sum_{j=1}^a c_{\D_\chi[\ka_{2j}, -\ka_{2j-1}]}
+
2\sum_{j=1}^b c_{\D_{\chi'}[\ka'_{2j},-\ka'_{2j-1}]}. 
\]
This proves the case when both $I_{\phi,\chi}\setminus I_{\phi_0,\chi}$ and $I_{\phi,\chi'}\setminus I_{\phi_0,\chi'}$ are non-empty. 
The other case can be proved in a similar way. 
\end{proof}

By \lmref{L:cond non-seed}, we see that if $a_{\pi_0}=c_{\pi_0}$, then $a_{\Pi}=c_\pi$. In particular, if we can show that 
$\Pi^{K_{n,m}}=\pi^{K_{n,m}}$ for every $m\ge 0$, then we can conclude that $\pi^{K_{n,c_\pi}}\ne 0$. 
To this end, we require the following key proposition, whose proof is postponed to the remaining subsections of this section. 

\begin{prop}\label{P:key Jac}
Let $\pi$ be an irreducible generic square-integrable representation of $\SO_{2n+1}(F)$ with $n>0$. Assume that $\pi$ is not a seed
representation, and let $\pi_0$ be the seed representation associated with $\pi$. Thus, $\pi$ is the unique generic subrepresentation of 
\[
\Pi=\D_1\x\cdot\cdot\cdot\x\D_\el\rtimes\pi_0,
\]
where, for each $1\le j\le \el$, $\D_j$ is an irreducible essentially square-integrable representation in $\Irr_\chi\(\GL_{r_j}(F)\)$ or 
$\Irr_{\chi'}\(\GL_{r_j}(F)\)$ for some $r_j>0$. Then we have
\[
|\mu^*_{ur}(\Pi)|=|\mu^*_{ur}(\pi)|. 
\] 
\end{prop}

Assuming \propref{P:key Jac}, we now prove \propref{P:seed to sq}. 

\subsection{Proof of \propref{P:seed to sq}}
Let $\pi$ be an irreducible generic square-integrable representation of $\SO_{2n+1}(F)$ with $n>0$. Assume that $\pi$ is not a seed 
representation, and let $\Pi\in\Rep\(\SO_{2n+1}(F)\)$ be as in \propref{P:key Jac}. Then, by \lmref{L:cond}, \lmref{L:cond non-seed}, and 
the assumption, we have
\[
a_\Pi
=
a_{\pi_0}
+
2\sum_{j=1}^\el c_{\D_j}
=
c_{\pi_0}
+
2\sum_{j=1}^\el c_{\D_j}
=
c_\pi. 
\]
On the other hand, \corref{C:key P} and \propref{P:key Jac} imply
\[
\Pi^{K_{n,m}}=\pi^{K_{n,m}}
\]
for every $m\ge 0$. It follows that 
\[
\pi^{K_{n,c_\pi}}=\Pi^{K_{n,a_{\Pi}}}\ne 0, 
\]
as desired. \qed\\

The remainder of this section is devoted to proving \propref{P:key Jac}. 

\subsection{Unramified constituents I}
The goal of this and the next subsection is to prove \propref{P:key Jac}. In this subsection, we focus on computing the number 
\[
|\mu_{ur}^*(\Pi)|. 
\]
To state the next lemma, we introduce the following definition. If $\tau\in\Rep(\GL_r(F))$ with $r\ge 0$ and 
\[
M^*(\tau)
=
\sum_{i\in I}\tau_i\bt\tau_i'
\]
for some $\tau_i, \tau'_i\in\Irr_0(\GL)$, then we define
\[
M^*_{ur}(\tau)
=
\stt{i\in I\mid\text{$\tau_i$ is unramified and $\tau'_i$ is generic}}. 
\]
Now we have

\begin{lm}\label{L:unram 3}
Let $\tau\in\Rep(\GL_r(F))$ and $\pi\in\Rep\(\SO_{2n+1}(F)\)$ with $r, n\in\bbZ_+$. Then 
\[
|\mu^*_{ur}(\tau\rtimes\pi)|
=
|M^*_{ur}(\tau)|\cdot|\mu^*_{ur}(\pi)|. 
\]
\end{lm}

\begin{proof}
Write 
\[
M^*(\tau)=\sum_{i\in I}\tau_i\bt\tau_i'
\quad\text{and}\quad
\mu^*(\pi)
=
\sum_{j\in J}\tau''_j\bt\pi_j,
\]
for some $\tau_i, \tau'_i, \tau''_j\in\Irr_0(\GL)$ and $\pi_j\in\Irr_0(\SO)$. By \thmref{T:Tadic}, 
\[
\mu^*(\tau\rtimes\pi)
\simeq
M^*(\tau)\rtimes\mu^*(\pi)
\simeq
\sum_{i\in I,\,j\in J}\(\tau_i\x\tau''_j\)\rtimes\(\tau'_i\rtimes\pi_j\). 
\]
Since
\[
\text{$\tau_i\x\tau''_j$ is unramified $\Longleftrightarrow$ $\tau_i$ and $\tau''_j$ are unramified}
\]
by  \cite[Proposition 1.5]{Matringe2013}, and 
\[
\text{$\tau'_i\rtimes\pi_j$ is generic $\Longleftrightarrow$ $\tau'_i$ and $\pi_j$ are generic}
\]
by  \cite{Rodier1973}, it follows that the multiplicity of the unramified (resp. generic) constituent of $\tau_i\x\tau''_j$ 
(resp. $\tau'_i\rtimes\pi_j$) is at most one. This proves the lemma. 
\end{proof}

We next compute 
\[
|M_{ur}^*(\tau)|
\]
when $\tau$ is square-integrable. 

\begin{lm}\label{L:unram 4}
Let $\rho$ be an irreducible supercuspidal representation of $\GL_r(F)$ with $r>0$, and let $x,y\in\bbR$ with $x-y\in\bbZ_+$. Then 
\[
|M_{ur}^*\(\D_\rho[x,y]\)|
=
\begin{cases}
1\quad&\text{if $\rho$ is ramified},\\
3\quad&\text{if $\rho$ is unramified and $x-y=0$},\\
4\quad&\text{if $\rho$ is unramified and $x-y>0$}. 
\end{cases}
\]
\end{lm}

\begin{proof}
Since we can replace $\rho$ by any of its unramified twists without affecting the result, we may assume that $y=0$ and $x\in\bbZ_+$. 
By \cite[Proposition 9.5]{Zelevinsky1980} and a direct computation,
\begin{align*}
M^*\(\D_\rho[x,0]\)
=
\sum_{i=0}^{x+1}\sum_{j=0}^i
\(\D_{\rho^\vee}[0,i-x]\,\x\,\D_{\rho}[x, x+1-j]\)
\ot
\D_\rho[x-j, x+1-i]. 
\end{align*}
Note that $\D_\rho[x-j, x+1-i]$ is always generic; thus, it suffices to compute the number of representations 
\[
\D_{i,j}:=\D_{\rho^\vee}[0,i-x]\,\x\,\D_{\rho}[x, x+1-j]
\]
that contain an unramified constituent (necessarily with multiplicity one). 

If $\rho$ is ramified, then by \lmref{L:unram 1} and the proof of \lmref{L:unram 2}, we find that $\D_{i,j}$ is 
unramified if and only if it is the trivial representation of the trivial group. This, in turn, is equivalent to $(i,j)=(x+1, 0)$. Thus, 
\[
|M_{ur}^*\(\D_\rho[x,y]\)|
=
1
\]
when $\rho$ is ramified. On the other hand, if $\rho$ is unramified, then by the same lemmas, we have
\begin{align*}
\text{$\D_{i,j}$ is unramified}&\Longleftrightarrow\text{$0\le i-x$ and $x\le x+1-j$}\\
&\Longleftrightarrow\text{$(i,j)=(x,0), (x,1), (x+1, 0)$ or $(x+1, 1)$}. 
\end{align*}
Since $(i,j)$ cannot be $(x,1)$ when $x=0$, it follows that
\[
\text{$|M_{ur}^*\(\D_\rho[x,y]\)|=3$ or $4$,}
\]
depending on whether $x=0$ or $x>0$, respectively. This completes the proof. 
\end{proof}

These two lemmas together yield the following corollary. 

\begin{cor}\label{C:no of unram}
Let $\chi_i$ be unramified characters of $F^\x$, and let  $x_i, y_i\in\bbR$ with $x_i-y_i\in\bbN$. Set $\D_i=\D_{\chi_i}[x_i,y_i]$ 
for $i=1,\ldots, \el$. Then, for $\pi_0\in\Rep(\SO_{2n+1}(F))$ with $n\ge 0$, 
\[
|\mu^*_{ur}\(\D_1\x\cdot\cdot\cdot\x\D_\el\rtimes\pi_0\)|
=
4^\el\cdot|\mu^*_{ur}(\pi_0)|. 
\]
\end{cor}

\begin{proof}
Since 
\[
\D_1\x\cdot\cdot\cdot\x\D_\el\rtimes\pi_0
\simeq
\D_1\rtimes\(\D_2\x\cdot\cdot\cdot\x\D_\el\rtimes\pi_0\),
\]
the proof follows immediately from \lmref{L:unram 3}, \lmref{L:unram 4}, and induction on $\el$. 
\end{proof}

\subsection{Unramified constituents II}
In this subsection, we are devoted to computing the number
\[
|\mu^*_{ur}(\pi)|
\]
that appears in \propref{P:key Jac} when $\pi$ is a seed representation. The key to our computation is Atobe's paper \cite{Atobe2020}
(see also \cite{Xu2017}) on Jacquet modules and the local Langlands correspondence. Some preparations are needed. 

Let $\rho$ be an irreducible supercuspidal representation of $\GL_k(F)$ with $k>0$, and let $x\in\bbR$. 
Given $\tau\in\Rep(\GL_r(F))$ with $r\ge 0$, define the partial Jacquet module $\Jac_{\rho|\cdot|^x}(\tau)$ as follows. Write
\[
m^*(\tau)=\sum_{i\in I} \tau_i\bt\tau'_i,
\]
where $\tau_i, \tau'_i\in\Irr_0(\GL)$ for each $i\in I$. Then 
\[
\Jac_{\rho|\cdot|^x}
:=
\sum_{\substack{i\in I\\\tau_i\simeq\rho|\cdot|^x}}\tau'_i. 
\]
More generally, if $\ul{x}=(x_1,\ldots, x_\el)\in \bbR^\el$, then define
\[
\Jac_{\rho|\cdot|^{\ul{x}}}
=
\Jac_{\rho|\cdot|^{x_\el}}\circ\Jac_{\rho|\cdot|^{x_{\el-1}}}\circ\cdot\cdot\cdot\circ\Jac_{\rho|\cdot|^{x_1}}. 
\]
Note that if $r=k\el$, then 
\[
\Jac_{\rho|\cdot|^{\ul{x}}}(\tau)\in\sR(\GL_0(F)). 
\]
In particular,  $\Jac_{\rho|\cdot|^{\ul{x}}}(\tau)$ is a finite-dimensional $\bbC$-linear space. 

The following lemma, which computes $\Jac_{\rho|\cdot|^{\ul{x}}}(\tau)$ for certain $\tau$ and $\ul{x}$, is due to Atobe
(\cite[Lemma 2.4]{Atobe2020}). 

\begin{lm}\label{L:Atobe 2}
Let $\ul{z}\in\Omega_\el$ be as in \eqref{E:ul(z)}. For $(x,y)\in\stt{(x_i, y_i)\mid 1\le i\le t}$,  set 
\[
m_{(x,y)}=|\stt{1\le i\le t\mid (x,y)=(x_i, y_i)}|. 
\]
Then, for $\ul{z}'\in\Omega_\el$, we have
\[
\dim_\bbC\Jac_{\rho|\cdot|^{\ul{z}'}}\(\d_\rho(\ul{z})\)
=
\begin{cases}
\prod_{(x,y)\in\stt{(x_i, y_i)\mid 1\le i\le t}} m_{(x,y)}!\quad&\text{if $\ul{z}'=\ul{z}$},\\
0\quad&\text{if $\ul{z}'<\ul{z}$}. 
\end{cases}
\]
Here, $\bbR^\el$ is regarded as a totally ordered set with respect to the lexicographical order. 
\end{lm}

To proceed, let $\pi\in\Rep\(\SO_{2n+1}(F)\)$ with $n\ge 0$, and write 
\[
\mu^*(\pi)=\sum_{i\in I}\tau_i\bt\pi_i
\]
for some $\tau_i\in\Irr_0(\GL)$ and $\pi_i\in\Irr_0(\SO)$. Define
\[
\mu^*_\rho(\pi)
=
\sum_{\substack{i\in I\\\tau_i\in\Irr_\rho(\GL)}}
\tau_i\bt\pi_i,
\]
where 
\[
\Irr_\rho(\GL)
:=
\bigcup_{\ell\ge 0}\Irr_\rho\(\GL_{k\el}(F)\). 
\]
The next lemma is also due to Atobe (\cite[Lemma 2.7]{Atobe2020}). 

\begin{lm}\label{L:Atobe 3}
If we define 
\[
\iota:\sR(S)\longto\sR\ot\sR(S); \quad \pi\longmapsto 1_{\GL_0(F)}\ot \pi,
\]
then 
\[
\mu^*=\circ_\rho\((m\ot{\rm id})\circ({\rm id}\ot\mu^*_\rho)\)\circ\iota,
\]
where $\rho$ runs over all irreducible unitary supercuspidal representations of $\GL_k(F)$ with $k>0$. 
\end{lm}

At this point, we make some observations based on \lmref{L:Atobe 3}. Let $\pi$ be an irreducible generic square-integrable representation 
of $\SO_{2n+1}(F)$ with $n>0$, and let $\phi$ be the associated $L$-parameter, written as in \eqref{E:L-par}. Then, by \lmref{L:Atobe 3}, 
the summands of $\mu^*(\pi)$ are of the form
\[
\(\bigtimes_{\rho\in I_\phi}\tau_\rho\)\rtimes \pi_0,
\]
where $\tau_\rho\in\Irr_\rho(\GL)$ for each $\rho\in I_\phi$ and $\pi_0\in\Irr_0(\SO)$. 

Since our goal is to compute $|\mu^*_{ur}(\pi)|$, we are mainly interested in those summands for which 
\[
\bigtimes_{\rho\in I_\phi}\tau_\rho
\]
contains an unramified constituent (necessarily with multiplicity one). By \cite[Proposition 1.5]{Matringe2013}, it contains an unramified 
constituent if and only if each $\tau_\rho$ is unramified. Now, \lmref{L:unram 2} implies that $\tau_\rho=1_{\GL_0(F)}$ when 
$\rho\ne\chi, \chi'$, where we recall that $\chi$ and $ \chi'$ denote the unramified quadratic characters of $F^\x$. 

Consequently, to compute $|\mu^*_{ur}(\pi)|$, it suffices to understand 
\[
\(m\ot{\rm id}\)\circ\({\rm id}\ot\mu^*_{\chi'}\)\circ\mu^*_{\chi}(\pi),
\]
i.e., we only need to look at the summands of $\mu^*(\pi)$ of the form
\[
(\tau\,\x\,\tau')\rtimes\pi_0,
\]
where $\tau\in\Irr_\chi(\GL)$ and $\tau'\in\Irr_{\chi'}(\GL)$ are unramified, and $\pi_0\in\Irr_0(\SO)$. 

Now, we can state and prove the following key lemma. 

\begin{lm}\label{L:key Jac}
Let $\pi$ be a seed representation of $\SO_{2n+1}(F)$ with $n\ge 0$, and let $\phi$ be the associated $L$-parameter, 
written as in \eqref{E:L-par} when $n>1$. Then
\[
|\mu^*_{ur}(\pi)|
= 
2^{|I_{\phi, \chi}|+|I_{\phi, \chi'}|}.
\]
\end{lm}

\begin{proof}
The case $n=0$ follows from the definition; thus, we may assume that $n>0$ throughout the proof. We then have three cases 
depending on whether $I_{\phi, \chi}$ or $I_{\phi, \chi'}$ is empty.

Assume first that both of them are empty. Then, by the construction at the beginning of \S\ref{SS:construction}, $\pi$ is the unique 
generic subrepresentation of 
\[
\Pi
=
\(\bigtimes_{\rho\in I_\phi\setminus I^{00}_\phi}\bigtimes_{j\in J_\rho}\D_{\rho, j}\)\rtimes \sigma,
\]
where $\sigma\in\Irr_0(\SO)$ is a generic supercuspidal representation and $\rho\in\Irr_1(\GL)$ are self-dual supercuspidal 
representations.  Now, the assumption that $I_{\phi, \chi}=I_{\phi, \chi'}=\emptyset$ implies that each $\rho$ is ramified. Since 
\[
|\mu^*_{ur}(\sigma)|=1,
\]
it follows from \lmref{L:unram 3}, \lmref{L:unram 4}, and induction (see the proof of \corref{C:no of unram}) that 
\[
|\mu^*_{ur}(\Pi)|=1. 
\]
This implies that $|\mu^*_{ur}(\pi)|=1$. 

Assume next that both $I_{\phi, \chi}$ and $I_{\phi, \chi'}$ are non-empty. Then $I_{\phi, \chi}=\stt{\ka}$ and $I_{\phi, \chi'}=\stt{\ka'}$ 
for some $\ka, \ka'\in\frac{1}{2}+\bbZ_+$. In this case, we have to compute 
\[
\(m\ot{\rm id}\)\circ\({\rm id}\ot\mu^*_{\chi'}\)\circ\mu^*_{\chi}(\pi). 
\]
To compute $\mu^*_{\chi}(\pi)$, we define, for each integer $0\le \el\le\min\stt{n, 2\ka+1}$,
\[
\ul{\ka}(\el)=(\ka, \ka-1,\ldots, \ka-\el+1)\in\Omega_\el. 
\]
Then, by \cite[Theorem 4.3]{Atobe2020}, we have
\[
\mu^*_{\chi}(\pi)
=
\sum_{\el=0}^{\min\stt{n,2\ka+1}}\D_{\chi}[\ka, \ka-\el+1]\bt\pi_\el, 
\]
where
\[
\pi_\el=\Jac_{\chi|\cdot|^{\ul{\ka}(\el)}}(\pi)
\]
for $0\le \el\le\min\stt{n, 2\ka+1}$, so that $\pi_0=\pi$. Since 
$\D_\chi[\ka, \ka-\el+1]$ is unramified if and only if $\el=0$ or $1$ by \lmref{L:unram 1}, it suffices to compute 
\[
\(m\ot{\rm id}\)\circ\({\rm id}\ot\mu^*_{\chi'}\)\(\D_\chi[\ka,\ka-\el+1]\bt\pi_\el\)
\]
for $\el=0$ and $1$, by the observation preceding the lemma. 

When $\el=0$, we have 
\[
\(m\ot{\rm id}\)\circ\({\rm id}\ot\mu^*_{\chi'}\)(1\bt \pi)
=
\mu^*_{\chi'}(\pi),
\]
which, by applying \cite[Theorem 4.3]{Atobe2020} again, is equal to
\[
\sum_{\el'=0}^{\min\stt{n,2\ka'+1}}\D_{\chi'}[\ka', \ka'-\el'+1]\bt\Jac_{\chi'|\cdot|^{\ul{\ka}'(\el')}}(\pi)
\]
with $\ul{\ka}'(\el')\in\Omega_{\el'}$ defined analogously to $\ul{\ka}(\el)$ for 
$0\le \el'\le\min\stt{n, 2\ka'+1}$. Since $\D_{\chi'}[\ka', \ka'-\el'+1]$ is unramified if and only if $\el'=0$ or $1$, 
it suffices to consider the summands
\[
1\bt \pi
\quad\text{and}\quad
\chi'|\cdot|^{\ka'}\bt\Jac_{\chi'|\cdot|^{\ka'}}(\pi)
\]
corresponding to $\el'=0$ and $1$, respectively. Since $\pi$ is generic, we have 
\[
1\le \mu^*_{ur}(\pi).
\] 
On the other hand, by \cite[Theorem 4.2 (2)]{Atobe2020} (or rather \cite[Lemma 7.3]{Xu2017}), 
\[
\Jac_{\chi'|\cdot|^{\ka'}}(\pi)
\] 
is the irreducible generic representation of $\SO_{2n-1}(F)$ with $L$-parameter 
\[
\phi_1=\phi-\phi_{\chi'}\bt S_{2\ka'+1}+\phi_{\chi'}\bt S_{2\ka'-1}. 
\]
It follows that
\[
1+\chi'|\cdot|^{\ka'}\le\mu_{ur}^*(\pi). 
\]

Suppose that $\el=1$. Then we have
\begin{align*}
\(m\ot{\rm id}\)\circ\({\rm id}\ot\mu^*_{\chi'}\)&(\chi|\cdot|^\ka\bt \pi_1)\\
&=
\(m\ot{\rm id}\)\(\chi|\cdot|^\ka\ot\mu^*_{\chi'}\(\pi_1\)\)\\
&=
\sum_{\el'=0}^{\min\stt{n-1, 2\ka'+1}}\(\chi|\cdot|^\ka\,\x\,\D_{\chi'}[\ka',\ka'-\el'+1]\)\bt\Jac_{\chi'|\cdot|^{\ul{\ka}'(\el')}}(\pi_1).
\end{align*}
Since 
\[
\chi|\cdot|^\ka\,\x\,\D_{\chi'}[\ka',\ka'-\el'+1]
\]
is unramified if and only if $\D_{\chi'}[\ka',\ka'-\el'+1]$ is unramified, which holds if and only if $\el'=0$ or $1$, we only need to consider 
the summands
\[
\chi|\cdot|^\ka\bt \pi_1
\quad\text{and}\quad
\(\chi|\cdot|^\ka\,\x\,\chi'|\cdot|^{\ka'}\)\bt\Jac_{\chi'|\cdot|^{\ka'}}(\pi_1)
\]
corresponding to $\el'=0$ and $1$, respectively. Since $\pi_1$ is generic, we have
\[
\chi|\cdot|^\ka\le\mu_{ur}^*(\pi). 
\]
On the other hand, since $\chi|\cdot|^\ka\,\x\,\chi'|\cdot|^{\ka'}$ is unramified and irreducible, and by applying 
\cite[Theorem 4.2]{Atobe2020} (or \cite[Lemma 7.3]{Xu2017}) once more, 
\[
\Jac_{\chi'|\cdot|^{\ka'}}(\pi_1)
\] 
is the irreducible generic representation of $\SO_{2n-3}(F)$ with $L$-parameter 
 \[
\phi_1-\phi_{\chi'}\bt S_{2\ka'+1}+\phi_{\chi'}\bt S_{2\ka'-1},
 \]
we also have
\[
\chi|\cdot|^\ka\,\x\,\chi'|\cdot|^{\ka'}
\le
\mu^*_{ur}(\pi). 
\]
Altogether, we conclude that 
\[
\mu^*_{ur}(\pi)=1+\chi|\cdot|^\ka+\chi'|\cdot|^{\ka'}+\(\chi|\cdot|^\ka\,\x\,\chi'|\cdot|^{\ka'}\),
\]
which shows that $|\mu^*_{ur}(\pi)|=4$. 

Finally, assume that exactly one of $I_{\phi, \chi}$ or $I_{\phi, \chi'}$ is empty, say $I_{\phi, \chi}=\emptyset$. Then 
$I_{\phi, \chi'}=\stt{\ka'}$ for some $\ka'\in\frac{1}{2}+\bbZ_+$. In this case, we have
\[
\(m\ot{\rm id}\)\circ\({\rm id}\ot\mu^*_{\chi'}\)\circ\mu^*_{\chi}(\pi)
=
\mu^*_{\chi'}(\pi). 
\]
Now, the computation when $\el=0$ in the previous case shows that 
\[
\mu^*_{ur}(\pi)=1+\chi'|\cdot|^{\ka'}
\]
and hence $|\mu^*_{ur}(\pi)|=2$. This completes the proof. 
\end{proof}

\subsection{Proof of \propref{P:key Jac}}
Let $\phi$ be the $L$-parameter of $\pi$, written as in \eqref{E:L-par}, and let $\phi_0$ be the associated seed $L$-parameter, written
as in \eqref{E:L-par}, with $\phi$ replaced by $\phi_0$. We show that 
\[
|\mu^*_{ur}(\Pi)|=2^{|I_{\phi, \chi}|+|I_{\phi, \chi'}|}=|\mu^*_{ur}(\pi)|. 
\]
The assumption that $\phi\ne\phi_0$ implies that one of $I_{\phi, \chi}\setminus I_{\phi_0,\chi}$ or $I_{\phi, \chi'}\setminus I_{\phi_0,\chi'}$ is 
non-empty.  Assume first that both of them are non-empty. In this case, we can write 
(see \S\ref{SS:red to seed})
\[
\phi
=
\bigoplus_{\ka\in I_{\phi, \chi}\setminus I_{\phi_0,\chi}}\phi_\chi\bt S_{2\ka+1}
\oplus\phi_0\oplus
\bigoplus_{\ka'\in I_{\phi,\chi'}\setminus I_{\phi_0,\chi'}}\phi_{\chi'}\bt S_{2\ka'+1}.  
\]
Let
\[
I_{\phi, \chi}\setminus I_{\phi_0,\chi}
=
\stt{0<\ka_1<\ka_2\cdot\cdot\cdot<\ka_{2a}}
\quad
\text{and}
\quad
I_{\phi, \chi'}\setminus I_{\phi_0,\chi'}
=
\stt{0<\ka'_1<\ka'_2\cdot\cdot\cdot<\ka'_{2b}}
\]
for some $a,b\in\bbN$. It then follows from the construction at the beginning of \S\ref{SS:construction} and \lmref{L:inj conj} that 
\[
\Pi
=
\bigtimes_{j=1}^a\D_\chi[\ka_{2j}, -\ka_{2j-1}]\,\,\x\,\,\bigtimes_{j=1}^b\D_{\chi'}[\ka'_{2j}, -\ka'_{2j-1}]\rtimes\pi_0. 
\]
Now, \lmref{L:key Jac} and \corref{C:no of unram} imply that
\[
|\mu^*_{ur}(\Pi)|
=
2^{2a+2b}|\mu_{ur}^*(\pi_0)|
=
2^{2a+2b+|I_{\phi_0,\chi}|+|I_{\phi_0,\chi'}|}
=
2^{|I_{\phi,\chi}|+|I_{\phi,\chi'}|}. 
\]
This proves the first equality. 

To verify the second equality, note that 
\[
|\mu^*_{ur}(\pi)|
\le
|\mu^*_{ur}(\Pi)|
=
2^{|I_{\phi,\chi}|+|I_{\phi,\chi'}|}
\]
since $\pi$ is a subrepresentation of $\Pi$. Thus, it is enough to show that 
\[
|\mu^*_{ur}(\pi)|
\ge 
2^{|I_{\phi,\chi}|+|I_{\phi,\chi'}|}. 
\]
For this, we again apply the results of Atobe. By the observations preceding \lmref{L:key Jac}, we have to investigate
\[
\(m\ot{\rm id}\)\circ\({\rm id}\ot\mu^*_{\chi'}\)\circ\mu^*_{\chi}(\pi). 
\]
Assume that 
\[
I_{\phi,\chi}=\stt{0<\ka_1<\ka_2<\cdots<\ka_d}
\quad\text{and}\quad
I_{\phi,\chi'}=\stt{0<\ka'_1<\ka'_2<\cdots<\ka'_{d'}}. 
\]
Note that $d+d'\le n$. 
Following \cite[Theorem 4.3]{Atobe2020}, for an integer $0\le\el \le n$, denote by $K_{\phi,\chi}^{(\el)}$ the set of tuples of integers  
$\ul{a}=(a_1,\ldots, a_d)$ such that 
\begin{itemize}
\item[-] $0\le a_i\le 2\ka_i+1$ for $1\le i\le d$;
\item[-] $a_1+\cdots +a_d=\el$. 
\end{itemize}
For $\ul{a}\in K_{\phi,\chi}^{(\el)}$, set 
\[
\ul{\ka}(\ul{a})
=
(\ka_1, \ka_1-1,\dots, \ka_1-a_1+1, \ldots, \ka_d, \ka_d-1, \ldots, \ka_d-a_d+1)
\in
\Omega_\el. 
\]
For $\ul{a}, \ul{b}\in K_{\phi,\chi}^{(\el)}$, we also set 
\[
m_{\ul{a}, \ul{b}}
=
\dim_\bbC\Jac_{\chi|\cdot|^{\ul{\ka}(\ul{a})}}\(\d_{\chi}(\ul{\ka}(\ul{b}))\),
\]
and define $\(m'_{\ul{a}, \ul{b}}\)_{\ul{a}, \ul{b}\in K_{\phi, \chi}^{(\el)}}$ to be the inverse matrix of 
$\(m_{\ul{a}, \ul{b}}\)_{\ul{a}, \ul{b}\in K_{\phi, \chi}^{(\el)}}$, i.e., 
\[
\sum_{\ul{c}\,\in K_{\phi,\chi}^{(\el)}}m'_{\ul{a},\ul{c}}m_{\ul{c}, \ul{b}}
=
\begin{cases}
1\quad&\text{if $\ul{a}=\ul{b}$},\\
0\quad&\text{if $\ul{a}\ne\ul{b}$}. 
\end{cases}
\]
Now, by \cite[Theorem 4.3]{Atobe2020}, we have
\[
\mu^*_{\chi}(\pi)
=
\sum_{\el=0}^n\,\,\sum_{\ul{a},\, \ul{b}\in K_{\phi,\chi}^{(\el)}} 
m'_{\ul{a}, \ul{b}}\cdot \d_{\chi}(\ul{\ka}(\ul{a}))\bt\Jac_{\chi|\cdot|^{\ul{\ka}(\ul{b})}}(\pi). 
\]
Moreover, \lmref{L:Atobe 2} implies that $\(m_{\ul{a}, \ul{b}}\)_{\ul{a}, \ul{b}\in K_{\phi, \chi}^{(\el)}}$
is a diagonal matrix and that $m_{\ul{a}, \ul{a}}=1$ for each $\ul{a}\in K_{\phi,\chi}^{(\el)}$. Consequently, we also have
$m'_{\ul{a}, \ul{a}}=1$ for each $\ul{a}\in K_{\phi,\chi}^{(\el)}$.  

From the observations preceding \lmref{L:key Jac}, we have to determine when $\d_\chi(\ul{\ka}(\ul{a}))$ is unramified. 
By \lmref{L:unram 2}, we have
\[
\text{$\d_{\chi}(\ul{\ka}(\ul{a}))$ is unramified}
\Longleftrightarrow 
\text{$a_i\le 1$ for $1\le i\le d$}. 
\]
Accordingly, we set 
\[
K^{ur}_{\phi,\chi}
=
\stt{\ul{a}=(a_1,\ldots, a_d)\mid\text{$a_i\in\stt{0,1}$ for each $1\le i\le d$}}. 
\]
Note that 
\[
K^{ur}_{\phi,\chi}\cap K_{\phi, \chi}^{(\el)}\ne\emptyset
\Longleftrightarrow
\el\le d,
\]
in which case
\[
|K^{ur}_{\phi,\chi}\cap K_{\phi, \chi}^{(\el)}|=\begin{pmatrix}d\\\el\end{pmatrix}. 
\]
We proceed to compute
\[
\pi_{\ul{a}}:=\Jac_{\chi|\cdot|^{\ul{\ka}(\ul{a})}}(\pi)
\]
when $\ul{a}\in K^{ur}_{\phi,\chi}\cap K_{\phi, \chi}^{(\el)}$ with $\el\le d$. By applying \cite[Theorem 4.2 (2)]{Atobe2020} repeatedly, 
we find that $\pi_{\ul{a}}$ is the irreducible generic square-integrable representation of $\SO_{2(n-\el)+1}(F)$,
whose $L$-parameter $\phi_{\ul{a}}$ is given by 
\[
\phi_{\ul{a}}
=
\phi
-
\(\bigoplus_{\substack{1\le i\le d\\a_i=1}}\phi_\chi\bt S_{2\ka_i+1}\)
+
\(\bigoplus_{\substack{1\le i\le d\\a_i=1}}\phi_\chi\bt S_{2\ka_i-1}\). 
\]
To summarize, we have shown that 
\[
\sum_{\el=0}^d\,\,\sum_{\ul{a}\in K^{ur}_{\phi,\chi}\cap K^{(\el)}_{\phi,\chi}}
\d_\chi(\ul{\ka}(\ul{a}))\bt\pi_{\ul{a}}
\le
\mu^*_{\chi}(\pi),
\]
where $\d_\chi(\ul{\ka}(\ul{a}))$ (resp. $\pi_{\ul{a}}$) is unramified (resp. generic). 

Our next step is to analyze
\[
(m\ot{\rm id})\circ({\rm id}\bt\mu^*_{\chi'})\(\d_\chi(\ul{\ka}(\ul{a}))\bt\pi_{\ul{a}}\)
=
(m\ot{\rm id})\(\d_\chi(\ul{\ka}(\ul{a}))\bt\mu^*_{\chi'}(\pi_{\ul{a}})\)
\]
for $\ul{a}\in K^{ur}_{\phi,\chi}\cap K_{\phi, \chi}^{(\el)}$ with $\el\le d$. By applying the above computations to each $\pi_{\ul{a}}$, 
we obtain
\[
\sum_{\el'=0}^{d'}\,\,\sum_{\ul{a}'\in K^{ur}_{\phi_{\ul{a}},\chi'}\cap K^{(\el')}_{\phi_{\ul{a}},\chi'}}
\d_{\chi'}(\ul{\ka}'(\ul{a}'))\bt\Jac_{\chi'|\cdot|^{\ul{\ka}'(\ul{a}')}}(\pi_{\ul{a}})
\le
\mu^*_{\chi'}(\pi_{\ul{a}}),
\]
where $\d_{\chi'}(\ul{\ka}'(\ul{a}'))$ is unramified, and 
$
\Jac_{\chi'|\cdot|^{\ul{\ka}'(\ul{a}')}}(\pi_{\ul{a}})
$
is an irreducible generic square-integrable representation of $\SO_{2(n-\el-\el')+1}(F)$ with $L$-parameter 
\[
\phi_{\ul{a}}
-
\(\bigoplus_{\substack{1\le i\le d'\\a'_i=1}}\phi_{\chi'}\bt S_{2\ka'_i+1}\)
+
\(\bigoplus_{\substack{1\le i\le d'\\a'_i=1}}\phi_{\chi'}\bt S_{2\ka'_i-1}\). 
\]
Observe that 
\[
K^{ur}_{\phi_{\ul{a}},\chi'}\cap K^{(\el')}_{\phi_{\ul{a}},\chi'}
=
K^{ur}_{\phi,\chi'}\cap K_{\phi,\chi'}^{(\el')}
\]
for every $\ul{a}\in K^{ur}_{\phi,\chi}\cap K_{\phi, \chi}^{(\el)}$ with $0\le \el\le d$ and every $0\le \el'\le d'$.  Together, we conclude that 
\[
\(m\ot{\rm id}\)\circ\({\rm id}\ot\mu^*_{\chi'}\)\circ\mu^*_{\chi}(\pi)
\]
contains the summands 
\[
\(\d_\chi(\ul{\ka}(\ul{a}))\,\x\,\d_{\chi'}(\ul{\ka}'(\ul{a}'))\)\bt\Jac_{\chi'|\cdot|^{\ul{\ka}'(\ul{a}')}}(\pi_{\ul{a}})
\]
for each $\ul{a}\in K_{\phi, \chi}\cap K_{\phi, \chi}^{(\el)}$ and $\ul{a}'\in K_{\phi, \chi'}\cap K_{\phi, \chi'}^{(\el')}$
with $0\le \el\le d$ and $0\le \el'\le d'$. Since
\[
\text{
$\d_\chi(\ul{\ka}(\ul{a}))\,\x\,\d_{\chi'}(\ul{\ka}'(\ul{a}'))$ (resp. $\Jac_{\chi'|\cdot|^{\ul{\ka}'(\ul{a}')}}(\pi_{\ul{a}})$)}
\]
is unramified (resp. generic), it follows that $\mu^*_{ur}(\pi)$ contains the unramified constituent (necessarily with multiplicity one)
of $\d_\chi(\ul{\ka}(\ul{a}))\,\x\,\d_{\chi'}(\ul{\ka}'(\ul{a}'))$ for each $\ul{a}\in K^{ur}_{\phi, \chi}\cap K_{\phi, \chi}^{(\el)}$ and 
$\ul{a}'\in K^{ur}_{\phi, \chi'}\cap K_{\phi, \chi'}^{(\el')}$ with $0\le \el\le d$ and $0\le \el'\le d'$. Consequently, 
\begin{align*}
 |\mu^*_{ur}(\pi)|
 \ge
 \sum_{\el=0}^d\,\sum_{\el'=0}^{d'} 
 |K^{ur}_{\phi, \chi}\cap K_{\phi, \chi}^{(\el)}|\cdot |K^{ur}_{\phi, \chi'}\cap K_{\phi, \chi'}^{(\el')}|
 =
 \sum_{\el=0}^d\,\sum_{\el'=0}^{d'} 
 \begin{pmatrix}d\\\el\end{pmatrix}\cdot\begin{pmatrix}d'\\\el'\end{pmatrix}
 =
 2^{d+d'} 
 =
 2^{|I_{\phi,\chi}|+|I_{\phi,\chi'}|}. 
\end{align*}
This proves the case where both $I_{\phi,\chi}\setminus I_{\phi_0,\chi}$ and $I_{\phi,\chi'}\setminus I_{\phi_0,\chi'}$ are non-empty. 

Assume that exactly one of $I_{\phi,\chi}\setminus I_{\phi_0,\chi}$ or $I_{\phi,\chi'}\setminus I_{\phi_0,\chi'}$  is empty, say 
$I_{\phi, \chi'}\setminus I_{\phi_0,\chi'}=\emptyset$, and 
\[
I_{\phi, \chi}\setminus I_{\phi_0,\chi}
=
\stt{0<\ka_1<\ka_2\cdot\cdot\cdot<\ka_{2a}}
\]
for some $a\in\bbN$. Then 
\[
\Pi
=
\D_{\chi}[\ka_{2}, -\ka_{1}]\x\cdot\cdot\cdot\x\D_{\chi}[\ka_{2a}, -\ka_{2a-1}]\rtimes\pi_0,
\]
and a similar argument shows that
\[
|\mu^*_{ur}(\pi)|
\le
|\mu^*_{ur}(\Pi)|
=
2^{|I_{\phi,\chi}|}. 
\]
On the other hand, since
\[
\(m\ot{\rm id}\)\circ\({\rm id}\ot\mu^*_{\chi'}\)\circ\mu^*_{\chi}(\pi)
=
\mu^*_{\chi}(\pi),
\]
the above computations imply that $\mu^*_{ur}(\pi)$ contains the unramified constituent (necessarily with multiplicity one) of 
$\d_\chi(\ul{\ka}(\ul{a}))$ for every $\ul{a}\in K^{ur}_{\phi,\chi}\cap K_{\phi,\chi}^{(\el)}$ with $0\le\el\le d$, where the sets $K^{ur}_{\phi,\chi}$ 
and $K_{\phi,\chi}^{(\el)}$ are defined as above, and $d=|I_{\phi,\chi}|$. It follows that 
\[
|\mu^*_{ur}(\pi)|
\ge 
\sum_{\el=0}^d|K^{ur}_{\phi,\chi}\cap K_{\phi,\chi}^{(\el)}|
=
\sum_{\el=0}^d\begin{pmatrix}d\\\el\end{pmatrix}
=
2^d
=
2^{|I_{\phi,\chi}|}. 
\]
This finishes the proof of \propref{P:key Jac}, and hence the proof of \propref{P:seed to sq}.  \qed\\

Although we will not need the following corollary, which follows from \lmref{L:key Jac} and the proof of \propref{P:key Jac}, we record it 
here, as it might be useful elsewhere.  

\begin{cor}\label{C:from key Jac}
Let $\pi$ be an irreducible generic square-integrable representation of $\SO_{2n+1}(F)$ with $n>0$, and let $\phi$ be the associated 
$L$-parameter, written as in \eqref{E:L-par}. Then 
\[
|\mu^*_{ur}(\pi)|=2^{|I_{\phi,\chi}|+|I_{\phi,\chi'}|}. 
\]
\end{cor}

\begin{remark}
Instead of proving \corref{C:from key Jac} first when $\pi$ is a seed representation and then proving it for general $\pi$, as we have done here, 
one may try to verify \corref{C:from key Jac} directly using the results of Atobe, i.e., without using the information from 
$|\mu^*_{ur}(\Pi)|$. Our proof of \propref{P:key Jac} shows that 
\[
|\mu^*_{ur}(\pi)|
\ge
2^{|I_{\phi,\chi}|+|I_{\phi,\chi'}|}. 
\]
To establish the opposite inequality, one also needs to investigate 
\[
\Jac_{\chi|\cdot|^{\ul{\ka}(\ul{b})}}(\pi)
\]
for $\ul{b}\in K_{\phi, \chi}^{(\el)}\setminus K_{\phi,\chi}$. We do not know whether this investigation is complicated; however, 
since we need to know the number $|\mu^*_{ur}(\Pi)|$, the present argument seems more suitable for us. 
\end{remark}

\section{Preliminaries for the final reduction}\label{S:P for 3rd red}
\subsection{Level-raising operators}\label{SS:level-raising operator}
The goal of the next section is to further reduce the proof of the existence part of the newform conjecture from seed 
representations to the case of generic supercuspidal representations. To achieve this goal, we make use of the level-raising operators 
\[
\theta_m,\,\theta'_m:\Pi^{K_{n,m}}\longto\Pi^{K_{n,m+1}}
\]
defined by 
\[
\theta_m(v)
=
\vol\(K_{n,m}\cap K_{n,m+1},dk\)^{-1}\int_{K_{n,m+1}}\Pi(k)v\,dk
\]
and
\[
\theta'_m=\Pi(w_{\ep_n, m+1})\circ\theta_m\circ\Pi(w_{\ep_n,m})
\]
for $v\in\Pi^{K_{n,m}}$, following \cite{RobertsSchmidt2007}. Here $\Pi\in\Rep(\SO_{2n+1}(F))$ with $n\ge 1$, and $w_{\ep_1, m}$, 
$w_{\ep_n, m+1}$ are the Weyl elements defined in \cite[\S 2.5]{YCheng2025}. We note that our convention for $\theta_m, \theta'_m$ 
is different from that of \cite{RobertsSchmidt2007}. Furthermore, we have $w_{\ep_n,m}\in J_{n,m}\setminus K_{n,m}$ for $m>0$, and
$w_{\ep_n, 0}\in J_{n,0}=K_{n,0}$. 

Since 
\begin{equation}\label{E:theta sum}
\theta_m(v)=\sum_{k\in K_{n,m+1}/\(K_{n,m}\cap K_{n,m+1}\)}\Pi(k)v
\end{equation}
and 
\begin{equation}\label{E:theta' sum}
\theta'_{m}(v)
=
\sum_{k\in K_{n,m+1}/\(K_{n,m}\cap K_{n,m+1}\)}\Pi(w_{\ep_n,m+1}\,k\,w_{\ep_n,m})v
\end{equation}
for $v\in\Pi^{K_{n,m}}$, we need to investigate the following left cosets:
\[
K_{n,m+1}/\(K_{n,m}\cap K_{n,m+1}\).  
\]
This is the aim of the next subsection. 

\subsection{Coset decompositions}
In this and the next section, we follow the notation and conventions in \cite[\S 2 and \S 3]{YCheng2025}. In particular, we regard the split
group $\SO_{2n}(F)$ as a subgroup of $\SO_{2n+1}(F)$ via the embedding \eqref{E:embedding}, and for each $m\ge 0$
we have
\[
H_{n,m}=K_{n,m}\cap\SO_{2n}(F). 
\]
To describe and prove the result in this section, we also introduce some additional notation. 
Let $S\subset \cI_n:=\stt{1,2,\ldots, n}$ be a (possibly empty) subset. Define 
\begin{equation}\label{E:w_S}
w_{S,m}=\prod_{j\in S}w_{\ep_j,m}. 
\end{equation}
Note that if $m>0$, then $w_{S,m}\in H_{n,m}$ if and only if $|S|$ is even. We also define
\begin{equation}\label{E:I_S}
I_S
=\stt{\ep_i+\ep_j\mid 1\le i<j\le n,\, i\nin S}. 
\end{equation}
Now we can state the main result of this subsection. 

\begin{prop}\label{P:coset decomp}
Let the notation be as above. Then 
\[
\stt{
w_{S, m+1}\prod_{\beta\in I_S}x_{\beta}\(\varpi^{-m-1}y_{S,\beta}\)
\mid
\text{$|S|$ is even and $y_{S,\beta}\in\frak{o}/\frak{p}$ for each $S$ and $\beta$}
}
\]
forms a set of representatives of 
\[
K_{n,m+1}/\(K_{n,m}\cap K_{n,m+1}\)
\]
for each $m\ge 0$. 
\end{prop}

We need the following lemma. For a given $X\in\M_n(F)$ satisfying ${}^tX=-J_nXJ_n$, denote
\[
u(X)
=
\begin{pmatrix}
I_n&&X\\
&1&\\
&&I_n
\end{pmatrix}
\in\SO_{2n}(F)\hookto\SO_{2n+1}(F). 
\]
Then 
\[
u(X)=\prod_{1\le i<j\le n} x_{\ep_i+\ep_j}(x_{i,j})
\]
for some $x_{i,j}\in F$, and we say that $\ep_k+\ep_h$ occurs in $u(X)$ whenever $x_{k,h}\ne 0$. 

\begin{lm}\label{L:u(X)}
Let $X\in\M_n(\frak{o})$ with ${}^tX=-J_nXJ_n$ and let $y\in\frak{o}$. Suppose that $\ep_k+\ep_h$ does not occur in $u(X)$. Then
\begin{align*}
x_{-\ep_k-\ep_h}&\(\varpi^{m+1}y\)u\(\varpi^{-m-1}X\)x_{-\ep_k-\ep_h}\(-\varpi^{m+1}y\)\cdot\(H_{n,m}\cap H_{n,m+1}\)\\
&=
u\(\varpi^{-m-1}(X-XYX)\)\cdot\(H_{n,m}\cap H_{n,m+1}\),
\end{align*}
where
\[
Y=y\(E^n_{n+1-h, k}-E^n_{n+1-k, h}\). 
\]
\end{lm}

\begin{proof}
We have 
\[
u\(\varpi^{-m-1}X\)
=
\begin{pmatrix}
I_n&&\varpi^{-m-1}X\\
&1&\\
&&I_n
\end{pmatrix}
\quad\text{and}\quad
x_{-\ep_k-\ep_h}\(\varpi^{m+1}y\)
=
\begin{pmatrix}
I_n&&\\
&1&\\
\varpi^{m+1}Y&&I_n
\end{pmatrix}. 
\]
A direct computation shows that 
\begin{align*}
x_{-\ep_k-\ep_h}\(\varpi^{m+1}y\)u\(\varpi^{-m-1}X\)x_{-\ep_k-\ep_h}\(-\varpi^{m+1}y\)
=
\begin{pmatrix}
I_n-XY&&\varpi^{-m-1}X\\
&1&\\
-\varpi^{m+1}YXY&&I_n+YX
\end{pmatrix}. 
\end{align*}
We first claim that 
\[
YXY=0. 
\]
To this end, note that 
\[
X
=
\sum_{1\le i<j\le n} x_{i,j}\(E^n_{i, n+1-j}-E^n_{j, n+1-i}\)
\]
for some $x_{i,j}\in\frak{o}$. The assumption that $\ep_k+\ep_h$ does not occur in $u(X)$ then implies that $x_{k,h}=0$.
Now, if
\[
YXY=y^2\(E^n_{n+1-h, k}-E^n_{n+1-k, h}\)X\(E^n_{n+1-h, k}-E^n_{n+1-k, h}\)\ne 0, 
\]
then we would obtain $x_{k,h}\ne 0$, a contradiction; hence the claim holds. 

Note that both $u\(\varpi^{-m-1}X\)$ and $x_{-\ep_k-\ep_h}\(\varpi^{m+1}y\)$ are contained in $H_{n,m+1}$. The claim then implies 
that $(I_n-XY)\in\GL_n(\frak{o})$. It follows that 
\[
\diag{I_n-XY, 1, I_n+YX}\in H_{n,m}\cap H_{n,m+1}, 
\]
so that 
\begin{align*}
x_{-\ep_k-\ep_h}\(\varpi^{m+1}y\)u&\(\varpi^{-m-1}X\)x_{-\ep_k-\ep_h}\(-\varpi^{m+1}y\)\cdot\(H_{n,m}\cap H_{n,m+1}\)\\
&=
\begin{pmatrix}
I_n&&\varpi^{-m-1}X(I_n+YX)^{-1}\\
&1&\\
&&I_n
\end{pmatrix}
\cdot
\(H_{n,m}\cap H_{n,m+1}\). 
\end{align*}
To proceed, we compute 
\[
(I_n+YX)^{-1}. 
\]
The fact that 
\[
\diag{I_n-XY, 1, I_n+YX}\in\SO_{2n}(F) 
\]
implies that
\[
(I_n+YX)
=
J_n{}^t\(I_n-XY\)^{-1} J_n
=
J_n\(I_n-{}^tY{}^tX\)^{-1}J_n. 
\]
Since 
\[
{}^tX=-J_nXJ_n
\quad\text{and}\quad
{}^tY=-J_nYJ_n,
\]
we get 
\[
(I_n+YX)^{-1}=I_n-J_n{}^tY{}^tXJ_n=I_n-YX. 
\]
Together, we obtain
\begin{align*}
x_{-\ep_k-\ep_h}&\(\varpi^{m+1}y\)u\(\varpi^{-m-1}X\)x_{-\ep_k-\ep_h}\(-\varpi^{m+1}y\)\cdot\(H_{n,m}\cap H_{n,m+1}\)\\
&=
u\(\varpi^{-m-1}(X-XYX)\)\cdot\(H_{n,m}\cap H_{n,m+1}\). 
\end{align*}
This finishes the proof. 
\end{proof}

\subsection{Proof of \propref{P:coset decomp}}
Since $m+1>0$, we have the decomposition
\[
K_{n,m+1}
=
H_{n,m+1}
\prod_{i=1}^n x_{\ep_i}(\frak{o})
\prod_{j=1}^n x_{-\ep_j}\(\frak{p}^{m+1}\)
\]
by \cite[Lemma 3.1]{YCheng2025}. Furthermore, since both $x_{\ep_i}(\frak{o})$ and $x_{-\ep_j}\(\frak{p}^{m+1}\)$ are contained in 
$K_{n,m}$ for $1\le i,j\le n$, the natural map
\[
H_{n,m+1}/\(H_{n,m}\cap H_{n,m+1}\)\longto K_{n,m+1}/\(K_{n,m}\cap K_{n,m+1}\)
\]
is a bijection. Thus, it suffices to find a set of representatives of 
\[
H_{n,m+1}/\(H_{n,m}\cap H_{n,m+1}\). 
\]
This will be done in three steps. 

\subsubsection*{\underline{Step 1}}
We first show that the set 
\[
A
=
\stt{w_{S,m+1}\prod_{1\le i<j\le n} x_{\ep_i+\ep_j}\(\varpi^{-m-1}y_{i,j}\)
\mid 
\text{$|S|$ is even and $y_{i,j}\in\frak{o}/\frak{p}$ for every $1\le i<j\le n$}
}
\]
contains a set of representatives of $H_{n,m+1}/\(H_{n,m}\cap H_{n,m+1}\)$. To this end, recall that $H_{n,0}$ and $H_{n,1}$ are two 
non-conjugate hyperspecial maximal compact subgroups of $\SO_{2n}(F)$, and 
\[
H_{n,m}=\varpi^{-\lfloor\frac{m}{2}\rfloor\lambda_n} H_{n,e} \varpi^{\lfloor\frac{m}{2}\rfloor\lambda_n}
\]
where $e\in\stt{0,1}$ satisfies $m\equiv e\pmod{2}$ (see \cite[\S 3.3]{YCheng2025}). In particular, each $H_{n,m}$ is a parahoric 
subgroup of $\SO_{2n}(F)$, and hence we have the following decomposition:
\begin{align*}
H_{n,m+1}
=
\(N_{\SO_{2n}(F)}(T_n(F))\cap H_{n,m+1}\)
\cdot
\(U^0_n(F)\cap H_{n,m+1}\)
\cdot
\(\bar{U}^0_n(F)\cap H_{n,m+1}\)
\end{align*}
by \cite[Proposition 6.4.9 (iii)]{BruhatTits1972} (see also \cite[Proposition 7.3.12 (1)]{KalethaPrasad2023}), 
where $\bar{U}_n$ denotes the opposite of $U_n$, and 
\[
U^0_n(F)=U_n(F)\cap\SO_{2n}(F)
\quad\text{and}\quad
\bar{U}^0_n(F)=\bar{U}_n(F)\cap\SO_{2n}(F). 
\]
We have 
\[
U^0_n(F)\cap H_{n,m+1}
=
\prod_{1\le i<j\le n}x_{\ep_i+\ep_j}\(\frak{p}^{-m-1}\)\prod_{1\le i<j\le n} x_{\ep_i-\ep_j}(\frak{o})
\] 
and since both
\[
\bar{U}^0_n(F)\cap H_{n,m+1}
\quad\text{and}\quad
\prod_{1\le i<j\le n} x_{\ep_i-\ep_j}(\frak{o}) 
\]
are contained in $H_{n,m}\cap H_{n,m+1}$, it follows that 
\[
\(N_{\SO_{2n}(F)}(T_n(F))\cap H_{n,m+1}\)
\prod_{1\le i<j\le n}x_{\ep_i+\ep_j}\(\frak{p}^{-m-1}\)
\]
contains a set of representatives of $H_{n,m+1}/\(H_{n,m}\cap H_{n,m+1}\)$. 

To proceed, note that 
\[
H_{n,m+1}\cap T_n(F)=T_n(\frak{o})\subset H_{n,m}\cap H_{n,m+1},
\]
and the set
\[
\stt{w_{S,m+1}\hat{w}_\sigma\mid\text{$\sigma\in\frak{S}_n$ and $|S|$ is even}}
\]
forms a set of representatives of
\[
\(N_{\SO_{2n}(F)}(T_n(F))\cap H_{n,m+1}\)/T_n(\frak{o})
\simeq
N_{\SO_{2n}(F)}(T_n(F))/T_n(F)
\simeq
\frak{S}_n\ltimes\stt{\pm 1}^{n-1}. 
\]
Here $\frak{S}_n$ is the symmetric group on $\cI_n$, $w_\sigma\in\GL_n(\frak{o})$ is the permutation matrix associated with 
$\sigma\in\frak{S}_n$, and $\hat{w}_\sigma=\diag{w_\sigma, 1, w^*_\sigma}\in\SO_{2n}(F)$. Since 
\[
T_n(\frak{o})\prod_{1\le i<j\le n}x_{\ep_i+\ep_j}\(\frak{p}^{-m-1}\)
=
\prod_{1\le i<j\le n}x_{\ep_i+\ep_j}\(\frak{p}^{-m-1}\)T_n(\frak{o}),
\]
and 
\[
\hat{w}_\sigma\prod_{1\le i<j\le n}x_{\ep_i+\ep_j}\(\frak{p}^{-m-1}\)
=
\prod_{1\le i<j\le n}x_{\ep_i+\ep_j}\(\frak{p}^{-m-1}\)\hat{w}_\sigma
\]
for each $\sigma\in\frak{S}_n$, we further deduce that 
\[
\bigsqcup_{\substack{S\subset\cI_n\\\text{$|S|$ is even}}}
w_{S,m+1}\prod_{1\le i<j\le n}x_{\ep_i+\ep_j}\(\frak{p}^{-m-1}\)
\]
contains a set of representatives of $H_{n,m+1}/\(H_{n,m}\cap H_{n,m+1}\)$. Finally, since 
\[
\prod_{1\le i<j\le n}x_{\ep_i+\ep_j}\(\frak{p}^{-m}\)\subset H_{n,m}\cap H_{n,m+1}, 
\]
Step 1 is now complete. 

\subsubsection*{\underline{Step 2}}
In Step 2, we verify that 
\[
B
=
\stt{
w_{S, m+1}\prod_{\beta\in I_S}x_{\beta}\(\varpi^{-m-1}y_{S,\beta}\)
\mid
\text{$|S|$ is even and $y_{S,\beta}\in\frak{o}/\frak{p}$ for each $S$ and $\beta$}
}
\]
contains a set of representatives of $H_{n,m+1}/\(H_{n,m}\cap H_{n,m+1}\)$. This is done by proving that each element of $A$ is 
congruent to an element of $B$ modulo $H_{n,m}\cap H_{n,m+1}$. We prove this by induction on $|S|$. When $|S|=0$, so that 
$S=\emptyset$, we have 
\[
I_S=\stt{\ep_i+\ep_j\mid 1\le i<j\le n},
\]
and the assertion is clear. Assume that the assertion holds for elements with $|S|=2(k-1)\ge 0$.  Let $S\subseteq\cI_n$ be a subset 
with $|S|=2k$, and let
\[
w_{S,m+1}\prod_{1\le i< j\le n}x_{\ep_i+\ep_j}\(\varpi^{-m-1}y_{i,j}\)\in A. 
\]
We consider two cases. First, assume that there exist $k,h\in S$ with $k<h$ such that $\ep_k+\ep_h$ occurs in 
\[
\prod_{1\le i< j\le n}x_{\ep_i+\ep_j}\(\varpi^{-m-1}y_{i,j}\),
\]
i.e., we have $y:=y_{k,h}\in\frak{o}^\x$. Let 
\[
S_1=\stt{k,h}
\quad\text{and}\quad
S_2=S\setminus S_1. 
\]
Then 
\[
w_{S,m+1}=w_{S_1, m+1}\,w_{S_2, m+1}=w_{S_2,m+1}\,w_{S_1, m+1}
\]
and we may write
\begin{align}\label{E:coset 1}
\begin{split}
w_{S,m+1}\prod_{1\le i< j\le n}x_{\ep_i+\ep_j}\(\varpi^{-m-1}y_{i,j}\)
&=
w_{S_1,m+1}\,w_{S_2,m+1}\,x_{\ep_k+\ep_h}\(\varpi^{-m-1}y\)\,u\(\varpi^{-m-1}X\)\\
&=
w_{S_1,m+1}\,x_{\ep_k+\ep_h}\(\varpi^{-m-1}y\)\,w_{S_2,m+1}\,u\(\varpi^{-m-1}X\). 
\end{split}
\end{align}
for some $X\in\M_n(\frak{o})$ such that $\ep_k+\ep_h$ does not occur in $u\(\varpi^{-m-1}X\)$, where the last identity holds because $S_2$ does 
not contain $k$ or $h$. 

We now use the identity
\[
x_{\ep_k+\ep_h}\(\varpi^{-m-1}y\)
=
x_{-\ep_k-\ep_h}\(\varpi^{m+1}\,y^{-1}\)w_{S_1,m+1}\,x_{-\ep_k-\ep_h}\(\varpi^{m+1}y^{-1}\)\(-y^{-1}\)^{\ep_k^*+\ep_h^*}w_{\ep_k-\ep_h},
\]
together with
\[
w_{S_1,m+1}\,x_{-\ep_k-\ep_h}\(\varpi^{m+1}y^{-1}\)\,w_{S_1,m+1}
=
x_{\ep_k+\ep_h}\(-\varpi^{-m-1}y^{-1}\).  
\]
Since $w_{S_2,m+1}$ commutes with 
\[
w_{\ep_k-\ep_h},
\quad
\(-y^{-1}\)^{\ep_k^*+\ep_h^*},
\quad
x_{-\ep_k-\ep_h}\(\varpi^{m+1}y^{-1}\)
\quad\text{and}\quad
x_{\ep_k+\ep_h}\(-\varpi^{-m-1}y^{-1}\),
\]
it follows that \eqref{E:coset 1} becomes
\[
w_{S_2,m+1}\,x_{\ep_k+\ep_h}\(-\varpi^{-m-1}y^{-1}\)x_{-\ep_k-\ep_h}\(\varpi^{m+1}y^{-1}\)
\(-y^{-1}\)^{\ep^*_k+\ep^*_h}w_{\ep_k-\ep_h}u\(\varpi^{-m-1}X\).
\]
Now observe that
\[
\(-y^{-1}\)^{\ep^*_k+\ep^*_h}w_{\ep_k-\ep_h}\in H_{n,m}\cap H_{n,m+1}
\]
and 
\[
\(-y^{-1}\)^{\ep^*_k+\ep^*_h}w_{\ep_k-\ep_h}\,u\(\varpi^{-m-1}X\)\,w_{\ep_k-\ep_h}^{-1}\(-y\)^{\ep^*_k+\ep^*_h}
=
u\(\varpi^{-m-1}X_1\)
\]
for some $X_1\in\M_n(\frak{o})$. Moreover, $\ep_k+\ep_h$ does not occur in $u\(\varpi^{-m-1}X_1\)$. Indeed, since 
$\(-y^{-1}\)^{\ep_k^*+\ep_h^*}\in T_n(\frak{o})$, it suffices to verify the same assertion for 
\[
w_{\ep_k-\ep_h}u\(\varpi^{-m-1}X\)w_{\ep_k-\ep_h}^{-1}. 
\]
Because $\ep_i+\ep_j$ does not occur in $u\(\varpi^{-m-1}X\)$ for $(i,j)=(k,h)$ by our assumption, and since conjugation by 
$w_{\ep_k-\ep_h}$ interchanges the roots $\ep_k$ and $\ep_h$, while leaving all other roots unchanged, the claim follows. Hence
\begin{align*}
w_{S,m+1}&\prod_{1\le i< j\le n}x_{\ep_i+\ep_j}\(\varpi^{-m-1}y_{i,j}\)\cdot\(H_{n,m}\cap H_{n,m+1}\)\\
&=
w_{S_2,m+1}\,x_{\ep_k+\ep_h}\(-\varpi^{-m-1}y^{-1}\)x_{-\ep_k-\ep_h}\(\varpi^{m+1}y^{-1}\)
u\(\varpi^{-m-1}X_1\)\cdot\(H_{n,m}\cap H_{n,m+1}\)
\end{align*}
Since
\[
x_{-\ep_k-\ep_h}\(\varpi^{m+1}y^{-1}\)\in H_{n,m}\cap H_{n,m+1},
\]
and $\ep_k+\ep_h$ does not occur in $u\(\varpi^{-m-1}X_1\)$, we obtain
\begin{align*}
x_{-\ep_k-\ep_h}&\(\varpi^{m+1}y^{-1}\)u\(\varpi^{-m-1}X_1\)x_{-\ep_k-\ep_h}\(-\varpi^{m+1}y^{-1}\)\cdot\(H_{n,m}\cap H_{n,m+1}\)\\
&=
u\(\varpi^{-m-1}X_2\)\cdot\(H_{n,m}\cap H_{n,m+1}\)
\end{align*}
for some $X_2\in\M_n(\frak{o})$ by \lmref{L:u(X)}. Thus
\begin{align*}
w_{S,m+1}&\prod_{1\le i< j\le n}x_{\ep_i+\ep_j}\(\varpi^{-m-1}y_{i,j}\)\cdot\(H_{n,m}\cap H_{n,m+1}\)\\
&=
w_{S_2,m+1}\,x_{\ep_k+\ep_h}\(-\varpi^{-m-1}y^{-1}\)u\(\varpi^{-m-1}X_2\)\cdot\(H_{n,m}\cap H_{n,m+1}\)\\
&=
w_{S_2,m+1}\,u\(\varpi^{-m-1}X_3\)\cdot\(H_{n,m}\cap H_{n,m+1}\),
\end{align*}
for some $X_3\in\M_n(\frak{o})$. Since $|S_2|=2(k-1)$, the induction hypothesis implies that 
\begin{align*}
w_{S,m+1}&\prod_{1\le i< j\le n}x_{\ep_i+\ep_j}\(\varpi^{-m-1}y_{i,j}\)\cdot\(H_{n,m}\cap H_{n,m+1}\)\\
&=
w_{S_2,m+1}\,u\(\varpi^{-m-1}X_3\)\cdot\(H_{n,m}\cap H_{n,m+1}\)\\
&=
\alpha\(H_{n,m}\cap H_{n,m+1}\)
\end{align*}
for some $\alpha\in B$. This verifies the first case.  

Next, we assume that for every $k, h\in S$ with $k<h$, $\ep_k+\ep_h$ does not occur in 
\[
\prod_{1\le i< j\le n}x_{\ep_i+\ep_j}\(\varpi^{-m-1}y_{i,j}\). 
\]
Note that if $\ep_k+\ep_h$ also does not occur in the above product for every $k\in S$ and every $h$ wtih $k<h\le n$ and $h\nin S$, then 
\[
w_{S,m+1}\prod_{1\le i< j\le n}x_{\ep_i+\ep_j}\(\varpi^{-m-1}y_{i,j}\)
\]
is itself an element of $B$, and there is nothing to prove. Thus, we may assume that there exist $k\in S$ and $h\nin S$ such that 
$\ep_k+\ep_h$ occurs in the product. In fact, we may further assume that $k$ is the smallest such integer. This assumption implies that 
$\ep_i+\ep_j$ does not occur in the product for any $i\in S$ with $1\le i<k$ and any $j$ with $i<j\le n$. 

Let $\el\in S$ such that $\el\ne k$, and put
\[
S_1=\stt{k, \el}\quad\text{and}\quad S_2=S\setminus S_1,
\]
as in the first case. Note that $S_2\cap\stt{k,h}=\emptyset$. We may write 
\begin{align}\label{E:coset 2}
\begin{split}
w_{S,m+1}\prod_{1\le i< j\le n}x_{\ep_i+\ep_j}\(\varpi^{-m-1}y_{i,j}\)
&=
w_{S_1,m+1}\,w_{S_2,m+1}\,x_{\ep_k+\ep_h}\(\varpi^{-m-1}y\)\,u\(\varpi^{-m-1}X\)\\
&=
w_{S_1,m+1}\,x_{\ep_k+\ep_h}\(\varpi^{-m-1}y\)\,w_{S_2,m+1}\,u\(\varpi^{-m-1}X\)
\end{split}
\end{align}
for some $X\in\M_n(\frak{o})$ such that $\ep_k+\ep_h$ does not occur in $u\(\varpi^{-m-1}X\)$, where $y=y_{k,h}\in\frak{o}^\x$. Similar to 
the first case, \eqref{E:coset 2} can be further written as
\begin{equation}\label{E:coset 3}
w_{S_1,m+1}\,x_{-\ep_k-\ep_h}\(\varpi^{m+1}y^{-1}\)\,w_{\stt{k,h},m+1}\,w_{S_2,m+1}\,x_{-\ep_k-\ep_h}\(\varpi^{m+1}y^{-1}\)
\(-y^{-1}\)^{\ep_k^*+\ep_h^*}\,w_{\ep_k-\ep_h}\,u\(\varpi^{-m-1}X\). 
\end{equation}
We now apply the identity
\[
w_{\stt{k,h},m+1}\,x_{-\ep_k-\ep_h}\(\varpi^{m+1}y^{-1}\)\,w_{\stt{k,h},m+1}
=
x_{\ep_k+\ep_h}\(-\varpi^{-m-1}y^{-1}\)
\]
together with 
\[
w_{S_1,m+1}\,w_{\stt{k,h},m+1}=w_{\stt{h,\el}, m+1}. 
\]
Since $w_{S_2,m+1}$ commutes with $x_{\ep_k+\ep_h}\(-\varpi^{-m-1}y^{-1}\)$, and 
\[
w_{\stt{h,\el},m+1}\,w_{S_2,m+1}=w_{S',m+1},
\]
where
\[
S'=S_2\cup\stt{h,\el},
\]
it follows that \eqref{E:coset 3} becomes
\[
w_{S',m+1}\,x_{\ep_k+\ep_h}\(-\varpi^{-m-1}y^{-1}\)\,x_{-\ep_k-\ep_h}\(\varpi^{m+1}y^{-1}\)
\(-y^{-1}\)^{\ep_k^*+\ep_h^*}\,w_{\ep_k-\ep_h}\,u\(\varpi^{-m-1}X\). 
\]
Now observe that
\[
\(-y^{-1}\)^{\ep^*_k+\ep^*_h}w_{\ep_k-\ep_h}\in H_{n,m}\cap H_{n,m+1}
\]
and 
\[
\(-y^{-1}\)^{\ep^*_k+\ep^*_h}w_{\ep_k-\ep_h}\,u\(\varpi^{-m-1}X\)\,w_{\ep_k-\ep_h}^{-1}\(-y\)^{\ep^*_k+\ep^*_h}
=
u\(\varpi^{-m-1}X_1\)
\]
for some $X_1\in\M_n(\frak{o})$. Furthermore, $\ep_k+\ep_h$ does not occur in $u\(\varpi^{-m-1}X_1\)$, as above. Hence
\begin{align*}
w_{S,m+1}&\prod_{1\le i< j\le n}x_{\ep_i+\ep_j}\(\varpi^{-m-1}y_{i,j}\)\cdot\(H_{n,m}\cap H_{n,m+1}\)\\
&=
w_{S',m+1}\,x_{\ep_k+\ep_h}\(-\varpi^{-m-1}y^{-1}\)x_{-\ep_k-\ep_h}\(\varpi^{m+1}y^{-1}\)
u\(\varpi^{-m-1}X_1\)\cdot\(H_{n,m}\cap H_{n,m+1}\). 
\end{align*}
Since
\[
x_{-\ep_k-\ep_h}\(\varpi^{m+1}y^{-1}\)\in H_{n,m}\cap H_{n,m+1},
\]
and $\ep_k+\ep_h$ does not occur in $u\(\varpi^{-m-1}X_1\)$, we obtain
\begin{align*}
x_{-\ep_k-\ep_h}&\(\varpi^{m+1}y^{-1}\)u\(\varpi^{-m-1}X_1\)x_{-\ep_k-\ep_h}\(-\varpi^{m+1}y^{-1}\)\cdot\(H_{n,m}\cap H_{n,m+1}\)\\
&=
u\(\varpi^{-m-1}X_2\)\cdot\(H_{n,m}\cap H_{n,m+1}\)
\end{align*}
for some $X_2\in\M_n(\frak{o})$ by \lmref{L:u(X)}. Thus
\begin{align*}
w_{S,m+1}&\prod_{1\le i< j\le n}x_{\ep_i+\ep_j}\(\varpi^{-m-1}y_{i,j}\)\cdot\(H_{n,m}\cap H_{n,m+1}\)\\
&=
w_{S',m+1}\,x_{\ep_k+\ep_h}\(-\varpi^{-m-1}y^{-1}\)u\(\varpi^{-m-1}X_2\)\cdot\(H_{n,m}\cap H_{n,m+1}\). 
\end{align*}

At this point, the arguments are parallel to those in the first case; however, in contrast to the first case, $|S'|$ remains unchanged.
To simplify further, we claim that $\ep_i+\ep_j$ does not occur in both 
\[
u\(\varpi^{-m-1}X_1\)
\quad\text{and}\quad
u\(\varpi^{-m-1}X_2\)
\] 
for every $i\in S$ with $1\le i<k$ and $i<j\le n$. Recall that  $\ep_i+\ep_j$ does not occur in $u\(\varpi^{-m-1}X\)$ for every $i\in S$ with 
$1\le i<k$ and $i<j\le n$, by our assumption on $k$. In particular, the argument showing that $\ep_k+\ep_h$ does not occur in 
$u\(\varpi^{-m-1}X_1\)$ applies here to prove the claim for $u\(\varpi^{-m-1}X_1\)$.  We now verify the claim for $u\(\varpi^{-m-1}X_2\)$. By 
\lmref{L:u(X)}, we have
\[
X_2=X_1-X_1YX_1,
\]
where
\[
Y=y^{-1}\(E^n_{n+1-h, k}-E^n_{n+1-k, h}\). 
\]
Thus, it suffices to prove the assertion for 
\[
u\(\varpi^{-m-1}X_1YX_1\). 
\]
To this end, we write 
\[
X_1
=
\sum_{1\le i<j\le n} x_{i,j}
\(E^n_{i,n+1-j}-E^n_{j,n+1-i}\)
\quad\text{and}\quad
X_1YX_1
=
\sum_{1\le i<j\le n} x'_{i,j}
\(E^n_{i,n+1-j}-E^n_{j,n+1-i}\)
\]
for some $x_{i,j}, x'_{i,j}\in\frak{o}$. Note that $x_{i,j}=0$ whenever $i\in S$, $1\le i<k$, and $i<j\le n$. On the other hand, $X_1YX_1$ is the
sum of terms of the form
\begin{equation}\label{E:coset 4}
y^{-1}x_{i,j}x_{i',j'}
\(E^n_{i,n+1-j}-E^n_{j,n+1-i}\)
\(E^n_{n+1-h,k}-E^n_{n+1-k,h}\)
\(E^n_{i',n+1-j'}-E^n_{j',n+1-i'}\)
\end{equation}
for $1\le i<j\le n$ and $1\le i'<j'\le n$. Since we are only interested in the entries $x'_{i,j}$ with $1\le i<k$, it is enough to investigate the product 
\eqref{E:coset 4} when $i<k$ or $j<k$.  However, if $j<k$, then
\[
\(E^n_{i,n+1-j}-E^n_{j,n+1-i}\)
\(E^n_{n+1-h,k}-E^n_{n+1-k,h}\)
=
0. 
\]
Therefore, $j\ge k$, and it remains to consider the case when $i<k$. If moreover $i\in S$, then $x_{i,j}=0$, and hence the claim for 
$u\(\varpi^{-m-1}X_2\)$ follows. 

Let $X_3\in\M_n(\frak{o})$ be such that
\[
x_{\ep_k+\ep_h}\(-\varpi^{-m-1}y^{-1}\)u\(\varpi^{-m-1}X_2\)
=
u\(\varpi^{-m-1}X_3\). 
\]
Then the above claim implies that $\ep_i+\ep_j$ does not occur in $u\(\varpi^{-m-1}X_3\)$ for every $i\in S$ with $1\le i<k$ and $i<j\le n$, and 
we have 
\[
w_{S,m+1}\prod_{1\le i< j\le n}x_{\ep_i+\ep_j}\(\varpi^{-m-1}y_{i,j}\)\cdot\(H_{n,m}\cap H_{n,m+1}\)
=
w_{S',m+1}u\(\varpi^{-m-1}X_3\)\cdot\(H_{n,m}\cap H_{n,m+1}\). 
\]
Now, if $w_{S',m+1}u\(\varpi^{-m-1}X_3\)\in B$ or if it reduces to the first case, then we are done. Suppose that we are still in the second case. 
Then there exists a smallest $k'\in S'$ such that $\ep_{k'}+\ep_{h'}$ occurs in $u\(\varpi^{-m-1}X_3\)$ for some $k'<h'\le n$. The key 
observation is that
\[
k'>k. 
\] 
Indeed, this follows from the fact that 
\[
S'=\(S\setminus\stt{k}\)\cup\stt{h}
\]
and that $\ep_i+\ep_j$ does not occur in $u\(\varpi^{-m-1}X_3\)$ for every $i\in S$ with $1\le i<k$ and $i<j\le n$. In particular, this implies that 
above process must terminate, i.e., 
\[
w_{S,m+1}\prod_{1\le i< j\le n}x_{\ep_i+\ep_j}\(\varpi^{-m-1}y_{i,j}\)
\]
must be congruent to an element of $B$ modulo $H_{n,m}\cap H_{n,m+1}$. This completes Step 2. 

\subsubsection*{\underline{Step 3}}
The final step is to show that two distinct elements of $B$ are incongruent modulo $H_{n,m}\cap H_{n,m+1}$. Let 
\[
w_{S,m+1}\prod_{\beta\in I_S}x_\beta\(\varpi^{-m-1}y_{S,\beta}\)
\ne
w_{S',m+1}\prod_{\beta\in I_S'}x_\beta\(\varpi^{-m-1}y'_{S',\beta}\)
\]
be two distinct elements of $B$. Suppose, for contradiction, that 
\begin{equation}\label{E:coset 5}
w_{S,m+1}\prod_{\beta\in I_S}x_\beta\(\varpi^{-m-1}y_{S,\beta}\)
\in
w_{S',m+1}\prod_{\beta\in I_{S'}}x_\beta\(\varpi^{-m-1}y'_{S',\beta}\)
\cdot
\(H_{n,m}\cap H_{n,m+1}\). 
\end{equation}
Note that if $S'=S$, then \eqref{E:coset 5} implies that 
\[
\prod_{\beta\in I_S}x_{\beta}\(\varpi^{-m-1}\(y_{S,\beta}-y'_{S,\beta}\)\)
\in
H_{n,m}\cap H_{n,m+1},
\]
which, in turn, implies that $y'_{S,\beta}=y_{S,\beta}$ for every $\beta\in I_S$, contradicting our assumption. Assume that $S'\ne S$. Then
\[
w_{S',m+1}\,w_{S,m+1}
=
w_{S,m+1}\,w_{S',m+1}
=
w_{S'',m+1},
\]
where
\[
S''=\(S\cup S'\)\setminus\(S\cap S'\)\ne\emptyset. 
\]
By \eqref{E:coset 5}, we have
\[
h
:=
\prod_{\beta\in I_{S'}}x_\beta\(-\varpi^{-m-1}y'_{S',\beta}\)
\cdot
w_{S'',m+1}
\cdot
\prod_{\beta\in I_{S}}x_\beta\(\varpi^{-m-1}y_{S,\beta}\)
\in
H_{n,m}\cap H_{n,m+1}. 
\] 
At this point, recall that $\stt{e_{-n}, \ldots, e_{-1}, e_0, e_1,\ldots, e_n}$ is an ordered basis of the quadratic space 
$\(V_n, \langle\cdot,\cdot\rangle\)$ defining the group $\SO_{2n+1}(F)$, whose Gram matrix is given by 
\[
\begin{pmatrix}
&&J_n\\
&2&\\
J_n
\end{pmatrix}. 
\]

Let $\el\in S''$. To obtain a contradiction, we compute 
\[
\langle he_{n+1-\el}\,,\, e_{n+1-\el}\rangle. 
\]
Since $h\in H_{n,m}\cap H_{n,m+1}$, we have 
\begin{equation}\label{E:coset 6}
\langle he_{n+1-\el}\,,\, e_{n+1-\el}\rangle
\in
\frak{p}^{-m}.  
\end{equation}
On the other hand, we have
\[
\prod_{\beta\in I_S}x_{\beta}\(\varpi^{-m-1}y_{S,\beta}\)e_{n+1-\el}
=
e_{n+1-\el}
+
\sum_{\substack{1\le i\le n\\i\nin S}}
\varpi^{-m-1}y_{S,\ep_i+\ep_\el} e_{-n-1+i}. 
\]
It follows that 
\begin{align*}
w_{S'',m+1}&\prod_{\beta\in I_S}x_{\beta}\(\varpi^{-m-1}y_{S,\beta}\)e_{n+1-\el}\\
&=
-\varpi^{-m-1}e_{-n-1+\el}
-
\sum_{\substack{1\le i\le n\\i\in S'\setminus S}} y_{S,\ep_i+\ep_\el} e_{n+1-i}
+
\sum_{\substack{1\le i\le n\\ i\nin S\cup S'}}\varpi^{-m-1}y_{S,\ep_i+\ep_\el} e_{-n-1+i}. 
\end{align*}
Since 
\[
\prod_{\beta\in I_{S'}}x_{\beta}\(\varpi^{-m-1}y'_{S',\beta}\)e_{n+1-\el}
=
e_{n+1-\el}
+
\sum_{\substack{1\le i\le n\\i\nin S'}}
\varpi^{-m-1}y'_{S',\ep_i+\ep_\el} e_{-n-1+i},
\]
we obtain
\begin{align*}
\langle he_{n+1-\el}\,,\, e_{n+1-\el}\rangle
=&
\langle
w_{S'',m+1}\prod_{\beta\in I_S}x_{\beta}\(\varpi^{-m-1}y_{S,\beta}\)e_{n+1-\el}\,,\,
\prod_{\beta\in I_{S'}}x_{\beta}\(\varpi^{-m-1}y'_{S',\beta}\)e_{n+1-\el}
\rangle\\
=&
\langle
\sum_{\substack{1\le i\le n\\ i\nin S\cup S'}}\varpi^{-m-1}y_{S,\ep_i+\ep_\el} e_{-n-1+i}\,,\,
\sum_{\substack{1\le i\le n\\i\nin S'}}\varpi^{-m-1}y'_{S',\ep_i+\ep_\el} e_{-n-1+i}
\rangle\\
&-
\langle
\sum_{\substack{1\le i\le n\\i\in S'\setminus S}} y_{S,\ep_i+\ep_\el} e_{n+1-i}\,,\,
\sum_{\substack{1\le i\le n\\i\nin S'}}\varpi^{-m-1}y'_{S',\ep_i+\ep_\el} e_{-n-1+i}
\rangle\\
&\,\,-
\langle
\varpi^{-m-1}e_{-n-1+\el}\,,\, \sum_{\substack{1\le i\le n\\i\nin S'}}\varpi^{-m-1}y'_{S',\ep_i+\ep_\el} e_{-n-1+i}
\rangle\\
&\,\,\,\,+
\langle
\sum_{\substack{1\le i\le n\\ i\nin S\cup S'}}\varpi^{-m-1}y_{S,\ep_i+\ep_\el} e_{-n-1+i}\,,\,
e_{n+1-\el}
\rangle\\
&\,\,\,\,\,\,-
\langle
\sum_{\substack{1\le i\le n\\i\in S'\setminus S}} y_{S,\ep_i+\ep_\el} e_{n+1-i}\,,\, e_{n+1-\el}
\rangle\\
&\,\,\,\,\,\,\,\,-
\langle
\varpi^{-m-1}e_{-n-1+\el}\,,\, e_{n+1-\el}
\rangle
=
-\langle
\varpi^{-m-1}e_{-n-1+\el}\,,\, e_{n+1-\el}
\rangle
=
-\varpi^{-m-1},
\end{align*}
contradicting \eqref{E:coset 6}. We thus conclude that two distinct elements of $B$ are incongruent modulo $H_{n,m}\cap H_{n,m+1}$.
This completes Step 3 and the proof of \propref{P:coset decomp}. \qed

\subsection{Hecke operators on $\GL_r$}\label{SS:Hecke GL}
We will need a result of Kondo--Yasuda (\cite{KondoYasuda2012}) concerning Hecke operators on $\GL_r$. To state their result, 
let $A_r$ be the diagonal torus of $\GL_r$, and let $\e^*_1,\ldots,\e^*_r\in{\rm Hom}(\bbG_m, A_r)$ be the standard basis. Denote
\[
\nu_i=\e^*_1+\cdots+\e^*_i
\]
for $1\le i\le r$, and set $\nu_0=0$. 

Let $\tau$ be an irreducible generic representation of $\GL_r(F)$ with $r\ge 1$. For each integer $m\ge 0$, denote by 
$\Gamma_{r,m}\subset\GL_r(\frak{o})$ the open compact subgroup introduced in \cite{JPSS1981} to establish the newform theory 
for generic representations of $\GL_r(F)$. In particular, we have 
\[
\dim_\bbC\tau^{\Gamma_{r,c_\tau}}=1. 
\]
Let $v_\tau\in\tau^{\Gamma_{r,c_\tau}}$ be a basis vector, referred to as a newform of $\tau$. 

To each integer $0\le i\le r$, define a Hecke operator $T_i$ on $\tau^{\Gamma_{r,c_\tau}}$ by 
\begin{align}\label{E:Hecke op}
\begin{split}
T_i(v_\tau)
&=
\vol\(\Gamma_{r,c_\tau},dk\)^{-1}\int_{\Gamma_{r,c_\tau}\varpi^{\nu_i}\Gamma_{r,c_\tau}}\tau(k)v_\tau\,dk\\
&=
\sum_{k\in\Gamma_{r,c_\tau}/\(\Gamma_{r,c_\tau}\bigcap\varpi^{\nu_i}\Gamma_{r, c_\tau}\varpi^{-\nu_i}\)}
\tau\(k\varpi^{\nu_i}\)v_\tau. 
\end{split}
\end{align}
Since $\tau^{\Gamma_{r,c_\tau}}$ is one-dimensional, there exist $\la_{\tau, 0},\ldots, \la_{\tau, r}\in\bbC$ such that 
\begin{equation}\label{E:Hecke ev}
T_i(v_\tau)=\la_{\tau, i}\,v_\tau
\end{equation}
for each $0\le i\le r$. Now, we have 

\begin{lm}\label{L:Hecke ev}
Let $\phi_\tau$ be the $L$-parameter of $\tau$, and suppose that $c_\tau\ge 1$. Then
\[
L\(s,\phi_\tau\)
=
\(
\sum_{i=0}^{r-1}(-1)^i\,\la_{\tau, i}\, q^{-i\(s+\tfrac{r-1}{2}\)+\tfrac{i(i-1)}{2}}
\)^{-1}. 
\]
\end{lm}

\begin{proof}
This follows immediately from \cite[Theorem 1.1]{KondoYasuda2012} and the identification of $L\(s,\phi_\tau\)$ with the $L$-factor attached to 
$\tau$ in \cite{GodementJacquet1972}. 
\end{proof}

We now rewrite the expression for $T_i(v_\tau)$ in a form convenient for our application. For this purpose, let 
$\e_1,\ldots,\e_r\in{\rm Hom}(A_r,\bbG_m)$ denote the standard basis. For each subset $S\subset\cI_{r}$, set 
\begin{equation}\label{E:J_S}
J_S
=
\stt{\e_i-\e_j\mid 1\le i<j\le r,\,i\in S,\,j\nin S}. 
\end{equation}
Define
\[
\nu_S=\sum_{i\in S}\ep_i^*,
\]
so that $\nu_{\stt{1,\ldots, i}}=\nu_i$ for $1\le i\le r$. For $1\le i<j\le r$ and $y\in F$, set
\[
\chi_{\e_i-\e_j}(y)
=
I_r+yE^r_{i,j},
\]
a root element of ${\rm SL}_r(F)\subset\GL_r(F)$ associated with $\e_i-\e_j$. 

\begin{lm}\label{L:Hecke op}
Assume that $r\ge 2$ and $c_\tau\ge 1$. Then, for $1\le i\le r-1$, we have
\[
T_i(v_\tau)
=
\sum_{\substack{S\subset \cI_{r-1}\\|S|=i}}\,
\sum_{y_{S,\beta}\in\frak{o}/\frak{p}}
\tau\(
\prod_{\beta\in J_S}\chi_{\beta}\(y_{S,\beta}\)\varpi^{\nu_S}
\)v_\tau. 
\]
\end{lm}

\begin{proof}
By \eqref{E:Hecke op}, we have to investigate the quotient
\begin{equation}\label{E:quotient}
\Gamma_{r,c_\tau}/\(\Gamma_{r,c_\tau}\bigcap\varpi^{\nu_i}\Gamma_{r, c_\tau}\varpi^{-\nu_i}\). 
\end{equation}
By \cite[Lemma 3.1]{Miyauchi2014}, this quotient admits a set of representatives of the form
\[
\pMX{a}{y}{0}{1}
\]
where
\[
a\in\Gamma_{r-1,0}/\(\Gamma_{r-1,0}\bigcap\varpi^{\nu_i}\Gamma_{r-1,0}\varpi^{-\nu_i}\)
\]
and 
\[
y
\in
L_0/aL_i. 
\]
Here $L_0$ and $L_i$ are $\frak{o}$-lattices defined by  
\[
L_0
=
\bigoplus_{j=1}^{r-1}\frak{o} E_j
\quad\text{and}\quad
L_i
=
\bigoplus_{j=1}^i\frak{p} E_j\oplus\bigoplus_{j=i+1}^{r-1}\frak{o} E_j,
\]
respectively, where $E_1,\ldots, E_{r-1}$ form the standard basis of $\M_{r-1, 1}(F)$, and we regard $\varpi^{\nu_i}$ as an element of
$\GL_{r-1}(F)$ via the embedding 
\begin{equation}\label{E:embedding GL}
\GL_{r-1}(F)\longto\GL_r(F);
\quad
a\longmapsto\pMX{a}{}{}{1}. 
\end{equation}
We also need a sublemma (\cite{Shintani1976}) due to Shintani, which asserts that
\[
\Gamma_{r-1,0}\,\varpi^{\nu_i}\,\Gamma_{r-1,0}
=
\bigsqcup_{\substack{S\subset\cI_{r-1}\\|S|=i}}\,\bigsqcup_{x\in Z_{r-1}(\frak{o})/ Z_{r-1}(\frak{o})_S}
x\varpi^{\nu_S}\Gamma_{r-1,0},
\]
where
\[
Z_{r-1}(\frak{o})_S
=
Z_{r-1}(\frak{o})\cap\varpi^{\nu_S}\Gamma_{r-1,0}\varpi^{-\nu_S}. 
\]
We note that in \cite[Sublemma]{Shintani1976}, one should take the set $I_i$ to be $I_i=\stt{\e\in\stt{0,1}^n\mid \sum_{j=1}^n \e_j=i}$, as pointed
out in the proof of \cite[Lemma 3.2]{Miyauchi2014}. 

Now, for each $\emptyset\ne S\subset\cI_{r-1}$ with $|S|=i$, let $u_S\in\Gamma_{r-1,0}$ be a permutation matrix such that 
\[
u^{-1}_S\varpi^{\nu_S}u_S
=
\varpi^{\nu_i}. 
\]
Then we have
\[
\Gamma_{r-1,0}\,\varpi^{\nu_i}\,\Gamma_{r-1,0}
=
\bigsqcup_{\substack{S\subset\cI_{r-1}\\|S|=i}}\,\bigsqcup_{x\in Z_{r-1}(\frak{o})/ Z_{r-1}(\frak{o})_S}
xu_S\cdot\varpi^{\nu_i}\Gamma_{r-1,0}. 
\]
This implies that 
\[
\stt{xu_S\mid \emptyset\ne S\subset\cI_{r-1},\,|S|=i,\,x\in Z_{r-1}(\frak{o})/ Z_{r-1}(\frak{o})_S}
\]
is a set of representatives for
\begin{equation}\label{E:quotient 2}
\Gamma_{r-1,0}/\(\Gamma_{r-1,0}\bigcap\varpi^{\nu_i}\Gamma_{r-1,0}\varpi^{-\nu_i}\). 
\end{equation}
Moreover, a simple computation shows that $Z_{r-1}(\frak{o})/Z_{r-1}(\frak{o})_S$ has a set of representatives given by
\[
\stt{\prod_{\beta\in J'_S}\chi_\beta\(y_{S,\beta}\)\mid y_{S,\beta}\in \frak{o}/\frak{p}},
\]
where 
\[
J'_S
=
\stt{\e_i-\e_j\mid 1\le i<j\le r-1,\,i\in S,\,j\nin S}. 
\]
Here we again regard $\prod_{\beta\in J'_S}\chi_\beta\(y_{S,\beta}\)$ as an element of $\GL_{r-1}(F)$ via the embedding \eqref{E:embedding GL}. 
Together, the quotient \eqref{E:quotient 2} has a set of representatives of the form $xw_S$
with $S\subset\cI_{r-1}$, $|S|=i$, and $x=\prod_{\beta\in J'_S}\chi_{\beta}\(y_{S,\beta}\)$ for some $y_{S,\beta}\in\frak{o}/\frak{p}$.   
We proceed to analyze the quotient 
\[
L_0/xu_S L_i.
\] 
To this end, note that
\[
L_0=xL_0
\quad
\text{and} 
\quad
xu_SL_i
=
xu_S\varpi^{\nu_i}L_0
=
x\varpi^{\nu_S} u_SL_0
=
x\varpi^{\nu_S}L_0. 
\]
It follows that 
\[
\stt{y_jE_j\mid j\in S,\,y_j\in\frak{o}/\frak{p}}
\]
forms a set of representatives for $L_0/xu_S L_i$. Together, we conclude that 
\[
\stt{\prod_{\beta\in J_S}\chi_{\beta}\(y_{S,\beta}\) u_S\mid S\subset\cI_{r-1},\,|S|=i,\,y_{S,\beta}\in\frak{o}/\frak{p}}
\]
is a set of representatives for the quotient \eqref{E:quotient}, where we view $u_S$ as an element of $\GL_r(F)$ through the embedding 
\eqref{E:embedding GL}. 

To complete the proof, we compute
\begin{align*}
T_i(v_\tau)
&=
\sum_{\substack{S\subset\cI_{r-1}\\|S|=i}}\,\sum_{y_{S,\beta}\in\frak{o}/\frak{p}}
\tau\(\prod_{\beta\in J_S}\chi_\beta\(y_{S,\beta}\)u_S\varpi^{\nu_i}\)v_\tau\\
&=
\sum_{\substack{S\subset\cI_{r-1}\\|S|=i}}\,\sum_{y_{S,\beta}\in\frak{o}/\frak{p}}
\tau\(\prod_{\beta\in J_S}\chi_\beta\(y_{S,\beta}\)\varpi^{\nu_S}u_S\)v_\tau
=
\sum_{\substack{S\subset\cI_{r-1}\\|S|=i}}\,\sum_{y_{S,\beta}\in\frak{o}/\frak{p}}
\tau\(\prod_{\beta\in J_S}\chi_\beta\(y_{S,\beta}\)\varpi^{\nu_S}\)v_\tau. 
\end{align*}
This proves the lemma. 
\end{proof}

\section{Final reduction}\label{S:3rd red}
In this section, we prove the following proposition. 

\begin{prop}\label{P:sc to seed}
If the space of newforms is non-zero for every irreducible generic supercuspidal representation, then it is also non-zero for all seed representations.
\end{prop}

Let $\pi$ be a seed representation of $\SO_{2n+1}(F)$ with $n>0$, and let $\phi$ be the corresponding 
seed $L$-parameter, written as in \eqref{E:L-par}.  Assume that $\pi$ is not supercuspidal. The proof is divided into two cases, according to  
whether or not both $|I_{\phi,\chi}|=|I_{\phi,\chi'}|=0$. The first case, namely when both $|I_{\phi,\chi}|=|I_{\phi,\chi'}|=0$, is much easier and 
in fact follows from the results of the previous sections. In contrast, the second case requires more effort and is where the level raising operators 
introduced in the previous section are needed. Note that $|I_{\phi,\chi}|=|I_{\phi,\chi'}|=0$ if and only if each representation occurring in $I_\phi$ is 
ramified, and hence if and only if $L\(s,\phi\)=1$. 

\subsection{The case $|I_{\phi,\chi}|=|I_{\phi,\chi'}|=0$}
By the construction in \S\ref{SS:construction}, $\pi$ can be realized as the unique generic subrepresentation of 
\[
\Pi
=
\Delta_1\x\cdot\cdot\cdot\x\Delta_\el\rtimes\sigma
\]
for some irreducible generic supercuspidal representation $\sigma\in\Irr_0\(\SO\)$ and irreducible essentially square-integrable representations 
$\D_j\in\Irr_1(\GL)$ for $j=1,\ldots,\el$. We will need the following lemma. 

\begin{lm}\label{L:seed cond case 1}
Let the notation be as above, and assume that $a_\sigma=c_\sigma$. Then $a_\Pi=c_\pi$. 
\end{lm}

\begin{proof}
Since $|I_{\phi,\chi}|=|I_{\phi,\chi'}|=0$, we have
\[
L\(s,\phi\)=L\(s,\phi_{\D_j}\)=L\(s,\phi_{\D^\vee_j}\)=1,
\]
where $\phi_{\D_j}$ (resp. $\phi_{\D_j^\vee}$) denotes the $L$-parameter of $\D_j$ (resp. $\D^\vee_j$) for $j=1,\ldots,\el$. Therefore,  
\[
\ep\(s,\phi,\psi\)
=
\gamma\(s,\phi,\psi\),
\quad
\ep\(s,\phi_{\D_j},\psi\)
=
\gamma\(s,\phi_{\D_j},\psi\)
\quad\text{and}\quad
\ep\(s,\phi_{\D^\vee_j},\psi\)
=
\gamma\(s,\phi_{\D^\vee_j},\psi\)
\]
for each $j$. Moreover, by \cite[Theorem 6.1]{JiangSoudry2003}, we also have $L\(s,\phi_\sigma\)=1$; hence 
\[
\ep\(s,\phi_\sigma,\psi\)
=
\gamma\(s,\phi_\sigma,\psi\),
\] 
where $\phi_\sigma$ is the $L$-parameter of $\sigma$. On the other hand, \cite[Theorem B]{JiangSoudry2004} implies that 
\[
\gamma\(s,\phi,\psi\)
=
\gamma\(s,\pi,\psi\),
\]
where $\gamma(s,\pi,\psi)$ denotes the $\gamma$-factor associated with $\pi$ and $\psi$ defined by the Langlands--Shahidi method 
(\cite{Shahidi1990}). Since $\pi$ is a constituent of $\Pi$, it follows that 
\[
\gamma\(s,\pi,\psi\)=\gamma\(s,\Pi,\psi\). 
\]
By the multiplicativity of the $\gamma$-factor and by applying \cite[Theorem B]{JiangSoudry2004} again, we obtain
\begin{align*}
\gamma\(s,\Pi,\psi\)
&=
\prod_{j=1}^\el\left[\gamma\(s,\D_j,\psi\)\gamma\(s,\D^\vee_j,\psi\)\right]\cdot\gamma(s,\sigma,\psi)\\
&=
\prod_{j=1}^\el\left[\gamma\(s,\phi_{\D_j},\psi\)\gamma\(s,\phi_{\D^\vee_j},\psi\)\right]\cdot\gamma\(s,\phi_\sigma,\psi\)\\
&=
\prod_{j=1}^\el\left[\ep\(s,\phi_{\D_j},\psi\)\ep\(s,\phi_{\D^\vee_j},\psi\)\right]\cdot\ep\(s,\phi_\sigma,\psi\). 
\end{align*}
 Combining the above identities, we deduce that
\[
\ep\(s,\phi,\psi\)
=
\prod_{j=1}^\el\left[\ep\(s,\phi_{\D_j},\psi\)\ep\(s,\phi_{\D^\vee_j},\psi\)\right]\cdot\ep\(s,\phi_\sigma,\psi\). 
\]
By comparing the exponents of $q$ on both sides, and by applying \lmref{L:useful}, we find that 
\begin{align*}
c_\pi
=
\sum_{j=1}^\el \(c_{\D_j}+c_{\D^\vee_j}\)+c_{\sigma}
=
2\sum_{j=1}^\el c_{\D_j}+a_\sigma
=
a_\Pi. 
\end{align*}
This completes the proof. 
\end{proof}

As a corollary, we obtain

\begin{cor}\label{C:sc to seed case 1}
If the space of newforms is non-zero for every irreducible generic supercuspidal representation, then it is also non-zero for all seed 
representations with trivial $L$-factor. 
\end{cor}

\begin{proof}
Let $\pi$ be a seed representation of $\SO_{2n+1}(F)$ with $n>0$, and let $\phi$ be the corresponding seed $L$-parameter, written as in 
\eqref{E:L-par}.  Assume that $\pi$ is not supercuspidal. By the construction in \S\ref{SS:construction}, $\pi$ can be realized as the unique
generic subrepresentation of 
\[
\Pi
=
\Delta_1\x\cdot\cdot\cdot\x\Delta_\el\rtimes\sigma
\]
for some irreducible generic supercuspidal representation $\sigma\in\Irr_0\(\SO\)$ and irreducible essentially square-integrable representations 
$\D_j\in\Irr_1(\GL)$ for $j=1,\ldots,\el$. Now the assumption that $L(s,\phi)=1$ implies that every representation occurring in $I_\phi$ is ramified. 
Hence, by \lmref{L:unram 4},
\[
|M^*_{ur}(\D_j)|=1
\]
for $j=1,\ldots,\el$. Since $\sigma$ is supercuspidal, we also have 
\[
|\mu^*_{ur}(\sigma)|=1. 
\] 
It then follows from \lmref{L:unram 3}, and by induction on $\el$ that 
\[
|\mu^*_{ur}\(\Pi\)|=1. 
\]
Therefore,
\[
|\mu^*_{ur}(\pi)|=|\mu^*_{ur}\(\Pi\)|, 
\]
and hence
\[
\pi^{K_{n,m}}=\Pi^{K_{n,m}}
\]
for each $m\ge 0$ by \corref{C:key P}. Since $a_\sigma=c_\sigma$ by the assumption, it follows from \lmref{L:seed cond case 1} that 
$c_\pi=a_\Pi$. In particular,
\[
\pi^{K_{n,c_\pi}}=\Pi^{K_{n,a_\Pi}}\ne 0. 
\]
This proves the corollary. 
\end{proof}

Because of \corref{C:sc to seed case 1}, we will focus on Case 2, that is, the case where $|I_{\phi,\chi}|=1$ or $|I_{\phi,\chi'}|=1$, throughout the 
rest of this section.

\subsection{The case $|I_{\phi,\chi}|=1$ or $|I_{\phi,\chi'}|=1$}\label{SS:seed case 2}
We may assume, without loss of generality, that 
\[
|I_{\phi,\chi}|=1. 
\]
Then $I_{\phi,\chi}=\stt{\ka}$ for some $\ka\in\frac{1}{2}+\bbZ_+$. We can write 
\[
\phi
=
\phi_\chi\bt S_{2\ka+1}\oplus \phi'
\]
for some discrete $L$-parameter $\phi'$ (possibly empty) of $\SO_{2n-2\ka}(F)$. Note that 
\[
|I_{\phi',\chi}|=0. 
\]
In particular, if $|I_{\phi',\chi'}|=0$ as well, then we reduce to Case 1. Denote by $\pi'$ the irreducible generic square-integrable 
representation of $\SO_{2n-2\ka}(F)$ with $L$-parameter $\phi'$. Then, by the construction in \S\ref{SS:construction} and \lmref{L:inj conj}, 
$\pi$ can be realized as the unique generic subrepresentation of the standard module
\[
\Pi:=\D_\chi[\ka, \tfrac{1}{2}]\rtimes \pi'. 
\]
The following lemma compares $c_\pi$ and $a_\Pi$. 

\begin{lm}\label{L:seed cond case 2}
Let the notation be as above and assume that $a_{\pi'}=c_{\pi'}$. Then $a_{\Pi}=c_\pi-1$. 
\end{lm}

\begin{proof}
Since 
\[
\phi
=
\phi_\chi\bt S_{2\ka+1}\oplus \phi',
\]
it follows from \lmref{L:cond} that 
\[
c_\pi=c_{\pi'}+2\ka. 
\]
On the other hand, by \lmref{L:useful}, \eqref{E:twist}, and \lmref{L:cond}, we have
\[
a_\Pi
=
a_{\pi'}+2\(2\(\tfrac{\ka-\frac{1}{2}}{2}\)\)
=
c_{\pi'}+2\ka-1
=
c_\pi-1. 
\]
This proves the lemma. 
\end{proof}

\lmref{L:seed cond case 2} suggests that we should investigate the subspace $\Pi^{K_{n,a_\Pi+1}}$. 
To this end, knowledge of the length of $\Pi$ is necessary. The following lemma, proved by Atobe when $\ka=\frac{1}{2}$  
(\cite[Lemma 5.1]{Atobe2020}) and by Lo in general (see \S\ref{S:appendix}), is crucial to our investigation. 

\begin{lm}\label{L:length 2}
In the above setting, $\Pi$ has length two and we have the exact sequence 
\[
0\longto\pi\longto\Pi\longto\varrho\longto 0,
\]
where $\varrho$ is the Langlands quotient of $\Pi$. 
\end{lm}

At this point, we can describe our strategy for settling Case 2, which proceeds by computing the level-raising operators. More precisely, let 
\[
\Pi':=\D_\chi[\ka,\tfrac{1}{2}]^\vee\rtimes\pi',
\]
and let
\[
M:\Pi\longto\Pi'
\]
denote the standard intertwining operator (\cite{Silberger1979}), defined by an absolutely convergent integral. The image of $M$ is isomorphic 
to $\varrho$; hence, by \lmref{L:length 2}, $\pi$ can be characterized as the kernel of $M$. This characterization turns out to be very important 
for our argument. Indeed, assume that $a_{\pi'}=c_{\pi'}$. Then, by \lmref{L:useful} and \thmref{T:mainA} (1), we have
\[
\dim_{\bbC}\Pi^{K_{n,a_\Pi}}=1. 
\]
Denote by $f_{\Pi}$ a basis vector of this one-dimensional subspace. Then $\theta_{a_{\Pi}}(f_\Pi)$ and $\theta'_{a_\Pi}(f_{\Pi})$ both belong to 
\[
\Pi^{K_{n,a_\Pi}+1}=\Pi^{K_{n,c_\pi}}.
\] 
Now, if there exist $c, c'\in\bbC$ such that 
\[
c\,\theta_{a_\Pi}(f_{\Pi})+c'\theta'_{a_\Pi}(f_{\Pi})\ne 0
\quad\text{and}\quad
M\(c\,\theta_{a_\Pi}(f_{\Pi})+c'\theta'_{a_\Pi}(f_{\Pi})\)=0,
\]
then we can conclude that
\[
c\,\theta_{a_\Pi}(f_{\Pi})+c'\theta'_{a_\Pi}(f_{\Pi})\in \pi^{K_{n,c_\pi}},
\]
and hence $\pi^{K_{n,c_\pi}}\ne 0$. 

\subsection{Setup for computations}\label{SS:setup for level-raising}
In this subsection, we prepare for computing the level raising operators. We retain the setup in \S\ref{SS:seed case 2}. Thus, we have 
\[
\Pi=\D_\chi[\ka,\tfrac{1}{2}]\rtimes\pi'
\quad\text{and}\quad
\Pi'=\D_\chi[\ka,\tfrac{1}{2}]^\vee\rtimes\pi'. 
\]
To streamline our computation, let $r=\ka+\frac{1}{2}$, $\tau=\St_\chi\(\tfrac{r-1}{2}\)$, and set
\[
\Pi_s
=
\tau|\cdot|^s\rtimes\pi',
\]
where $s\in\bbC$. In particular, we have $\Pi=\Pi_\frac{r}{2}$ and $\Pi'=\Pi_{-\frac{r}{2}}$ by \eqref{E:twist}.  

By \lmref{L:useful} and \lmref{L:cond}, we have
\[
a_{\Pi_s}
=
a_{\pi'}+2(r-1)
=
a_{\pi'}+2\ka-1. 
\]
We will need the following lemma. Recall that the elements $\la_j$ are defined in \cite[(2.1)]{YCheng2025}. On the other hand,  for each integer 
$j$, we denote $x_{r,j}=x_{-\ep_r}\(\varpi^j\)$.

\begin{lm}\label{L:basis seed case 2}
Let $a=a_{\pi'}$ or $a=a_{\pi'}+1$, and let $v_\tau$ be a newform of $\tau$. Then
\begin{itemize}
\item[(1)]
If $r<n$, then the subspace $\Pi_s^{K_{n,a+2(r-1)}}$ admits a basis such that each basis vector $f_s$ is characterized by 
\[
\supp\(f_s\)
=
P_{\a_r}(F)\,x_{r,a+r-1 }\,K_{n,a+2(r-1)}
\quad\text{and}\quad
f_s\(x_{r,a+r-1}\)
=
v_\tau\ot v. 
\]
Here $v$ is a basis vector of $\pi'^{K'_{n-r,a}}$, where
\[
K'_{n-r,a}:=\varpi^{-(r-1)\la_{n-r}}K_{n-r,a}\varpi^{(r-1)\la_{n-r}}. 
\]
\item[(2)] 
If $r=n$, then $\pi'=1_{\SO_1(F)}$, and hence $a_{\pi'}=0$. In this case, $\Pi_s^{K_{n,2(n-1)}}$ is one-dimensional and has a basis vector 
$f_s$ that is characterized by 
\[
\supp\(f_s\)
=
P_{\a_n}(F)\,x_{n,n-1 }\,K_{n,2(n-1)}
\quad\text{and}\quad
f_s\(x_{n,n-1}\)
=
v_\tau. 
\]
On the other hand, $\Pi_s^{K_{n,2(n-1)+1}}$ is two-dimensional and has basis vectors $f'_s, f''_s$ that are characterized by 
\[
\supp\(f'_s\)
=
P_{\a_n}(F)\,x_{n,n }\,K_{n,2(n-1)+1}
\quad\text{and}\quad
f'_s\(x_{n,n}\)
=
v_\tau,
\]
and 
\[
\supp\(f''_s\)
=
P_{\a_n}(F)\,x_{n,n}w_{\ep_n,2(n-1)+1}\,K_{n,2(n-1)+1}
\quad\text{and}\quad
f''_s\(x_{n,n}w_{\ep_n,2(n-1)+1}\)
=
v_\tau. 
\]
\end{itemize}
\end{lm}

\begin{proof}
The proof is similar to that of \cite[Lemma 6.1]{YCheng2025}. Assume first that $r<n$, and let $t_m=\varpi^{\lfloor\frac{m}{2}\rfloor\la_n}$,  
where $m\ge 0$ is an integer. By \cite[Corollary 4.3]{YCheng2025},
\[
\SO_{2n+1}(F)
=
\bigsqcup_{j=e}^{\lceil\frac{m}{2}\rceil} P_{\a_r}(F)\,x_{r,j}\,K^0_{n,m},
\]
where $e\in\stt{0,1}$ satisfies $m\equiv e\pmod{2}$, and $K^0_{n,m}=t_m\,K_{n,m}\,t^{-1}_m$. Since 
\[
t^{-1}_m\,P_{\a_r}(F)\,t_m=P_{\a_r}(F)
\quad\text{and}\quad
t^{-1}_m\,x_{r,j}\,t_m=x_{r,j+\lfloor\frac{m}{2}\rfloor},
\]
it follows that 
\[
\SO_{2n+1}(F)
=
\bigsqcup_{j=\lceil\frac{m}{2}\rceil}^{m} P_{\a_r}(F)\,x_{r,j}\,K_{n,m}. 
\]
This implies that $P_{\a_r}(F)\backslash\SO_{2n+1}(F)/K_{n,m}$ has a set of representatives given by 
\begin{equation}\label{E:coset rep for std K when r<n}
\stt{x_{r,j}\mid \lceil\tfrac{m}{2}\rceil\le j\le m}. 
\end{equation}

Next, for a given open compact subgroup $K$ of $\SO_{2n+1}(F)$ and an element $x\in\SO_{2n+1}(F)$, we define
\[
M_{r, K}^x
=
\stt{g\in M_{\a_r}(F)\mid\text{$x^{-1}gux\in K$ for some $u\in N_{\a_r}(F)$}}. 
\]
By \cite[Corollary 5.2]{YCheng2025}, we have
\[
M_{r, K_{n,m}^0}^{x_{r,j}}
=
\stt{\diag{a, h, a^*}\mid a\in\Gamma_{r,\lceil\frac{m}{2}\rceil-j},\,h\in K_{n-r, 2j-e}^0}
\simeq
\Gamma_{r, \lceil\frac{m}{2}\rceil-j}\,\x\,K^0_{n-r, 2j-e}
\]
for $e\le j\le \lceil\frac{m}{2}\rceil$. Since $t_m$ normalizes both $M_{\a_r}(F)$ and $N_{\a_r}(F)$, it follows that 
\[
g\in M_{r, K_{n,m}^0}^{x_{r,j}}\Longleftrightarrow t^{-1}_mgt_m\in M^{x_{r, j'}}_{r, K_{n,m}},
\]
where $j'=j+\lfloor\frac{m}{2}\rfloor$. Combining these, we obtain
\[
M^{x_{r,j}}_{r, K_{n,m}}
\simeq
\Gamma_{r, m-j}\,\x\,K'_{n-r, 2j-m}
\]
for $\lceil\frac{m}{2}\rceil\le j\le m$, where $K'_{n-r, 2j-m}:=\varpi^{(j-m)\la_{n-r}}\,K_{n-r, 2j-m}\,\varpi^{-(j-m)\la_{n-r}}$. 

It then follows from the Mackey decomposition (\cite[Theorem 1]{Yamamoto2022}) and Frobenius reciprocity 
(\cite[Theorem 2.28]{BZ1976}) that 
\[
\Pi_s^{K_{n,m}}
\simeq
\bigoplus_{j=\lceil\frac{m}{2}\rceil}^m\(\tau|\cdot|^s\bt\pi'\)^{M^{x_{r,j}}_{r, K_{n,m}}}
\simeq
\bigoplus_{j=\lceil\frac{m}{2}\rceil}^m
\(\tau|\cdot|^s\)^{\Gamma_{r, m-j}}\ot\pi'^{K'_{n-r, 2j-m}}. 
\]
Note that if
\[
\(\tau|\cdot|^s\)^{\Gamma_{r,m-j}}\ot\pi'^{K'_{n-r,2j-m}}\ne 0,
\]
then $m-j\ge (r-1)$ and $2j-m\ge a_{\pi'}$, which implies that 
\[
j\ge a_{\pi'}+r-1. 
\]
In particular, when $m=a+2(r-1)$, we have
\[
\Pi_s^{K_{n,m}}
\simeq
\(\tau|\cdot|^s\bt\pi'\)^{M^{x_{r,a+r-1}}_{r, K_{n,m}}}
\simeq
\(\tau|\cdot|^s\)^{\Gamma_{r, r-1}}\ot\pi'^{K'_{n-r, a}}. 
\]
Thus, if $f_s\in\Pi_s^{K_{n,a+2(r-1)}}$, then it satisfies
\[
\supp\(f_s\)
=
P_{\a_r}(F)\,x_{r,a+r-1 }\,K_{n,a+2(r-1)}
\quad\text{and}\quad
f_s\(x_{r,a+r-1}\)
\in
\(\tau|\cdot|^s\)^{\Gamma_{r, r-1}}\ot\pi'^{K'_{n-r, a}}. 
\]
This proves $(1)$. 

Assume that $r=n$. The proof is similar to the above argument, except that in this case, the quotient 
$P_{\a_n}(F)\backslash\SO_{2n+1}(F)/K_{n,m}$ has a set of representatives given by 
\begin{equation}\label{E:coset rep for std K when r=n}
\stt{x_{r,j},\,\,x_{r,i}\,w_{\ep_n,m}\mid \lceil\tfrac{m}{2}\rceil\le j\le m,\,\,\lceil\tfrac{m+1}{2}\rceil\le i\le m}. 
\end{equation}
Since $w_{\ep_n,m}$ normalizes $K_{n,m}$, it follows that 
\[
M_{n, K_{n,m}}^{x_{n,j}w_{\ep_n, m}}
=
M_{n, K_{n,m}}^{x_{n,j}}. 
\]
Now, a similar argument shows that
\[
\Pi_s^{K_{n,m}}
\simeq
\(\tau|\cdot|^s\)^{M^{x_{n,n-1}}_{n, K_{n,m}}}
\simeq 
\(\tau|\cdot|^s\)^{\Gamma_{n,n-1}}
\]
when $m=2(n-1)$, and 
\[
\Pi_s^{K_{n,m}}
\simeq
\(\tau|\cdot|^s\)^{M^{x_{n,n}}_{n, K_{n,m}}}
\oplus
\(\tau|\cdot|^s\)^{M^{x_{n,n}w_{\ep_n,m}}_{n, K_{n,m}}}
\simeq 
\(\tau|\cdot|^s\)^{\Gamma_{n,n-1}}
\oplus
\(\tau|\cdot|^s\)^{\Gamma_{n,n-1}},
\]
when $m=2(n-1)+1$. This proves $(2)$ and hence completes the proof of the lemma. 
\end{proof}

Assume that $a_{\pi'}=c_{\pi'}$ and that $a=a_{\Pi_s}$ from now on.
By \thmref{T:mainA} (1) and \lmref{L:useful}, we have
\[
\dim_\bbC\Pi_s^{K_{n, a}}=\dim_{\bbC}\pi'^{K_{n-r},c_{\pi'}}=1. 
\]
It then follows from \lmref{L:basis seed case 2} $(1)$ that $\Pi_s^{K_{n,a}}$ admits a basis vector $f_s$, which is characterized by 
\[
\supp\(f_s\)
=
P_{\a_r}(F)\,x_{r, \xi}\,K_{n,a}
\quad\text{and}\quad
f_s\(x_{r,\xi}\)
=
v_\tau\ot v'_{\pi'},
\]
where $v'_{\pi'}\in\pi'^{K'_{n-r,c_{\pi'}}}$ is a basis vector, and 
\[
\xi:=c_{\pi'}+r-1. 
\]
For simplicity, we denote 
\[
\theta=\theta_{a}
\quad\text{and}\quad
\theta'=\theta'_{a}
\]
in the remainder of this section. Our aim is to determine 
\[
\theta(f_s),\, \theta'(f_s)\in\Pi^{K_{n,a+1}}.  
\]
By \lmref{L:basis seed case 2} $(2)$, it suffices to compute 
\begin{equation}\label{E:seed case 2 main}
\theta(f_s)\(x_{r,\xi+1}\)
\quad\text{and}\quad
\theta'(f_s)\(x_{r,\xi+1}\),
\end{equation}
when $r<n$, and, in addition, 
\begin{equation}\label{E:seed case 2 main'}
\theta(f_s)\(x_{n,n}w_{\ep_n,a+1}\)
\quad\text{and}\quad
\theta'(f_s)\(x_{n,n}w_{\ep_n, a+1}\),
\end{equation}
when $r=n$. However, we claim that when $r=n$, the computation of \eqref{E:seed case 2 main'} can be reduced to that of
\eqref{E:seed case 2 main}. Indeed, \thmref{T:mainA} $(2)$ and \cite[Lemma 6.1]{YCheng2025} imply that 
\[
\Pi_s\(w_{\ep_n, a}\)f_s=\e f_s
\]
for some $\ep\in\stt{\pm 1}$. It follows that 
\begin{equation}\label{E:theta when r=n}
\theta(f_s)\(x_{n,n}w_{\ep_n,a+1}\)
=
\ep\Pi_s\(w_{\ep_n,a_{\Pi_s}+1}\)\circ\theta\(\Pi_s\(w_{\ep_n, a}\)f_s\)\(x_{n,n}\)
=
\ep\theta'(f_s)\(x_{n,n}\). 
\end{equation}
On the other hand, since $w_{\ep_n,m}^2=1$ for every $m$, we also have
\begin{equation}\label{E:theta' when r=n}
\theta'(f_s)\(x_{n,n}w_{\ep_n, a+1}\)
=
\ep\theta(f_s)\(x_{n,n}\),
\end{equation}
by the same argument. Now the claim follows.  

We will need the following notation in later computations. 
Let $1\le b \le d\le n$ be integers and $S\subset\stt{b,\cdots, d}$ be a subset (possibly empty). Define the sets
\begin{equation}\label{E:I_S^pm 2} 
I^-_{S,b, d}
=
\stt{\ep_i-\ep_j\mid b\le i<j\le d,\,i\nin S,\,j\in S},
\quad
I^+_{S,b,d}
=
\stt{\ep_i+\ep_j\mid b\le i<j\le d,\,i\nin S,\,j\nin S}. 
\end{equation}
When $b, d$ are clear from the context, we simply denote $I_S^\pm=I_{S,b,d}^\pm$. 
Note that if $S_0\subset\stt{1,\ldots, r}$ and $S_1\subset\stt{r+1,\ldots, n}$ are subsets and $S=S_0\cup S_1$, then
\[
I^-_S
=
I^-_{S_0}\sqcup I^-_{S_1}\sqcup\stt{\ep_i-\ep_j\mid 1\le i\le r<j\le n,\,i\nin S_0,\,j\in S_1}
\]
and 
\[
I^+_{S}
=
I^+_{S_0}\sqcup I^+_{S_1}\sqcup\stt{\ep_i+\ep_j\mid 1\le i\le r<j\le n,\,i\nin S_0,\,j\nin S_1}. 
\]

\subsection{Some double cosets}
In this and the next subsections, we prepare for computing \eqref{E:seed case 2 main}. Note that
\begin{equation}\label{E:theta main id}
\theta(f_s)\(x_{r,\xi+1}\)
=
\sum_{k\in K_{n,a+1}/\(K_{n,a}\cap K_{n,a+1}\)}
f_s\(x_{r,\xi+1}k\)
\end{equation}
and 
\begin{equation}\label{E:theta' main id}
\theta'(f_s)\(x_{r,\xi+1}\)
=
\sum_{k\in K_{n,a+1}/\(K_{n,a}\cap K_{n,a+1}\)}
f_s\(x_{r,\xi+1}w_{\ep_n,a+1}kw_{\ep_n,a}\).
\end{equation}
Moreover, a set of representatives for $K_{n,a+1}/\(K_{n,a}\cap K_{n,a+1}\)$ is given by \propref{P:coset decomp}. Since 
\[
\supp(f_s)
=
P_{\a_r}(F)\,x_{r,\xi}\,K_{n,a},
\]
it is natural to consider the following lemma. To state it, let $S\subset\cI_n$ be a subset (possibly empty). Recall that $I_S$ (resp. $w_{S,m}$) 
is defined by \eqref{E:I_S} (resp. \eqref{E:w_S}), and $I_S^\pm=I_{S,1,n}^\pm$ are defined by \eqref{E:I_S^pm 2}. 
Note that 
\[
\ep_i-\ep_j\in I_S^-\Longrightarrow \ep_i+\ep_j\in I_S
\quad
\text{and}
\quad
I_S=I_S^+\sqcup\stt{\ep_i+\ep_j\mid \ep_i-\ep_j\in I_S^-}. 
\]
On the other hand, we denote 
\[
\la_S=\sum_{i\in S}\ep^*_i,
\]
so that $\la_i=\la_{\cI_i}$ for $1\le i\le n$. Now, the lemma can be stated as follows. 

\begin{lm}\label{L:eva}
Let $c_1, c_2\ge 0$ be integers and put $m=c_1+2c_2$ and $c=c_1+c_2$. Let $S\subset\cI_n$ be a subset, and for each $\beta\in I_S$, 
let $y_\beta\in\frak{o}/\frak{p}$. Then
\begin{itemize}
\item[(1)]
If $r\in S$, then
\[
x_{r, c+1}w_{S,m+1}\prod_{\beta\in I_S}x_{\beta}\(\varpi^{-m-1}y_\beta\)w_{S,m}
=
\prod_{\beta\in I_S^-}x_\beta\(y'_\beta\)\prod_{\beta\in I_S^+}x_{\beta}\(\varpi^{-m-1}y_\beta\)\varpi^{-\la_S} x_{r, c},
\]
where $y'_{\ep_i-\ep_j}=-y_{\ep_i+\ep_j}$ for $\ep_i-\ep_j\in I_S^-$. 
\item[(2)]
If $r\nin S$, then 
\[
P_{\a_r}(F)
x_{r, c+1}w_{S,m+1}\prod_{\beta\in I_S}x_{\beta}\(\varpi^{-m-1}y_\beta\)w_{S,m}
K_{n,m}
=
P_{\a_r}(F)x_{r, c+1}K_{n,m}. 
\]
\end{itemize}
\end{lm}

\begin{proof}
We first verify $(1)$. To this end, note that 
\[
w_{S,m+1}x_{r,c+1}w_{S,m+1}
=
x_{\ep_r}\(-\varpi^{c-m}\)
\quad\text{and}\quad
w_{S,m}x_{\ep_r}\(-\varpi^{c-m}\)w_{S,m}
=
x_{r,c}. 
\]
Since $x_{\ep_r}\(-\varpi^{c-m}\)$ commutes with $\prod_{\beta\in I_S}x_\beta\(\varpi^{-m-1}y_\beta\)$, 
$w_{S,m+1}$ commutes with $x_{\beta}\(\varpi^{-m-1}y_\beta\)$ for $\beta\in I_S^+$, and for every $\ep_i+\ep_j\in I_S$ with $j\in S$,
we have
\[
w_{S,m+1}\,x_{\ep_i+\ep_j}\(\varpi^{-m-1}y_{\ep_i+\ep_j}\)\,w_{S,m+1}
=
x_{\ep_i-\ep_j}\(-y_{\ep_i+\ep_j}\). 
\]
It follows that 
\begin{align*}
x_{r, c+1}w_{S,m+1}\prod_{\beta\in I_S}x_{\beta}\(\varpi^{-m-1}y_\beta\)w_{S,m}
&=
w_{S,m+1}x_{\ep_r}\(-\varpi^{c-m}\)\prod_{\beta\in I_S}x_{\beta}\(\varpi^{-m-1}y_\beta\)w_{S,m}\\
&=
w_{S,m+1}\prod_{\beta\in I_S}x_{\beta}\(\varpi^{-m-1}y_\beta\)x_{\ep_r}\(-\varpi^{c-m}\)w_{S,m}\\
&=
w_{S,m+1}\prod_{\beta\in I_S}x_{\beta}\(\varpi^{-m-1}y_\beta\)w_{S,m}x_{r,c}\\
&=
w_{S,m+1}\prod_{\beta\in I_S}x_{\beta}\(\varpi^{-m-1}y_\beta\)w_{S,m+1}^2w_{S,m}x_{r,c}\\
&=
w_{S,m+1}\prod_{\beta\in I_S}x_{\beta}\(\varpi^{-m-1}y_\beta\)w_{S,m+1}\varpi^{-\la_S}x_{r,c}\\
&=
\prod_{\beta\in I_S^-}x_\beta\(y'_\beta\)\prod_{\beta\in I_S^+}x_{\beta}\(\varpi^{-m-1}y_\beta\)\varpi^{-\la_S} x_{r, c}. 
\end{align*}
This proves $(1)$. 

To verify $(2)$, we repeatedly use the Chevalley commutator formula:
\[
[x_\a(y), x_\beta(z)]
=
\prod_{i,j\in\bbN}
x_{i\a+j\beta}\(c_{i,j}y^iz^j\)
\quad
(\a+\beta\ne 0)
\]
for some $c_{i,j}\in\bbZ$, where the product is taken over all pairs $i,j\in\bbN$ for which $i\a+j\beta$ is a root, in order of increasing $i+j$. 
In fact, the following corollary suffices for our argument. Since $c_{i,j}$ are integers (hence in $\frak{o}$), it follows that 
\[
[x_\a(y), x_\beta(z)]
\in
\prod_{i,j\in\bbN} x_{i\a+j\beta}\(\frak{p}^{ik+jh}\)
\subset
\prod_{i,j\in\bbN} x_{i\a+j\beta}\(F\)
\]
if $y\in\frak{p}^k$ and $z\in\frak{p}^h$.

Note that the above formula also implies that if $\a+\beta\ne 0$ and $i\a+j\beta$ is not a root for any $i,j\in\bbN$, then $x_\a(y)$ commutes 
with $x_\beta(z)$. In this case, we simply say that $x_\a$ commutes with $x_\beta$. More generally, by saying that $g\in\SO_{2n+1}(F)$ 
commutes with $x_\beta$, we mean that $g$ commutes with $x_\beta(y)$ for every $y\in F$.  We begin with the following decomposition: 
\[
I_S
=
I_S^{(1)}\sqcup I_S^{(2)}\sqcup I_S^{(3)},
\]
where
\[
I_S^{(1)}=\stt{\ep_i+\ep_j\in I_S\mid 1\le i<j\le n,\,i\ne r,\,j\ne r},
\]
\[
I_S^{(2)}=\stt{\ep_i+\ep_r\in I_S\mid 1\le i< r}
\quad\text{and}\quad
I_S^{(3)}=\stt{\ep_r+\ep_j\in I_S\mid  r<j\le n}. 
\]
Since $r\nin S$, it follows that $w_{S,m+1}$ commutes with $x_{r,c+1}$ and with each $x_\beta$ for $\beta\in I_S^{(2)}$. The same reason also 
implies that $x_{r,c+1}$ commutes with $x_\beta$ for every $\beta\in I_S^{(1)}$. On the other hand, if $\ep_i+\ep_j\in I_S^{(1)}$ and $y\in F$, 
then
\[
w_{S,m+1}x_{\ep_i+\ep_j}(y)w_{S,m+1}\in x_{\ep_i-\ep_j}(F)\subset P_{\a_r}(F)
\]
if $j\in S$, while 
\[
w_{S,m+1}x_{\ep_i+\ep_j}(y)w_{S,m+1}\in x_{\ep_i+\ep_j}(F)\subset P_{\a_r}(F)
\]
if $j\nin S$. We thus conclude that 
\[
w_{S,m+1}\prod_{\beta\in I_S^{(1)}}x_\beta\(\varpi^{-m-1}y_\beta\)w_{S,m+1}\in P_{\a_r}(F). 
\]
Combining these, we obtain
\begin{align*}
P_{\a_r}(F)&
x_{r, c+1}w_{S,m+1}
\prod_{\beta\in I_S^{(1)}\sqcup I_S^{(2)}\sqcup I_S^{(3)}}x_{\beta}\(\varpi^{-m-1}y_\beta\)
w_{S,m}K_{n,m}\\
&=
P_{\a_r}(F)w_{S,m+1}
\prod_{\beta\in I_S^{(1)}}x_{\beta}\(\varpi^{-m-1}y_\beta\)
x_{r,c+1}
\prod_{\beta\in I_S^{(2)}\sqcup I_S^{(3)}}x_{\beta}\(\varpi^{-m-1}y_\beta\)
w_{S,m}K_{n,m}\\
&=
P_{\a_r}(F)x_{r,c+1}w_{S,m+1}
\prod_{\beta\in I_S^{(2)}\sqcup I_S^{(3)}}x_{\beta}\(\varpi^{-m-1}y_\beta\)
w_{S,m}K_{n,m}\\
&=
P_{\a_r}(F)x_{r,c+1}
\prod_{\beta\in I_S^{(2)}}x_{\beta}\(\varpi^{-m-1}y_\beta\)
w_{S,m+1}
\prod_{\beta\in I_S^{(3)}}x_{\beta}\(\varpi^{-m-1}y_\beta\)
w_{S,m}K_{n,m}. 
\end{align*}
Since
\[
x_{r,c+1}x_{\ep_i+\ep_r}\(\varpi^{-m-1}y_{\ep_i+\ep_r}\)x_{r,c+1}^{-1}
\in
x_{\ep_i}\(F\)x_{\ep_i-\ep_r}(F)x_{\ep_i+\ep_r}(F)
\subset 
P_{\a_r}(F)
\]
for every $\ep_i+\ep_r\in I_S^{(2)}$, it follows that 
\begin{align*}
P_{\a_r}(F)&x_{r,c+1}
\prod_{\beta\in I_S^{(2)}}x_{\beta}\(\varpi^{-m-1}y_\beta\)
w_{S,m+1}
\prod_{\beta\in I_S^{(3)}}x_{\beta}\(\varpi^{-m-1}y_\beta\)
w_{S,m}K_{n,m}\\
&=
P_{\a_r}(F)x_{r,c+1}w_{S,m+1}
\prod_{\beta\in I_S^{(3)}}x_{\beta}\(\varpi^{-m-1}y_\beta\)
w_{S,m}K_{n,m}\\
&=
P_{\a_r}(F)w_{S,m+1}x_{r,c +1}
\prod_{\beta\in I_S^{(3)}}x_{\beta}\(\varpi^{-m-1}y_\beta\)
w_{S,m}K_{n,m}. 
\end{align*}

To proceed, we compute 
\begin{align*}
x_{r,c+1}x_{\ep_r+\ep_j}\(\varpi^{-m-1}y_{\ep_r+\ep_j}\)x_{r,c+1}^{-1}
&\in
x_{\ep_j}\(\frak{p}^{-c_2}\)x_{-\ep_r+\ep_j}\(\frak{p}^{c_1+1}\)x_{\ep_r+\ep_j}\(\frak{p}^{-m-1}\)\\
&=
x_{\ep_r+\ep_j}\(\frak{p}^{-m-1}\)x_{\ep_j}\(\frak{p}^{-c_2}\)x_{-\ep_r+\ep_j}\(\frak{p}^{c_1+1}\),
\end{align*}
for $\ep_r+\ep_j\in I_S^{(3)}$, where the equality follows from the fact that $x_{\ep_r+\ep_j}$ commutes with both $x_{\ep_j}$ and 
$x_{-\ep_r+\ep_j}$. We also compute
\begin{align*}
x_{-\ep_r+\ep_j}\(\frak{p}^{c_1+1}\)x_{\ep_r+\ep_\el}\(\frak{p}^{-m-1}\)
&\subset
x_{\ep_j+\ep_\el}\(\frak{p}^{-2c_2}\)x_{\ep_r+\ep_\el}\(\frak{p}^{-m-1}\)x_{-\ep_r+\ep_j}\(\frak{p}^{c_1+1}\),
\end{align*}
for $r<j<\el\le n$. Using these results, we deduce that 
\begin{align*}
P_{\a_r}(F)&w_{S,m+1}x_{r,c+1}
\prod_{\beta\in I_S^{(3)}}x_{\beta}\(\varpi^{-m-1}y_\beta\)
w_{S,m}K_{n,m}\\
&=
P_{\a_r}(F)w_{S,m+1}
\prod_{r<j\le n}x_{r,c+1}x_{\ep_r+\ep_j}\(\varpi^{-m-1}y_{\ep_r+\ep_j}\)x_{r,c+1}^{-1}
x_{r,c+1}w_{S,m}K_{n,m}\\
&\subset
P_{\a_r}(F)w_{S,m+1}
\prod_{r<j\le n}x_{\ep_r+\ep_j}\(\frak{p}^{-m-1}\)x_{\ep_j}\(\frak{p}^{-c_2}\)x_{-\ep_r+\ep_j}\(\frak{p}^{c_1+1}\)
x_{r,c+1}w_{S,m}K_{n,m}\\
&\subset
P_{\a_r}(F)w_{S,m+1}
\prod_{r<j\le n} x_{\ep_r+\ep_j}\(\frak{p}^{-m-1}\)
\prod_{r<j\le n} x_{\ep_j}\(\frak{p}^{-c_2}\)
\prod_{r<j<\el\le n} x_{\ep_j+\ep_\el}\(\frak{p}^{-2c_2}\)\\
&\quad\quad\quad\quad\quad\quad\quad
\cdot\prod_{r<j\le n}x_{-\ep_r+\ep_j}\(\frak{p}^{c_1+1}\)
x_{r,c+1}w_{S,m}K_{n,m}. 
\end{align*}
In the above derivation, we use the fact that $x_{\ep_j+\ep_\el}$ commutes with all other root elements in the product. 

Since $j>r$, we have
\[
w_{S,m+1}x_{\ep_r+\ep_j}\(\frak{p}^{-m-1}\)w_{S,m+1}
\subset
x_{\ep_r\pm\ep_j}(F)\subset P_{\a_r}(F),
\]
and
\[
w_{S,m+1}x_{\ep_j}\(\frak{p}^{-c_2}\)w_{S,m+1}
\subset
x_{\pm\ep_j}(F)
\subset P_{\a_r}(F),
\]
where the sign $\pm$ in front of $\ep_j$ depends on whether or not $j\in S$. On the other hand, since $x_{-\ep_r+\ep_j}$ commutes with 
$x_{r,c+1}$, and 
\[
w_{S,m}^{-1}x_{-\ep_r+\ep_j}\(\frak{p}^{c_1+1}\)w_{S,m}
=
x_{-\ep_r+\ep_j}\(\frak{p}^{c_1+1}\)\subset K_{n,m}
\]
if $j\nin S$, while 
\[
w_{S,m}^{-1}x_{-\ep_r+\ep_j}\(\frak{p}^{c_1+1}\)w_{S,m}
=
x_{-\ep_r-\ep_j}\(\frak{p}^{m+c_1+1}\)\subset K_{n,m}
\]
if $j\in S$, we find that 
\begin{align*}
&P_{\a_r}(F)w_{S,m+1}
\prod_{r<j\le n} x_{\ep_r+\ep_j}\(\frak{p}^{-m-1}\)
\prod_{r<j\le n} x_{\ep_j}\(\frak{p}^{-c_2}\)
\prod_{r<j<\el\le n} x_{\ep_j+\ep_\el}\(\frak{p}^{-2c_2}\)\\
&\quad\quad\quad\quad\quad\quad
\cdot\prod_{r<j\le n}x_{-\ep_r+\ep_j}\(\frak{p}^{c_1+1}\)
x_{r,c+1}w_{S,m}K_{n,m}\\
&=
P_{\a_r}(F)w_{S,m+1}
\prod_{r<j<\el\le n}x_{\ep_j+\ep_\el}\(\frak{p}^{-2c_2}\)
x_{r,c+1}w_{S,m}K_{n,m}. 
\end{align*}
At this point, we obtain
\begin{align*}
P_{\a_r}(F)&
x_{r, c+1}w_{S,m+1}\prod_{\beta\in I_S}x_{\beta}\(\varpi^{-m-1}y_\beta\)w_{S,m}
K_{n,m}\\
&\subset
P_{\a_r}(F)w_{S,m+1}
\prod_{r<j<\el\le n}x_{\ep_j+\ep_\el}\(\frak{p}^{-2c_2}\)
x_{r,c+1}w_{S,m}K_{n,m}. 
\end{align*}

To complete the proof, let 
\[
J=\stt{\ep_j+\ep_\el\mid r<j<\el\le n}
=
J^{(1)}\sqcup J^{(2)},
\]
where
\[
J^{(1)}=\stt{\ep_j+\ep_\el\in J\mid j\nin S,\, \el\nin S}
\quad\text{and}\quad
J^{(2)}=J\setminus J^{(1)}. 
\]
Note that $w_{S,m}$ commutes with $x_{r,c+1}$ and with $x_\beta$ for every $\beta\in J^{(1)}$. Furthermore, for $\ep_j+\ep_\el\in J^{(2)}$, 
we have
\[
w_{S,m}^{-1}x_{\ep_j+\ep_\el}\(\frak{p}^{-2c_2}\)w_{S,m}
=
\begin{cases}
x_{\ep_j-\ep_\el}\(\frak{p}^{c_1}\)\quad&\text{if $j\nin S$ and $\el\in S$},\\
x_{-\ep_j+\ep_\el}\(\frak{p}^{c_1}\)\quad&\text{if $j\in S$ and $\el\nin S$},\\
x_{-\ep_j-\ep_\el}\(\frak{p}^{2c_1+2c_2}\)\quad&\text{if $j\in S$ and $\el\in S$}. 
\end{cases}
\]
In particular, if $\beta\in J^{(2)}$, then
\[
w_{S,m}^{-1}x_\beta\(\frak{p}^{-2c_2}\)w_{S,m}\subset K_{n,m}. 
\]
Since $x_{r,c+1}$ commutes with $x_\beta$ for every $\beta\in J$, and 
$x_\beta\(\frak{p}^{-2c_2}\)\subset P_{\a_r}(F)$ for each $\beta\in J^{(1)}$, it follows that   
\begin{align*}
P_{\a_r}(F)&w_{S,m+1}
\prod_{r<j<\el\le n}x_{\ep_j+\ep_\el}\(\frak{p}^{-2c_2}\)
x_{r,c+1}w_{S,m}K_{n,m}\\
&=
P_{\a_r}(F)w_{S,m+1}
\prod_{\beta\in J^{(1)}}x_{\beta}\(\frak{p}^{-2c_2}\)x_{r,c+1}w_{S,m}
w_{S,m}^{-1}
\prod_{\beta\in J^{(2)}}x_{\beta}\(\frak{p}^{-2c_2}\)
w_{S,m}
K_{n,m}\\
&=
P_{\a_r}(F)w_{S,m+1}w_{S,m}
\prod_{\beta\in J^{(1)}}x_{\beta}\(\frak{p}^{-2c_2}\)
x_{r,c+1}K_{n,m}\\
&=
P_{\a_r}(F)\varpi^{-\la_S}
\prod_{\beta\in J^{(1)}}x_{\beta}\(\frak{p}^{-2c_2}\)
x_{r,c+1}K_{n,m}\\
&=
P_{\a_r}(F)x_{r,c+1}K_{n,m}. 
\end{align*}
To conclude, we have shown that 
\[
P_{\a_r}(F)
x_{r, c+1}w_{S,m+1}\prod_{\beta\in I_S}x_{\beta}\(\varpi^{-m-1}y_\beta\)w_{S,m}
K_{n,m}
\subset
P_{\a_r}(F)x_{r, c+1}K_{n,m},
\]
which implies that 
\[
P_{\a_r}(F)
x_{r, c+1}w_{S,m+1}\prod_{\beta\in I_S}x_{\beta}\(\varpi^{-m-1}y_\beta\)w_{S,m}
K_{n,m}
=
P_{\a_r}(F)x_{r, c+1}K_{n,m}. 
\]
This proves $(2)$ and hence completes the proof of the lemma. 
\end{proof}

\subsection{Computation of level-raising operators I}\label{SS:compute operator I}
The aim of this and the next subsections is to compute \eqref{E:seed case 2 main}. For this, we follow the setup in 
\S\ref{SS:setup for level-raising}. 
In addition, if $1\le n_0\le n$ is an integer, we regard $\SO_{2n_0+1}(F)$ as a subgroup of $\SO_{2n+1}(F)$ via the embedding
\[
\SO_{2n_0+1}(F)\longto\SO_{2n+1}(F);\quad g_0\longmapsto\diag{I_{n-n_0}, g_0, I_{n-n_0}}. 
\]
In particular, the root system of $\SO_{2n_0+1}$ can be identified with the set
\[
\stt{\pm\ep_i\pm\ep_j\mid n-n_0+1\le i<j\le n}\cup\stt{\pm\ep_\el\mid n-n_0+1\le \el\le n}. 
\]
We split the computation into the cases $r=1$ and $r>1$. In this subsection, we focus on the case $r=1$. The case $r>1$ will be computed in 
the next subsection. 

Assume that $r=1$ in this subsection. Since 
\[
a=a_{\Pi_s}=c_{\pi'}+2(r-1)=c_{\pi'}
\quad
\text{and}
\quad
\xi=c_{\pi'}+r-1=c_{\pi'}, 
\]
it follows that $x_{1,\xi+e}\in K_{n, a+e}$ for $e\in\stt{0,1}$. Hence, we have
\[
\theta(f_s)\(x_{1,\xi+1}\)
=
\theta(f_s)(I_{2n+1})
\quad
\text{and}
\quad
\theta'(f_s)\(x_{1,\xi+1}\)
=
\theta'(f_s)(I_{2n+1}). 
\]
Note that
\[
\supp(f_s)=P_{\a_1}(F)K_{n,a}
\quad
\text{and}
\quad
f_s(I_{2n+1})=v_{\pi'}\in\pi'^{K_{n-1, c_{\pi'}}}
\]
by \lmref{L:basis seed case 2}. 

We first compute $\theta(f_s)(I_{2n+1})$, which, by \propref{P:coset decomp} and \eqref{E:theta main id}, is given by
\[
\theta(f_s)(I_{2n+1})
=
\sum_{\substack{S\subset\cI_n\\\text{$|S|$ even}}}\,
\sum_{y_{S,\beta}\in\frak{o}/\frak{p}}
f_s\(
w_{S,a+1}\prod_{\beta\in I_S} x_\beta\(\varpi^{-a-1}y_{S,\beta}\)
\). 
\]
Note that if $\ep_i+\ep_j\in I_S$, then
\[
w_{S,a+1}\,x_{\ep_i+\ep_j}\(\varpi^{-a-1}y\)\,w_{S,a+1}^{-1}
=
\begin{cases}
x_{\ep_i+\ep_j}\(\varpi^{-a-1}y\)\quad&\text{if $j\nin S$},\\
x_{\ep_i-\ep_j}\(-y\)\quad&\text{if $j\in S$}. 
\end{cases}
\]
We thus deduce that
\[
\theta(f_s)\(I_{2n+1}\)
=
\sum_{\substack{S\subset\cI_n\\\text{$|S|$ even}}}\,
\sum_{y_{S,\beta}\in\frak{o}/\frak{p}}
f_s\(
\prod_{\beta\in I_S^-}x_\beta\(y_{S,\beta}\)
\prod_{\beta\in I^+_S} x_\beta\(\varpi^{-a-1}y_{S,\beta}\)w_{S,a+1}
\). 
\]
Now, let $S\subset\cI_n$ be any subset, and decompose $S=S_0\sqcup S_1$, where $S_0=S\cap\stt{1}$ and 
$S_1=S\setminus S_0\subset\stt{2,\ldots, n}$. Set
\[
I_S^\pm=I_{S,1,n}^\pm
\quad\text{and}\quad
I_{S_1}^\pm=I_{S_1,2,n}^\pm.  
\]
Note that 
\[
x_\beta(F)\subset N_{\a_1}(F)
\Longleftrightarrow
\beta\in 
\(I^+_S\setminus I^+_{S_1}\)
\cup
\(I^-_S\setminus I^-_{S_1}\),
\]
and 
\begin{equation}\label{E:no of root in N r=1}
\text{$|I^+_S\setminus I^+_{S_1}|+|I^-_S\setminus I^-_{S_1}|=n-1$ or $0$}
\Longleftrightarrow
\text{$S_0=\emptyset$ or $S_0=\stt{1}$.} 
\end{equation}

If $|S|$ is even and $S_0=\emptyset$, then $w_{S,a+1}w_{S, a}=\varpi^{-\la_S}=\varpi^{-\la_{S_1}}$, and hence 
\begin{align*}
f_s\(
\prod_{\beta\in I_S^-}x_\beta\(y_{S,\beta}\)
\prod_{\beta\in I^+_S} x_\beta\(\varpi^{-a-1}y_{S,\beta}\)w_{S,a+1}
\)
&=
f_s\(
\prod_{\beta\in I_S^-}x_\beta\(y_{S,\beta}\)
\prod_{\beta\in I^+_S} x_\beta\(\varpi^{-a-1}y_{S,\beta}\)w_{S,a+1}w_{S,a}
\)\\
&=
f_s\(
\prod_{\beta\in I_{S_1}^-}x_\beta\(y_{S,\beta}\)
\prod_{\beta\in I^+_{S_1}} x_\beta\(\varpi^{-a-1}y_{S,\beta}\)\varpi^{-\la_{S_1}}
\)\\
&=
\pi'\(
\prod_{\beta\in I_{S_1}^-}x_\beta\(y_{S,\beta}\)
\prod_{\beta\in I^+_{S_1}} x_\beta\(\varpi^{-a-1}y_{S,\beta}\)\varpi^{-\la_{S_1}}
\)v_{\pi'}. 
\end{align*}
Here we use the fact that $w_{S,a}\in K_{n,a}$ when $|S|$ is even, and we regard $\varpi^{-\la_S}=\varpi^{-\la_{S_1}}$ as an element of 
$\SO_{2n-1}(F)$. On the other hand, if $|S|$ is even and $S_0=\stt{1}$, then 
$\varpi^{-\la_S}=\varpi^{-\ep^*_1}\varpi^{-\la_{S_1}}$, $I_{S}^\pm=I_{S_1}^\pm$, and
$\varpi^{-\ep^*_1}$ commutes with $x_\beta(y)$ for every $\beta\in I_{S_1}^+\cup I_{S_1}^-$. It follows that
\begin{align*}
f_s&\(
\prod_{\beta\in I_S^-}x_\beta\(y_{S,\beta}\)
\prod_{\beta\in I^+_S} x_\beta\(\varpi^{-a-1}y_{S,\beta}\)w_{S,a+1}
\)\\
&\quad\quad\quad\quad\quad\quad=
f_s\(
\prod_{\beta\in I_S^-}x_\beta\(y_{S,\beta}\)
\prod_{\beta\in I^+_S} x_\beta\(\varpi^{-a-1}y_{S,\beta}\)\varpi^{-\la_{S}}
\)\\
&\quad\quad\quad\quad\quad\quad=
f_s\(\varpi^{-\ep^*_1}
\prod_{\beta\in I_{S_1}^-}x_\beta\(y_{S,\beta}\)
\prod_{\beta\in I^+_{S_1}} x_\beta\(\varpi^{-a-1}y_{S,\beta}\)\varpi^{-\la_{S_1}}
\)\\
&\quad\quad\quad\quad\quad\quad=
\chi(\varpi)q^{s+n-\frac{1}{2}}\pi'\(
\prod_{\beta\in I_{S_1}^-}x_\beta\(y_{S,\beta}\)
\prod_{\beta\in I^+_{S_1}} x_\beta\(\varpi^{-a-1}y_{S,\beta}\)\varpi^{-\la_{S_1}}
\)v_{\pi'}. 
\end{align*}
Combining with \eqref{E:no of root in N r=1}, we obtain
\begin{align}\label{E:theta r=1}
\begin{split}
\theta(f_s)&(I_{2n+1})\\
=&
q^{n-1}\sum_{\substack{S_1\subset\stt{2,\ldots, n}\\\text{$|S_1|$ even}}}\,
\sum_{y_{S_1,\beta}\in\frak{o}/\frak{p}}
\pi'\(
\prod_{\beta\in I_{S_1}^-}x_\beta\(y_{S_1,\beta}\)
\prod_{\beta\in I^+_{S_1}} x_\beta\(\varpi^{-a-1}y_{S_1,\beta}\)\varpi^{-\la_{S_1}}
\)v_{\pi'}\\
&+
 \chi(\varpi)q^{s+n-\frac{1}{2}}
\sum_{\substack{S_1\subset\stt{2,\ldots, n}\\\text{$|S_1|$ odd}}}\,
\sum_{y_{S_1,\beta}\in\frak{o}/\frak{p}}
\pi'\(
\prod_{\beta\in I_{S_1}^-}x_\beta\(y_{S_1,\beta}\)
\prod_{\beta\in I^+_{S_1}} x_\beta\(\varpi^{-a-1}y_{S_1,\beta}\)\varpi^{-\la_{S_1}}
\)v_{\pi'}. 
\end{split}
\end{align}

We next compute $\theta'(f_s)(I_{2n+1})$, which, by \propref{P:coset decomp} and \eqref{E:theta' main id}, is given by
\begin{align*}
\theta'(f_s)(I_{2n+1})
=
\sum_{\substack{S\subset\cI_n\\\text{$|S|$ even}}}\,
\sum_{y_{S,\beta}\in\frak{o}/\frak{p}}
f_s\(
w_{\ep_n, a+1}w_{S,a+1}
\prod_{\beta\in I_S} x_\beta\(\varpi^{-a-1}y_{S,\beta}\)w_{\ep_n,a}
\). 
\end{align*}
For a given $S\subset\cI_n$, we define $S'\subset\cI_n$ by $S'=S\setminus\stt{n}$ if $n\in S$, and $S'=S\cup\stt{n}$ if $n\nin S$. 
Note that 
\[
w_{\ep_n, m}w_{S,m}=w_{S',m}
\quad\text{and}\quad
I_{S}=I_{S'},
\] 
where $m$ is an integer. It follows that 
\begin{align*}
\theta'(f_s)\(I_{2n+1}\)
&=
\sum_{\substack{S\subset\cI_n\\\text{$|S|$ even}}}\,
\sum_{y_{S,\beta}\in\frak{o}/\frak{p}}
f_s\(
w_{\ep_n, a+1}w_{S,a+1}
\prod_{\beta\in I_S} x_\beta\(\varpi^{-a-1}y_{S,\beta}\)w_{\ep_n,a}
\)\\
&=
\sum_{\substack{S\subset\cI_n\\\text{$|S|$ even}}}\,
\sum_{y_{S,\beta}\in\frak{o}/\frak{p}}
f_s\(
w_{\ep_n, a+1}w_{S,a+1}
\prod_{\beta\in I_S} x_\beta\(\varpi^{-a-1}y_{S,\beta}\)w_{\ep_n,a}w_{S,a}
\)\\
&=
\sum_{\substack{S'\subset\cI_n\\\text{$|S'|$ odd}}}\,
\sum_{y_{S',\beta}\in\frak{o}/\frak{p}}
f_s\(
w_{S',a+1}
\prod_{\beta\in I_{S'}} x_\beta\(\varpi^{-a-1}y_{S',\beta}\)w_{S',a}
\). 
\end{align*}
By a similar computation as in the case of $\theta(f_s)(I_{2n+1})$, we deduce that 
\begin{align}\label{E:theta' r=1}
\begin{split}
\theta'&(f_s)(I_{2n+1})\\
&=
q^{n-1}
\sum_{\substack{S_1\subset\stt{2,\ldots, n}\\\text{$|S_1|$ odd}}}\,
\sum_{y_{S_1,\beta}\in\frak{o}/\frak{p}}
\pi'\(
\prod_{\beta\in I_{S_1}^-}x_\beta\(y_{S_1,\beta}\)
\prod_{\beta\in I^+_{S_1}} x_\beta\(\varpi^{-a-1}y_{S_1,\beta}\)\varpi^{-\la_{S_1}}
\)v_{\pi'}\\
&\quad+
\chi(\varpi)q^{s+n-\frac{1}{2}}
\sum_{\substack{S_1\subset\stt{2,\ldots, n}\\\text{$|S_1|$ even}}}\,
\sum_{y_{S_1,\beta}\in\frak{o}/\frak{p}}
\pi'\(
\prod_{\beta\in I_{S_1}^-}x_\beta\(y_{S_1,\beta}\)
\prod_{\beta\in I^+_{S_1}} x_\beta\(\varpi^{-a-1}y_{S_1,\beta}\)\varpi^{-\la_{S_1}}
\)v_{\pi'}. 
\end{split}
\end{align}
This completes the computation for the case $r=1$. 

\subsection{Computation of level-raising operators II}\label{SS:compute operator II}
In this subsection, we assume that $r>1$. Recall that 
\[
a=a_{\Pi_s}=c_{\pi'}+2(r-1)
\quad
\text{and}
\quad
\xi=c_{\pi'}+r-1. 
\]
On the other hand, by \lmref{L:basis seed case 2}, we have 
\[
\supp(f_s)
=P_{\a_r}(F)\,x_{r,\xi}\,K_{n,a}
\quad\text{and}\quad
f_s(x_{r,\xi})
=
v_\tau\ot v'_{\pi'},
\]
where $v_\tau$ is a newform of $\tau$, and $v'_{\pi'}$ is a basis vector of the one-dimensional space $\pi'^{K'_{n-r, c_{\pi'}}}$, with
\[
K'_{n-r, c_{\pi'}}=\varpi^{-(r-1)\la_{n-r}}K_{n-r, c_{\pi'}}\varpi^{(r-1)\la_{n-r}}. 
\]
When $r=n$, we understand that $v'_{\pi'}$ is a non-zero complex number. 

By \propref{P:coset decomp} and \eqref{E:theta main id}, we have
\begin{align*}
\theta(f_s)\(x_{r,\xi+1}\)
=
\sum_{\substack{S\subset\cI_n\\\text{$|S|$ even}}}\,
\sum_{y_{S,\beta}\in\frak{o}/\frak{p}}
f_s\(
x_{r,\xi+1}w_{S,a+1}
\prod_{\beta\in I_S} x_\beta\(\varpi^{-a-1}y_{S,\beta}\)
\). 
\end{align*}
We now analyze each term in the sum. If $r\in S$, then \lmref{L:eva} implies that
\begin{align*}
f_s\(
x_{r,\xi+1}w_{S,a+1}
\prod_{\beta\in I_S} x_\beta\(\varpi^{-a-1}y_{S,\beta}\)w_{S,a+1}
\)
&=
f_s\(
x_{r,\xi+1}w_{S,a+1}
\prod_{\beta\in I_S} x_\beta\(\varpi^{-a-1}y_{S,\beta}\)w_{S,a}
\)\\
&= 
f_s\(
\prod_{\beta\in I_S^-}x_\beta\(y'_{S,\beta}\)
\prod_{\beta\in I^+_S} x_\beta\(\varpi^{-a-1}y_{S,\beta}\)\varpi^{-\la_S}x_{r,\xi}
\),
\end{align*}
where $y'_{S, \ep_i-\ep_j}=-y_{S, \ep_i+\ep_j}$ for $\ep_i-\ep_j\in I_S^-$. If $r\nin S$, then since 
\[
\(P_{\a_r}(F)\,x_{r,\xi+1}\,K_{n,a}\)
\cap
\(P_{\a_r}(F)\,x_{r,\xi}\,K_{n,a}\)
=
\emptyset
\]
by \eqref{E:coset rep for std K when r<n} and \eqref{E:coset rep for std K when r=n}, the same lemma implies that 
\begin{align*}
f_s\(
x_{r,\xi+1}w_{S,a+1}
\prod_{\beta\in I_S} x_\beta\(\varpi^{-a-1}y_{S,\beta}\)
\)
=
f_s\(
x_{r,\xi+1}w_{S,a+1}
\prod_{\beta\in I_S} x_\beta\(\varpi^{-a-1}y_{S,\beta}\)w_{S,a}
\)
=
0. 
\end{align*}
Combining these, we deduce that
\begin{align*}
\theta(f_s)\(x_{r,\xi+1}\)
=
\sum_{\substack{r\in S\subset\cI_n\\\text{$|S|$ even}}}\,
\sum_{y_{S,\beta}\in\frak{o}/\frak{p}}
f_s\(
\prod_{\beta\in I_S^-}x_\beta\(y_{S,\beta}\)
\prod_{\beta\in I^+_S} x_\beta\(\varpi^{-a-1}y_{S,\beta}\)\varpi^{-\la_S}x_{r,\xi}
\). 
\end{align*}

At this point, let $S\subset\cI_n$ be any subset, and decompose $S=S_0\sqcup S_1$,
where $S_0=S\cap\cI_r$ and $S_1=S\setminus S_0\subset\stt{r+1,\ldots, n}$. Let 
\[
I_S^\pm=I_{S,1,n}^\pm,\quad I_{S_0}^\pm=I_{S_0,1,r}^\pm,\quad
\text{and}\quad I_{S_1}^\pm=I_{S_1,r+1, n}^\pm. 
\]
Note that 
\[
I^+_S
=
I_{S_0}^+\sqcup I_{S_1}^+\sqcup\stt{\ep_i+\ep_j\mid 1\le i\le r<j\le n,\,i\nin S_0,\,j\nin S_1},
\]
and 
\[
I^-_S
=
I_{S_0}^-\sqcup I_{S_1}^-\sqcup\stt{\ep_i-\ep_j\mid 1\le i\le r<j\le n,\,i\nin S_0,\,j\in S_1}. 
\]
Note also that 
\[
\text{$x_\beta(F)\subset N_{\a_r}(F)$}
\Longleftrightarrow
\text{$\beta\in I_S^+\setminus I^+_{S_1}$ or $\beta\in I^-_S\setminus\(I_{S_0}^-\cup I_{S_1}^-\)$,}
\]
and 
\begin{equation}\label{E:no of roots}
|I_S^+\setminus I^+_{S_1}|+|I^-_S\setminus\(I_{S_0}^-\cup I_{S_1}^-\)|
=
\tfrac{1}{2}\(r-|S_0|\)\(2n-r-|S_0|-1\). 
\end{equation}

We then decompose this sum as follows: 
\begin{align*}
\theta(f_s)\(x_{r,\xi+1}\)
&=
\sum_{\substack{r\in S\subset\cI_n\\\text{$|S|$ even}}}\,
\sum_{y_{S,\beta}\in\frak{o}/\frak{p}}
f_s\(
\prod_{\beta\in I_S^-}x_\beta\(y_{S,\beta}\)
\prod_{\beta\in I^+_S} x_\beta\(\varpi^{-a-1}y_{S,\beta}\)\varpi^{-\la_S}x_{r,\xi}
\)\\
&=
\sum_{\substack{S=S_0\sqcup S_1\\r\in S_0\subsetneq\cI_r,\,\text{$|S_0|$ even}}}\,
\sum_{y_{S,\beta}\in\frak{o}/\frak{p}}
f_s\(
\prod_{\beta\in I_S^-}x_\beta\(y_{S,\beta}\)
\prod_{\beta\in I^+_S} x_\beta\(\varpi^{-a-1}y_{S,\beta}\)\varpi^{-\la_S}x_{r,\xi}
\)\\
&\quad+
\sum_{\substack{S=S_0\sqcup S_1\\r\in S_0\subsetneq\cI_r,\,\text{$|S_0|$ odd}}}\,
\sum_{y_{S,\beta}\in\frak{o}/\frak{p}}
f_s\(
\prod_{\beta\in I_S^-}x_\beta\(y_{S,\beta}\)
\prod_{\beta\in I^+_S} x_\beta\(\varpi^{-a-1}y_{S,\beta}\)\varpi^{-\la_S}x_{r,\xi}
\)\\
&\quad\quad+
\sum_{\substack{S=S_0\sqcup S_1\\ S_0=\cI_r}}\,
\sum_{y_{S,\beta}\in\frak{o}/\frak{p}}
f_s\(
\prod_{\beta\in I_S^-}x_\beta\(y_{S,\beta}\)
\prod_{\beta\in I^+_S} x_\beta\(\varpi^{-a-1}y_{S,\beta}\)\varpi^{-\la_S}x_{r,\xi}
\). 
\end{align*}
We proceed to analyze each term in the sum. If $S=S_0\sqcup S_1$ with $r\in S_0\subsetneq\cI_r$, then 
\begin{align*}
f_s&\(
\prod_{\beta\in I_S^-}x_\beta\(y_{S,\beta}\)
\prod_{\beta\in I^+_S} x_\beta\(\varpi^{-a-1}y_{S,\beta}\)\varpi^{-\la_S}x_{r,\xi}
\)\\
&\quad\quad\quad=
q^{|S_0|\(s+n-\tfrac{r}{2}\)}\tau\(
\prod_{\beta\in I^-_{S_0}}\chi_\beta\(y_{S_0,\beta}\)\varpi^{-\nu_{S_0}}
\)v_\tau\\
&\quad\quad\quad\quad\quad\quad\ot
\pi'\(
\prod_{\beta\in I^-_{S_1}}x_\beta\(y_{S_1,\beta}\)
\prod_{\beta\in I^+_{S_1}}x_\beta\(\varpi^{-a-1}y_{S_1,\beta}\)\varpi^{-\la_{S_1}}
\)v'_{\pi'}. 
\end{align*}
On the other hand, if $S=S_0\sqcup S_1$ with $S_0=\cI_r$, then since $\tau$ has central character $\chi^r$,
\begin{align*}
f_s&\(
\prod_{\beta\in I_S^-}x_\beta\(y_{S,\beta}\)
\prod_{\beta\in I^+_S} x_\beta\(\varpi^{-a-1}y_{S,\beta}\)\varpi^{-\la_S}x_{r,\xi}
\)\\
&=
\chi(\varpi)^rq^{r\(s+n-\tfrac{r}{2}\)}\,v_\tau
\ot
\pi'\(
\prod_{\beta\in I^-_{S_1}}x_\beta\(y_{S_1,\beta}\)
\prod_{\beta\in I^+_{S_1}}x_\beta\(\varpi^{-a-1}y_{S_1,\beta}\)\varpi^{-\la_{S_1}}
\)v'_{\pi'}. 
\end{align*}
Here, we regard $\GL_r(F)$ as a subgroup of $\SO_{2n+1}(F)$ via the embedding
\[
a\longmapsto\diag{a, I_{2(n-r)+1}, a^*}
\]
for $a\in\GL_r(F)$, and we follow the notation introduced in \S\ref{SS:Hecke GL}. In the following, for simplicity, we write
\begin{equation}\label{E:u(S_1)}
u\(S_1,\stt{y_{S_1,\beta}}_\beta\)
=
\prod_{\beta\in I^-_{S_1}}x_\beta\(y_{S_1,\beta}\)
\prod_{\beta\in I^+_{S_1}}x_\beta\(\varpi^{-a-1}y_{S_1,\beta}\)\varpi^{-\la_{S_1}},
\end{equation}
where $S_1\subset\stt{r+1,\ldots, n}$ and $y_{S_1,\beta}\in\frak{o}/\frak{p}$ for each $\beta\in I_{S_1}^+\cup I_{S_1}^-$. 

By \eqref{E:no of roots} and the above analysis, we obtain
\begin{align*}
\theta&(f_s)\(x_{r,\xi+1}\)\\
&=
\sum_{\substack{S_1\subset\stt{r+1,\ldots, n}\\\text{$|S_1|$ even}}}\,
\sum_{\substack{1<\el< r\\\text{$\el$ even}}}\,
q^{\tfrac{1}{2}\left[2\el\(s+n-\tfrac{r}{2}\)+(r-\el)(2n-r-\el-1)\right]}\\
&\quad\quad\left[
\sum_{\substack{r\in S_0\subsetneq\cI_r\\ |S_0|=\el}}\,
\sum_{y_{S_0,\beta}\in\frak{o}/\frak{p}}
\tau\(
\prod_{\beta\in I^-_{S_0}}\chi_\beta\(y_{S_0,\beta}\)
\varpi^{-\nu_{S_0}}
\)v_\tau\right]
\ot
\left[
\sum_{y_{S_1,\beta}\in\frak{o}/\frak{p}}\pi'\(u\(S_1,\stt{y_{S_1,\beta}}_\beta\)\)v'_{\pi'}
\right]\\
&\quad+
\sum_{\substack{S_1\subset\stt{r+1,\ldots, n}\\\text{$|S_1|$ odd}}}\,
\sum_{\substack{1\le\el<r\\\text{$\el$ odd}}}\,
q^{\tfrac{1}{2}\left[2\el\(s+n-\tfrac{r}{2}\)+(r-\el)(2n-r-\el-1)\right]}\\
&\quad\quad\quad\left[
\sum_{\substack{r\in S_0\subsetneq\cI_r\\ |S_0|=\el}}\,
\sum_{y_{S_0,\beta}\in\frak{o}/\frak{p}}
\tau\(
\prod_{\beta\in I^-_{S_0}}\chi_\beta\(y_{S_0,\beta}\)
\varpi^{-\nu_{S_0}}
\)v_\tau\right]
\ot
\left[
\sum_{y_{S_1,\beta}\in\frak{o}/\frak{p}}\pi'\(u\(S_1,\stt{y_{S_1,\beta}}_\beta\)\)v'_{\pi'}
\right]\\
&\quad\quad+
\chi(\varpi)^rq^{r\(s+n-\tfrac{r}{2}\)}\,
v_\tau
\ot
\left[
\sum_{\substack{S_1\subset\stt{r+1,\ldots,n}\\\text{$|S_1|+r$ even}}}\,
\sum_{y_{S_1,\beta}\in\frak{o}/\frak{p}}\pi'\(u\(S_1,\stt{y_{S_1,\beta}}_\beta\)\)v'_{\pi'}
\right]. 
\end{align*}
At this point, let $S'_0=\cI_r\setminus S_0$. Then we have 
\[
I_{S_0}^-=J_{S'_0}
\quad\text{and}\quad
\varpi^{\nu_r-\nu_{S_0}}
=
\varpi^{\nu_{S'_0}},
\]
where $J_{S'_0}$ is the set defined in \eqref{E:J_S}. Since $\tau$ has central character $\chi^r$, it follows from \lmref{L:Hecke op} 
and \eqref{E:Hecke ev} that
\begin{align*}
\sum_{\substack{r\in S_0\subsetneq\cI_r\\|S_0|=\el}}\,
\sum_{y_{S_0,\beta}\in\frak{o}/\frak{p}}
\tau\(
\prod_{\beta\in I^-_{S_0}}\chi_\beta\(y_{S_0,\beta}\)
\varpi^{-\nu_{S_0}}
\)v_\tau
&=
\sum_{\substack{r\in S_0\subsetneq\cI_r\\|S_0|=\el}}\,
\sum_{y_{S_0,\beta}\in\frak{o}/\frak{p}}
\chi(\varpi)^r\tau\(
\prod_{\beta\in I^-_{S_0}}\chi_\beta\(y_{S_0,\beta}\)
\varpi^{\nu_r-\nu_{S_0}}
\)v_\tau\\
&=
\chi(\varpi)^r\sum_{\substack{S'_0\subset\cI_{r-1}\\|S'_0|=r-\el}}\,
\sum_{y_{S'_0,\beta}\in\frak{o}/\frak{p}}
\tau\(
\prod_{\beta\in J_{S'_0}}\chi_\beta\(y_{S'_0,\beta}\)
\varpi^{\nu_{S'_0}}
\)v_\tau\\
&=
\chi(\varpi)^rT_{r-\el}\(v_\tau\)\\
&=
\chi(\varpi^r)\la_{\tau, r-\el}\,v_\tau. 
\end{align*}
Since $c_\tau=r-1\ge 1$ and
\[
L\(s,\phi_\tau\)
=
L\(s,\phi_\chi\bt S_{r}\)
=
\(1-\chi(\varpi)q^{-\(s+\tfrac{r-1}{2}\)}\)^{-1}
\]
by \lmref{L:cond}, it follows from \lmref{L:Hecke ev} that
\[
\la_{\tau,0}=1,\,\la_{\tau,1}=\chi(\varpi),\,\text{and $\la_{\tau,i}=0$ for $2\le i\le r-1$}. 
\] 
Combining these, we deduce that 
\begin{align}\label{E:theta r>1}
\begin{split}
\theta&(f_s)\(x_{r,\xi+1}\)\\
&=
\chi(\varpi)^{r+1}q^{(r-1)\(s+n-\tfrac{r}{2}\)+(n-r)}\,
v_\tau
\ot
\left[
\sum_{\substack{S_1\subset\stt{r+1,\ldots, n}\\\text{$|S_1|+r-1$ even}}}\,
\sum_{y_{S_1,\beta}\in\frak{o}/\frak{p}}\pi'\(u\(S_1,\stt{y_{S_1,\beta}}_\beta\)\)v'_{\pi'}
\right]\\
&\quad+
\chi(\varpi)^rq^{r\(s+n-\tfrac{r}{2}\)}\,
v_\tau
\ot
\left[
\sum_{\substack{S_1\subset\stt{r+1,\ldots,n}\\\text{$|S_1|+r$ even}}}\,
\sum_{y_{S_1,\beta}\in\frak{o}/\frak{p}}\pi'\(u\(S_1,\stt{y_{S_1,\beta}}_\beta\)\)v'_{\pi'}
\right]. 
\end{split}
\end{align}

We next compute $\theta'(f_s)\(x_{r,\xi+1}\)$, which, by \propref{P:coset decomp} and \eqref{E:theta' main id} is given by 
\begin{align*}
\theta'(f_s)\(x_{r,\xi+1}\)
=
\sum_{\substack{S\subset\cI_n\\\text{$|S|$ even}}}\,
\sum_{y_{S,\beta}\in\frak{o}/\frak{p}}
f_s\(
x_{r,\xi+1}w_{\ep_n, a+1}w_{S,a+1}
\prod_{\beta\in I_S} x_\beta\(\varpi^{-a-1}y_{S,\beta}\)w_{\ep_n,a}
\). 
\end{align*}
For a given $S\subset\cI_n$, we define $S'\subset\cI_n$ by $S'=S\setminus\stt{n}$ if $n\in S$, and $S'=S\cup\stt{n}$ if $n\nin S$. 
Note that 
\[
w_{\ep_n, m}w_{S,m}=w_{S',m}
\quad\text{and}\quad
I_{S}=I_{S'},
\] 
where $m$ is an integer. It follows that 
\begin{align*}
\theta'(f_s)\(x_{r,\xi+1}\)
&=
\sum_{\substack{S\subset\cI_n\\\text{$|S|$ even}}}\,
\sum_{y_{S,\beta}\in\frak{o}/\frak{p}}
f_s\(
x_{r,\xi+1}w_{\ep_n, a+1}w_{S,a+1}
\prod_{\beta\in I_S} x_\beta\(\varpi^{-a-1}y_{S,\beta}\)w_{\ep_n,a}
\)\\
&=
\sum_{\substack{S\subset\cI_n\\\text{$|S|$ even}}}\,
\sum_{y_{S,\beta}\in\frak{o}/\frak{p}}
f_s\(
x_{r,\xi+1}w_{\ep_n, a+1}w_{S,a+1}
\prod_{\beta\in I_S} x_\beta\(\varpi^{-a-1}y_{S,\beta}\)w_{\ep_n,a}w_{S,a}
\)\\
&=
\sum_{\substack{S'\subset\cI_n\\\text{$|S'|$ odd}}}\,
\sum_{y_{S',\beta}\in\frak{o}/\frak{p}}
f_s\(
x_{r,\xi+1}w_{S',a+1}
\prod_{\beta\in I_{S'}} x_\beta\(\varpi^{-a-1}y_{S',\beta}\)w_{S',a}
\). 
\end{align*}
At this point, we can apply a similar computation to that for $\theta(f_s)\(x_{r,\xi+1}\)$ to obtain 
\begin{align}\label{E:theta' r>1}
\begin{split}
\theta'&(f_s)\(x_{r,\xi+1}\)\\
&=
\chi(\varpi)^{r+1}q^{(r-1)\(s+n-\tfrac{r}{2}\)+(n-r)}\,
v_\tau
\ot
\left[
\sum_{\substack{S_1\subset\stt{r+1,\ldots, n}\\\text{$|S_1|+r-1$ odd}}}\,
\sum_{y_{S_1,\beta}\in\frak{o}/\frak{p}}\pi'\(u\(S_1,\stt{y_{S_1,\beta}}_\beta\)\)v'_{\pi'}
\right]\\
&\quad+
\chi(\varpi)^rq^{r\(s+n-\tfrac{r}{2}\)}\,
v_\tau
\ot
\left[
\sum_{\substack{S_1\subset\stt{r+1,\ldots,n}\\\text{$|S_1|+r$ odd}}}\,
\sum_{y_{S_1,\beta}\in\frak{o}/\frak{p}}\pi'\(u\(S_1,\stt{y_{S_1,\beta}}_\beta\)\)v'_{\pi'}
\right]. 
\end{split}
\end{align}
This completes the computation for the case $r>1$. 

\begin{remark}
Observe that if we formally let $r=1$, then equation \eqref{E:theta r>1} (resp. \eqref{E:theta' r>1}) recovers \eqref{E:theta r=1} 
(resp. \eqref{E:theta' r=1}). Therefore, we may summarize our computations in this and the previous subsection by \eqref{E:theta r>1} 
and \eqref{E:theta' r>1}, allowing $r=1$ in these equations. 
\end{remark}

\subsection{Proof of the case $|I_{\phi,\chi}|=1$ or $|I_{\phi,\chi'}|=1$}
We follow the setup in \S\ref{SS:seed case 2} and \S\ref{SS:setup for level-raising}. Therefore, $\pi$ is a seed representation of 
$\SO_{2n+1}(F)$ with $L$-parameter 
\[
\phi=\phi_\chi\bt S_{2\ka+1}\oplus\phi'
\]
for some $\ka\in\tfrac{1}{2}+\bbZ_+$, and a discrete $L$-parameter $\phi'$ (possibly empty) of $\SO_{2n-2\ka}(F)$. 
Let $r=\ka+\tfrac{1}{2}$, and let $\pi'$ be the irreducible generic representation of $\SO_{2n-2\ka}(F)=\SO_{2(n-r)+1}(F)$ with 
$L$-parameter $\phi'$. Then $\pi'$ is a seed representation, and our assumption implies that 
\[
|I_{\phi',\chi}|=0.
\] 

Let $\tau=\St_\chi\(\tfrac{r-1}{2}\)$ and set
\[
\Pi_s
=
\tau|\cdot|^s\rtimes\pi',
\]
where $s\in\bbC$.  Let $\Pi=\Pi_{\frac{r}{2}}$ and $\Pi'=\Pi_{-\frac{r}{2}}$. Then $\pi$ is the unique irreducible generic 
subrepresentation of $\Pi$, and we have an intertwining map
\[
M:\Pi\longto\Pi'
\]
whose kernel is $\pi$, by \lmref{L:length 2} and the fact that $\Pi$ is a standard module. 

Assume first that $|I_{\phi,\chi'}|=|I_{\phi',\chi'}|=0$. Then, by \corref{C:sc to seed case 1} and our assumption on supercuspidal 
representations, we have $a_{\pi'}=c_{\pi'}$, and hence $a=a_\Pi=c_\pi-1$ by \lmref{L:seed cond case 2}. 
Now \lmref{L:basis seed case 2} implies that 
\[
\Pi_s^{K_{n,a}}
=
\bbC\,f_s
\]
is one-dimensional for every $s$. Since $M$ is an intertwining map, it follows that 
\[
M\(f_{\frac{r}{2}}\)\in\Pi'^{K_{n,a}}. 
\]
Note that $M\(f_{\frac{r}{2}}\)\ne 0$. Otherwise, we would have $f_{\frac{r}{2}}\in\pi^{K_{n,c_\pi-1}}$, contradicting \thmref{T:mainA} (1). 
We may assume that 
\[
M\(f_{\frac{r}{2}}\)
=
f_{-\frac{r}{2}}
\]
after choosing local newforms in $\tau$ and $\pi'$ properly. 

Let $\theta=\theta_a$ and $\theta'=\theta'_a$ be the level-raising operators introduced in \S\ref{SS:level-raising operator}. We have 
\[
\theta\(f_{\frac{r}{2}}\),\,\theta'\(f_{\frac{r}{2}}\)\in\Pi^{K_{n,c_\pi}},
\]
and
\[
M\(c\,\theta\(f_{\frac{r}{2}}\)+c'\theta'\(f_{\frac{r}{2}}\)\)
=
c\,\theta\(M\(f_{\frac{r}{2}}\)\)+c'\theta'\(M\(f_{\frac{r}{2}}\)\)
=
c\,\theta\(f_{-\frac{r}{2}}\)+c'\theta'\(f_{-\frac{r}{2}}\)
\]
for every $c, c'\in \bbC$ by \eqref{E:theta sum} and \eqref{E:theta' sum}. Now, by \lmref{L:basis seed case 2}, \eqref{E:theta when r=n} and 
\eqref{E:theta' when r=n}, together with \eqref{E:theta r>1} and \eqref{E:theta' r>1}, we deduce that
\[
\theta\(f_{-\frac{r}{2}}\)-\chi(\varpi)\theta'\(f_{-\frac{r}{2}}\)=0. 
\]
This implies that 
\[
\theta\(f_{\frac{r}{2}}\)-\chi(\varpi)\theta'\(f_{\frac{r}{2}}\)\in\pi^{K_{n,c_\pi}}. 
\]
It remains to show that 
\[
\theta\(f_{\frac{r}{2}}\)-\chi(\varpi)\theta'\(f_{\frac{r}{2}}\)\ne 0. 
\]
To this end, consider a Whittaker functional $\la$ on $\tau\ot\pi'$ defined as follows. Recall that $\psi$ is an additive character on $F$ that is 
trivial on $\frak{o}$ but non-trivial on $\frak{p}^{-1}$, and that there are non-degenerate characters $\psi_{U_{n-r}}$ and $\psi_{Z_r}$
of $U_{n-r}(F)$ and $Z_r(F)$, respectively, defined in \S\ref{SS:2.3}. Let $\la_{\tau,\psi}\in{\rm Hom}_{Z_r}\(\tau,\psi_{Z_r}\)$ and 
$\la_{\pi',\psi}\in{\rm Hom}_{U_{n-r}(F)}\(\pi',\psi_{U_{n-r}}\)$ be Whittaker functionals. Then the functional $\la$ is defined by 
\[
\la=\la_{\tau,\psi}\ot\la_{\pi',\psi}\circ\pi'(\varpi^{(r-1)\la_{n-r}}). 
\]

By \cite{Matringe2013} and \thmref{T:mainA} (3), we have
\[
\la(v_\tau\ot v'_{\pi'})
=
\la_{\tau,\psi}(v_\tau)\la_{\pi',\psi}\(\pi\(\varpi^{(r-1)\la_{n-r}}\)v'_{\pi'}\)
=
\la_{\tau,\psi}(v_\tau)\la_{\pi',\psi}(v_{\pi'})\ne 0,
\]
where $v_{\pi'}=\pi'\(\varpi^{(r-1)\la_{n-r}}\)v'_{\pi'}$ is a newform of $\pi'$. To prove that 
\[
\theta\(f_{\frac{r}{2}}\)-\chi(\varpi)\theta'\(f_{\frac{r}{2}}\)\ne 0,
\]
it suffices to show that
\[
\la\(\theta\(f_{\frac{r}{2}}\)\(x_{r,\xi+1}\)-\chi(\varpi)\theta'\(f_{\frac{r}{2}}\)\(x_{r,\xi+1}\)\)\ne 0. 
\]
For this purpose, we evaluate $\la$ at each term appearing in $\theta\(f_\frac{r}{2}\)\(x_{r,\xi+1}\)$ and 
$\theta'\(f_\frac{r}{2}\)\(x_{r,\xi+1}\)$ in \eqref{E:theta r>1} and \eqref{E:theta' r>1}, respectively. By the same equations, 
$\theta\(f_{\frac{r}{2}}\)\(x_{r,\xi+1}\)$ and $\theta'\(f_{\frac{r}{2}}\)\(x_{r,\xi+1}\)$ are linear combinations of vectors of the form
\[
v_\tau
\ot
\pi'\(u\(S_1,\stt{y_{S_1,\beta}}_\beta\)\)v'_{\pi'},
\]
where $u\(S_1,\stt{y_{S_1,\beta}}_\beta\)$ is given by \eqref{E:u(S_1)}. Since $u\(S_1,\stt{y_{S_1,\beta}}_\beta\)\in U_{n-r}(F)$, it follows that
\[
\la\(v_\tau\ot\pi'\(u\(S_1,\stt{y_{S_1,\beta}}_\beta\)\)v'_{\pi'}\)
=
c\,\la_{\tau,\psi}\(v_\tau\)\la_{\pi',\psi}\(\pi'\(\varpi^{-\la_{S_1}}\)v_{\pi'}\)
\]
for some $0\ne c\in\bbC$. At this point, we note that 
\[
\la_{\pi',\psi}\(\pi'\(\varpi^{-\la_{S_1}}\)v_{\pi'}\)\ne 0
\Longleftrightarrow
S_1=\emptyset. 
\]
Indeed, since $v_{\pi'}$ is fixed by $K_{n-r, c_{\pi'}}$ and $\psi$ is trivial on $\frak{o}$ but non-trivial on $\frak{p}^{-1}$, the assertion follows
from a standard argument, as in \cite[Proposition 6.1]{CasselmanShalika1980}. When $S_1=\emptyset$, we have
\[
I_{S_1}^-=\emptyset
\quad\text{and}\quad
I_{S_1}^+
=
\stt{\ep_i+\ep_j\mid r+1\le i<j\le n}. 
\]
Since
\[
|I_{S_1}^+|
=
\tfrac{(n-r)(n-r-1)}{2},
\]
it follows that 
\begin{align*}
\la&\(\theta\(f_{\frac{r}{2}}\)\(x_{r,\xi+1}\)-\chi(\varpi)\theta'\(f_{\frac{r}{2}}\)\(x_{r,\xi+1}\)\)\\
&\quad=
(-1)^r\chi(\varpi)^r\la_{\tau,\psi}(v_\tau)\la_{\pi',\psi}(v_{\pi'})q^{\frac{(n-r)(n-r-1)}{2}}\(q^{nr}-q^{(n-1)r}\)\ne 0. 
\end{align*}
This verifies the case when $|I_{\phi,\chi}|=1$ and $|I_{\phi,\chi'}|=0$. 

Finally, assume that $|I_{\phi,\chi}|=|I_{\phi,\chi'}|=1$. Then $|I_{\phi',\chi}|=0$ and $|I_{\phi',\chi'}|=1$, and hence 
$\pi'^{K_{n-r},c_{\pi'}}\ne 0$, by what we just proved. By applying the same argument as above, we conclude that 
\[
\pi^{K_{n,c_\pi}}\ne 0. 
\]
This completes the proof for the case where $|I_{\phi,\chi}|=1$ or $|I_{\phi,\chi'}|=1$, and hence the proof of \propref{P:sc to seed}.\qed

\subsection{Proof of \thmref{T:mainB}}
This follows from \propref{P:sq to gen}, \propref{P:seed to sq} and \propref{P:sc to seed}.\qed 

\subsection{Concluding remarks}\label{SS:Tsai lemma}
Let $\pi$ be an irreducible generic supercuspidal representation of $\SO_{2n+1}(F)$ with $n\ge 2$, and let $v\in\pi$ be a non-zero element that 
is fixed by $K_{n,m}$ for some $m\ge 0$. Such an element always exists by \cite[Corollary 3.4.2 and Theorem 7.3.1]{Tsai2013}, and we have 
$m\ge c_\pi\ge 2$. In \cite[Lemma 8.3.3]{Tsai2013}, it is asserted that 
\begin{equation}\label{E:Tsai lemma}
W_v\(\varpi^{a\ep_1^*}x_{-\ep_1}(y_1)x_{-\ep_1+\ep_i}(y_2)x_{-\ep_1+\ep_j}(y_3)\varpi^\mu\)
=
W_v\(\varpi^{a\ep_1^*+\mu}\)
\end{equation}
for every $y_1, y_2, y_3\in F$, $\mu\in X_*(T_n)$, and $a\in\bbZ$. Here $W_v$ is the Whittaker function attached to $v$, defined by 
\[
W_v(g)=\la_{\pi,\psi}(\pi(g)v)
\] 
for $g\in\SO_{2n+1}(F)$, where $\la_{\pi,\psi}\in{\rm Hom}_{U_n(F)}(\pi,\psi_{U_n})$ is a non-zero Whittaker functional. 
Unfortunately, this assertion appears to be incorrect. 

In fact, if \eqref{E:Tsai lemma} holds for every $y_1, y_2, y_3\in F$, $\mu\in X_*(T_n)$, and $a\in\bbZ$, then we particularly have
\[
W_v\(\varpi^{a\ep_1^*}x_{-\ep_1}(1)\)
=
W_v\(\varpi^{a\ep_1^*}\)
\]
for every $a\in\bbZ$. Since $\pi\(w_{\ep_1,m}\)$ induces an involution on $\pi^{K_{n,m}}$, we may assume that 
\[
\pi(w_{\ep_1,m})W_v=\ep W_v
\]
for some $\ep\in\stt{\pm 1}$. Now the identity
\[
x_{-\ep_1}(1)
=
x_{\ep_1}(1)\,\varpi^{m\ep^*_1}\,w_{\ep_1,m}\,x_{\ep_1}(1),
\]
and the the fact that 
\[
x_{\ep_1}(1)\in K_{n,m}
\]
imply
\begin{align*}
W_v\(\varpi^{a\ep_1^*}\)
&=
W_v\(\varpi^{a\ep_1^*}x_{-\ep_1}(1)\)\\
&=
W_v\(\varpi^{a\ep_1^*}\,x_{\ep_1}(1)\,\varpi^{m\ep^*_1}\,w_{\ep_1,m}\,x_{\ep_1}(1)\)\\
&=
W_v\(x_{\ep_1}\(\varpi^{a}\)\,\varpi^{(a+m)\ep_1^*}\,w_{\ep_1,m}\)\\
&=
\ep W_v\(\varpi^{(a+m)\ep_1^*}\). 
\end{align*}

Since $W_v$ is right $K_{n,m}$-invariant and $\psi$ is trivial on $\frak{o}$ but non-trivial on $\frak{p}^{-1}$, a standard argument, as in 
\cite[Proposition 6.1]{CasselmanShalika1980}, shows that $W_v\(\varpi^{a\ep_1^*}\)=0$ for every $a<0$. The above identity then implies that 
$W_v\(\varpi^{a\ep_1^*}\)=0$ for every $a\in\bbZ$. In particular, we have 
\[
W_v\(I_{2n+1}\)
=
\la_{\pi,\psi}(v)=0.  
\]
However, this contradicts \cite[Main Theorem 2]{Tsai2013}. 

Tsai used this lemma to compute the Hecke operators on $\pi^{K_{n,m}}$ in \cite[Section 8.3]{Tsai2013}. Consequently, she obtained two 
propositions (\cite[Proposition 8.3.4 and Proposition 8.3.6]{Tsai2013}), which allowed her to deduce two corollaries 
(\cite[Corollary 8.3.7 and Corollary 8.3.8]{Tsai2013}). Using these two corollaries, she proved the main theorems 
(\cite[Theorem 8.4.1 and Theorem 8.4.2]{Tsai2013}) in \cite[Section 8.4]{Tsai2013}. In other words, Proposition 8.3.4 and Proposition 8.3.6
are key ingredients in her proof, which in turn rely on the computation of the Hecke operators. 

That being said, it seems that, in computating the Hecke operators, she did not need the full strength of the lemma; namely, she only 
required \eqref{E:Tsai lemma} to hold for certain $y_1, y_2, y_3\in F$ and $\mu\in X_*(T_n)$. In this way, it may still be possible to address
this issue.

\subsection*{Acknowledgements}
Part of this work was carried out during the Relative Langlands Program Conference held at IMS. The author would like to thank the 
organizers for providing a stimulating and productive environment for discussion.

The author is deeply indebted to Chi-Heng Lo for addressing his questions and for writing the appendix to this paper; is particularly 
indebted to Masao Oi for his constant and insightful suggestions and discussions; is grateful to Hiraku Atobe for kindly and patiently 
answering his questions through numerous email exchanges; and thanks Yiyang Wang for helpful discussions and for sharing his paper.
This work was partially supported by MOST under Grant No. 112-2115-M-032-004-MY3.

\appendix
\section{A lemma on certain parabolic induction}\label{S:appendix}
\begin{center}
by Chi-Heng Lo
\end{center}
\subsection{Notation for derivatives}

For any supercuspidal representation $\rho$ of $\GL_k(F)$ and $x,y\in \bbR$ such that $x-y \in \bbZ_{\geq 0}$, we define the essentially 
square-integrable representation $\Delta_{\rho}[x,y]$ of $\GL_{k (x-y+1)}(F)$ to be the unique irreducible subrepresentation of the 
normalized parabolic induction
\[ 
\rho\lvert \cdot\rvert^{x} \times \rho\lvert \cdot\rvert^{x-1} \times \cdots \times \rho\lvert \cdot\rvert^{y}.
\]
We may uniquely write $\rho= \rho^u\lvert\cdot\rvert^{z}$, where $\rho^u$ is a unitary supercuspidal representation and $z \in \bbR$. 
Then if $\rho^u \not \cong (\rho')^{u}$, then we have
\begin{align}\label{eq GL commutative}
\Delta_{\rho}[x,y] \times \Delta_{\rho'}[z,w] = \Delta_{\rho'}[z,w] \times \Delta_{\rho}[x,y],
\end{align}
which is irreducible.

Let $D_{\rho}$ denote the derivatives of $\SO_{2n+1}(F)$ and let $L_{\rho}$ and $R_{\rho}$ denote the left and right derivatives of $\GL_{n}(F)$. 
These are linear maps between Grothendieck groups
\begin{align*}
    D_{\rho}&: \mathscr{R}( \SO_{2n+1}(F)) \to \mathscr{R}(\SO_{2n-2d+1}(F)),\\ 
    L_{\rho}&: \mathscr{R}(\GL_{n}(F)) \to \mathscr{R}(\GL_{n-d}(F)),\\
    R_{\rho}&: \mathscr{R}(\GL_{n}(F)) \to \mathscr{R}(\GL_{n-d}(F)),
\end{align*}
defined by the following formulas in the Grothendieck group: For $\pi \in \mathscr{R}( \SO_{2n+1}(F))$ and 
$\sigma \in \mathscr{R}(\GL_{n}(F))$, 
\begin{align*}
\Jac_{\GL_{d} \times \SO_{2n-2d+1}(F)}(\pi) &= \rho \boxtimes D_{\rho}(\pi) + \sum_{\tau } \tau \boxtimes \pi_{\tau},\\
\Jac_{\GL_d \times \GL_{n-d}}(\sigma)&= \rho \boxtimes L_{\rho}(\sigma)+ \sum_{\tau } \tau \boxtimes \sigma_{\tau},\\
\Jac_{\GL_{n-d}\times \GL_{d}}(\sigma)&= R_{\rho}(\sigma) \boxtimes \rho+ \sum_{\tau }  \sigma_{\tau} \boxtimes \tau,
\end{align*}
where the summation is taken over irreducible representations $\tau$ of $\GL_{d}(F)$ that are not isomorphic to $\rho$. By Tadi{\'c}'s formula for Jacquet modules (see \cite[Proposition 1.7]{Zelevinsky1980} \cite[Theorem 6.5]{Tadic1995}), we have Leibniz rule for derivatives:
\begin{enumerate}
\item $L_{\rho}(\tau_1 \times \tau_2)= L_{\rho}(\tau_1) \times \tau_2 + \tau_1 \times L_{\rho}(\tau_2)$.
    \item $D_{\rho}( \tau \rtimes \sigma )= L_{\rho}(\tau) \rtimes  \sigma + R_{\rho^{\vee}}(\tau) \rtimes \sigma+ \tau \rtimes D_{\rho}(\sigma).$
\end{enumerate}
Also, the left and right derivatives of essentially discrete series representations of $\GL_n(F)$ are known 
(\cite[Proposition 9.5]{Zelevinsky1980}): Let $\rho$ and $\rho'$ be unitary supercuspidal representations of general linear groups. Then
\begin{align}\label{eq der GL 1}
    L_{\rho|\cdot|^{z}}(\Delta_{\rho'}[x,y])= &\begin{cases}  
        \Delta_{\rho'}[x-1,y] & \text{ if } \rho' \cong \rho\text{ and } z=x \\
        0 & \text{ otherwise.}
        \end{cases}\\
         \label{eq der GL 2}
         R_{\rho|\cdot|^{z}}(\Delta_{\rho'}[x,y])= &\begin{cases}  
            \Delta_{\rho'}[x,y+1] & \text{ if } \rho' \cong \rho\text{ and } z=y \\
            0 & \text{ otherwise.}
            \end{cases}
\end{align}
In the Grothendieck group $\mathscr{R}(\SO_{2n+1}(F))$, we write $\Pi_1 \leq \Pi_2$ if $\Pi_2-\Pi_1$ is a positive linear combination of irreducible representations.

\subsection{Statement of the lemma}

Let $\phi$ be a tempered $L$-parameter of $\SO_{2n+1}(F)$. Let $\rho$ be an orthogonal supercuspidal representation of $\GL_k(F)$ and let $\phi_{\rho}$ denote its $L$-parameter, which is a selfdual $k$-dimensional irreducible representation of $W_F$. Assume that $\phi$ contains $ \phi_{\rho} \boxtimes S_{2 \kappa+1}$ for some $\kappa \in \frac{1}{2}+\bbZ_{\geq 0}$, and the $L$-parameter of $\SO_{2n+1- k(2\kappa +1)}(F)$
\[ \phi_0:= \phi - \phi_\rho \boxtimes S_{2\kappa +1}   \ \ \ \ \ \text{(possibly empty)}\]
does not contain any summand of the form $\phi_{\rho} \boxtimes S_{d}$ for any $d \in \bbZ_{> 0}$. Let $\pi_0$ denote the unique generic representation in $\Pi_{\phi_0}$ and let $\pi$ denote the unique generic representation in $\Pi_{\phi}$. Finally, set 
$\Pi:= \Delta_{\rho}[\kappa, \frac{1}{2}] \rtimes \pi_0$.

\begin{lm}
    In the above setting, $\Pi$ has length two and we have an exact sequence
    \[  0 \to \pi \to \Pi \to \sigma \to 0,\]
    where $\sigma$ is the Langlands quotient of $\Pi$.
\end{lm}

Clearly, the Lemma follows from the following proposition:
\begin{prop}\label{prop: length 2}
    In the above setting, the following hold:
    \begin{enumerate}
        \item [(a)] If $\tau$ is an irreducible subquotient of $\Pi$ that is not isomorphic to $\sigma$, then $\tau$ is tempered.
       
        \item [(b)] The representation $\pi$ has multiplicity one in the semisimplification of $\Pi$.
        
        \item [(c)] If $\tau$ is a tempered irreducible subquotient of $\Pi$, then $\tau$ is isomorphic to $\pi$. 
    \end{enumerate}
\end{prop}

\subsection{\texorpdfstring{Proof of Proposition \ref{prop: length 2}(a)}{}}

Suppose that $\tau$ is a non-tempered irreducible subquotient of $\Pi$. We write the subrepresentation version of the standard module of $\tau$ as
\[ M_{sub}(\tau)=\Delta_{\rho_1}[x_1,y_1] \times \cdots \times \Delta_{\rho_r}[x_r,y_r] \rtimes \tau_{temp}.\]
Here, $r\geq 1$ and
\begin{itemize}
    \item $\rho_i$ are unitary supercuspidal representations of $\GL_{k_i}(F)$ $x_i,y_i \in \bbR$ such that $x_i-y_i \in \bbZ_{\geq 0}$.
    \item $x_1+y_1\leq  \cdots  \leq x_r+y_r <0$.
    \item $\tau_{temp}$ is a tempered representation.
    \item $\tau$ is the unique irreducible subrepresentation of $M_{sub}(\tau)$
\end{itemize}
Given any irreducible representation $\tau$ of $\SO_{2n+1}(F)$, there exists a unique $M_{sub}(\tau)$ satisfying the above conditions.

First, we claim that $\rho_i \cong \rho$ for all $1 \leq i \leq r$. Indeed, suppose the contrary and let $s:= \min\{1 \leq i \leq r \ | \ \rho_i \not\cong \rho\}$. Then by \eqref{eq GL commutative}, we have 
\begin{align*}
    \tau & \hookrightarrow \Delta_{\rho_1}[x_1,y_1] \times \cdots \times \Delta_{\rho_r}[x_r,y_r] \rtimes \tau_{temp}\\
    &= \Delta_{\rho_s}[x_s,y_s] \times \bigtimes_{i \neq s} \Delta_{\rho_i}[x_i,y_i] \rtimes \tau_{temp}\\
    & \hookrightarrow (\rho_s \lvert\cdot\rvert^{x_s}  \times  \rho_s \lvert\cdot\rvert^{x_s-1} \times \cdots \times \rho_s \lvert\cdot\rvert^{y_s}) \times     \bigtimes_{i \neq s} \Delta_{\rho_i}[x_i,y_i] \rtimes \tau_{temp}.
\end{align*}
Set $D:= D_{\rho_s\lvert\cdot\rvert^{y_s}} \circ D_{\rho_s\lvert\cdot\rvert^{y_s+1}} \circ \cdots \circ D_{\rho_s\lvert\cdot\rvert^{x_s}}.$ Then Frobenius reciprocity implies that $D(\tau) \neq 0$. On the other hand, by the Leibniz rule, \eqref{eq der GL 1}, \eqref{eq der GL 2} and the assumption that $\rho_s \not \cong \rho$, we have
\[ 0 \neq D(\tau) \leq D (\Pi)= \Delta_{\rho}[\kappa, 1/2] \rtimes D(\pi_0).\]
However, since $x_s+ x_{s-1}+\cdots +y_s<0$, Casselman's criterion implies that $D(\pi_0) = 0$. This gives a contradiction and verifies the claim.

By the same argument as in the previous paragraph, we have $D_{\rho\lvert\cdot\rvert^{x_1}}(\tau) \neq 0$. Hence,
\[ 
0 \neq D_{\rho\lvert\cdot\rvert^{x_1}}(\tau) \leq D_{\rho\lvert\cdot\rvert^{x_1}}(\Pi)
= 
(L_{\rho\lvert\cdot\rvert^{x_1}}
+ R_{\rho\lvert\cdot\rvert^{-x_1}} )(\Delta_{\rho}[\kappa, 1/2]) \rtimes \pi_0
+ \Delta_{\rho}[\kappa, 1/2] \rtimes D_{\rho\lvert\cdot\rvert^{x_1}}(\pi_0).
\]
On the other hand, by the highest-derivatives formula for tempered representations in \cite[Theorem C.4.3]{AGIKMS24}, our assumption that $\phi_0$ does not contain any summand of the form $\phi_{\rho} \otimes S_{d}$ for any $d \in \bbZ_{> 0}$ implies that
\begin{align}\label{eq der pi 0}
    D_{\rho\lvert\cdot\rvert^{z}}(\pi_0) =0 \ \ \ \ \text{for all } z \in \bbR 
\end{align}
 Thus, \eqref{eq der GL 1} and \eqref{eq der GL 2} imply that $x_1= -\half{1}$ or $\kappa$. Moreover, set $[x,y]_{\rho}:= \{\rho \lvert\cdot\rvert^{x} , \rho \lvert\cdot\rvert^{x-1} , \ldots , \rho \lvert\cdot\rvert^{y}\}$. Comparing the infinitesimal parameters of $\tau$ and $\sigma$, we obtain a containment of multi-sets
\[ \sum_{i=1}^r [x_i,y_i]_{\rho} + [-y_i,-x_i]_{\rho} \subseteq [\kappa, 1/2]_{\rho}+ [- 1/2,-\kappa]_{\rho}= [\kappa,-\kappa]_{\rho}.\] 
Then since $x_1+y_1\leq x_2+y_2 \leq \cdots \leq x_r+y_r <0$, we must have $x_1= -\half{1}$, $y_1= -\kappa$ and $r=1$. This proves that $\tau$ must be the Langlands quotient of $\Pi$ and completes the proof of Part (a) of Proposition \ref{prop: length 2}.

\subsection{\texorpdfstring{Proof of Proposition \ref{prop: length 2}(b)}{}}
Let $D:= D_{\rho\lvert\cdot\rvert^{1/2}} \circ D_{\rho\lvert\cdot\rvert^{3/2}} \circ \cdots \circ D_{\rho\lvert\cdot\rvert^{\kappa}}$. 
By the Leibniz rule, we see that $D(\Pi)= \pi_0$, which is irreducible. Moreover,  by \eqref{eq der pi 0}, it is a composition of highest derivatives. As a consequence, there must exist an irreducible subquotient $\tau$ of $\Pi$ of multiplicity one such that $D(\tau) = \pi_0$. Such an irreducible representation is unique by \cite[Proposition 3.3]{AM23}. 

On the other hand, since $\pi_0$ and $\pi$ are generic, they correspond to the trivial character of the component group of $\phi_0$ and $\phi$, respectively. Then, applying the highest-derivatives formula for tempered representations in \cite[Theorem C.4.3]{AGIKMS24} repeatedly, 
we see that $D(\pi)= \pi_0$. This completes the proof of Part (b) of Proposition \ref{prop: length 2}.

\subsection{\texorpdfstring{Proof of Proposition \ref{prop: length 2}(c)}{}}
Suppose that $\tau$ is a tempered irreducible subquotient of $\Pi$. By comparing the extended cuspidal support (see \cite[Proposition 2.3]{Ato23}) or infinitesimal parameters, we see that $\tau$ is in the $L$-packet of $\phi$. Let $\eta$ denote the character of the component group of $\phi$ corresponding to $\tau$, which we identify with a function on irreducible summands of $\phi$ with value $\pm 1$ (see the notation in \cite[\S 1.5]{AGIKMS24}). 

Suppose that $\eta(\phi_{\rho} \boxtimes S_{2\kappa+1}) = 1$. Set $D:=D_{\rho\lvert\cdot\rvert^{1/2}} \circ D_{\rho\lvert\cdot\rvert^{3/2}} \circ \cdots \circ D_{\rho\lvert\cdot\rvert^{\kappa}}$. Then the formula for highest derivatives of tempered representations in \cite[Theorem C.4.3]{AGIKMS24} implies that $D(\tau)\neq 0$. Then $D(\tau) \leq D(\Pi)=\pi_0 $, which implies that $\tau \cong \pi$ as in the proof of Part (b).

Suppose that $\eta(\phi_{\rho} \boxtimes S_{2\kappa+1}) = -1$. 
Set $D':= D_{\rho\lvert\cdot\rvert^{3/2}} \circ \cdots \circ D_{\rho\lvert\cdot\rvert^{\kappa}}$. 
Then again by \cite[Theorem C.4.3]{AGIKMS24}, $D'(\tau)$ is an irreducible tempered representation with 
$L$-parameter $\phi_0+\phi_{\rho} \boxtimes S_{2}$, and $D_{\rho\lvert\cdot\rvert^{z}}(D'(\tau))=0$ for any $z \in \bbR$. 
As a consequence, we have  
\[ 
\tau \hookrightarrow  (\rho\lvert\cdot\rvert^{\kappa}  \times \rho\lvert\cdot\rvert^{\kappa-1} \times 
\cdots \times  \rho\lvert\cdot\rvert^{3/2}) \times  \rho_1\lvert\cdot\rvert^{x_1} \times \cdots \times \rho_r\lvert\cdot\rvert^{x_r} \rtimes \tau_{sc}, 
\]
where 
\begin{itemize}
    \item $\rho_i$'s are unitary supercuspidal representations not isomorphic to $\rho$ and $x_i \in \bbR$.
    \item $\tau_{sc}$ is a supercuspidal representation of some $\SO_{2m+1}(F)$ whose $L$-parameter contains $\phi_{\rho} \boxtimes S_{2}$.
\end{itemize}
Then it is clear that $\tau$ has different supercuspidal support from $\sigma$, which contradicts the assumption that $\tau$ and $\sigma$ are irreducible subquotients of the same induction $\Pi$. This completes the proof of Part (c) of Proposition \ref{prop: length 2}.\qed

\end{document}